\documentclass[letterpaper, 12pt, oneside]{book}

\usepackage[margin = 1in, includehead, footskip=0.25in]{geometry}

\usepackage{setspace}
\usepackage{caption}
\usepackage{subcaption}

\usepackage{url}
\usepackage{amsmath}
\usepackage{booktabs}
\usepackage{graphicx}
\usepackage{import} 
\usepackage{amssymb}
\usepackage{geometry}             
\usepackage{graphicx}
\usepackage{amssymb,amsbsy}
\usepackage{epstopdf}
\usepackage{float}
\usepackage{amsfonts}
\usepackage{amsmath}
\usepackage{amscd}
\usepackage{tikz-cd}
\usepackage{latexsym}
\usepackage{amsthm}
\usepackage{verbatim}
\usepackage{paralist}
\usepackage{bbm}
\usepackage{cite}
\usepackage[utf8]{inputenc}
\usepackage{etoolbox}
\usepackage{enumitem}
\usepackage{tikz} \usetikzlibrary{patterns}
\usetikzlibrary{matrix}
\usepackage[nottoc,notlot,notlof]{tocbibind}

\newtheoremstyle{thmstyle}
  {10 pt} 
  {\topsep} 
  {\itshape} 
  {} 
  {\bfseries} 
  {.} 
  {.5em} 
  {} 
  
\newtheoremstyle{defstyle}
  {10 pt} 
  {10 pt} 
  {} 
  {} 
  {\bfseries} 
  {.} 
  {.5em} 
  {} 
  
\newtheoremstyle{factstyle}
  {10 pt} 
  {\topsep} 
  {} 
  {} 
  {\itshape \bfseries} 
  {.} 
  {.5em} 
  {} 
  
 \newtheoremstyle{remark}
  {15 pt} 
  {\topsep} 
  {\slshape} 
  {} 
  {\itshape} 
  {.} 
  {.5em} 
  {} 

\theoremstyle{thmstyle} \newtheorem{thm}{Theorem}[section]
\theoremstyle{thmstyle} \newtheorem{lem}[thm]{Lemma}
\theoremstyle{thmstyle} \newtheorem{cor}[thm]{Corollary}
\theoremstyle{thmstyle} \newtheorem{prop}[thm]{Proposition}
\theoremstyle{defstyle} \newtheorem{defnx}[thm]{Definition} \newenvironment{defn}
  {\pushQED{\qed}\defnx}
  {\popQED\enddefnx}
\theoremstyle{remark} \newtheorem*{claim}{Claim} \newtheorem{claimno}{Claim}
\theoremstyle{factstyle} \newtheorem{fact}[thm]{Fact} 
\theoremstyle{factstyle} \newtheorem{question}[thm]{Question}
\newtheorem{remark}[thm]{Remark}
\AtBeginEnvironment{thm}{\setcounter{claimno}{0}}

\setlist[description]{leftmargin=5.5em,labelindent=\parindent}
\renewcommand{\P}{\mathbb{P}}
\newcommand{\Q}{\mathbb{Q}}
\newcommand{\M}{\mathbb{M}}
\newcommand{\Namba}{\mathbb{N}}
\newcommand{\D}{\mathbb{D}}
\newcommand{\N}{{\overline{N}}}
\renewcommand{\H}{\overline{H}}
\newcommand{\G}{\overline{G}}
\renewcommand{\S}{{\overline{S}}}
\newcommand{\R}{\mathbb{R}}

\newcommand{\B}{\mathbb{B}}
\newcommand{\U}{\mathcal{U}}
\renewcommand{\c}{\mathfrak{c}}

\newcommand{\PFA}{\textup{\ensuremath{\textsf{PFA}}}}
\newcommand{\MA}{\textup{\ensuremath{\textsf{MA}}}}
\newcommand{\MM}{\textup{\ensuremath{\textsf{MM}}}}
\newcommand{\SPFA}{\textup{\ensuremath{\textsf{SPFA}}}}
\newcommand{\ZFC}{\textup{\ensuremath{\textsf{ZFC}}}}
\newcommand{\SCFA}{\textup{\ensuremath{\textsf{SCFA}}}}
\newcommand{\BSCFA}{\textup{\ensuremath{\textsf{BSCFA}}}}
\newcommand{\CH}{\textup{\textsf{CH}}}

\newcommand{\Ord}{\textup{\ensuremath{\text{Ord}}}}
\newcommand{\id}{\textup{\ensuremath{\text{id}}}}
\newcommand{\MP}{\textup{\ensuremath{\textsf{MP}}}}

\newcommand{\BFA}{\textup{\ensuremath{\textsf{BFA}}}}

\DeclareMathOperator{\cof}{cof}
\DeclareMathOperator{\height}{height}
\DeclareMathOperator{\ran}{range}
\DeclareMathOperator{\otp}{otp}
\DeclareMathOperator{\cp}{cp}
\DeclareMathOperator{\dom}{dom}
\DeclareMathOperator{\Succ}{succ}
\DeclareMathOperator{\wfc}{wfc}
\DeclareMathOperator{\BA}{BA}
\DeclareMathOperator{\Add}{\mathcal A\textit{dd}\,}
\DeclareMathOperator{\Coll}{\mathcal C\textit{oll}\,}
\DeclareMathOperator*{\bigdoublevee}{\bigvee\mkern-15mu\bigvee}

\newcommand{\To}{\longrightarrow}
\newcommand{\st}{\; | \;}
\newcommand{\set}[2]{\left\{#1\st #2 \right\}}
\newcommand{\seq}[2]{\langle #1 \st #2 \rangle}

\newcommand{\Ptail}{\P_{\mathrm{tail}}}
\newcommand{\Gtail}{G_{\mathrm{tail}}}

\newcommand{\forces}{\Vdash}
\newcommand{\proves}{\vdash}
\newcommand{\rest}{\mathbin{\upharpoonright}}
\newcommand{\meet}{\wedge}
\newcommand{\Meet}{\bigwedge}

\DeclareMathOperator{\MPsc}{\textup{\textsf{MP}}_{\textit{sc}}}
\DeclareMathOperator{\MPc}{\textup{\textsf{MP}}_{<\omega_1\textup{-closed}}}
\newcommand{\lMPsc}{ \textup{\textsf{MP}}^{H_{\omega_2}}_{\textit{sc}}(H_{\omega_2}) }
\newcommand{\lflMPsc}{ \textup{\textsf{MP}}^{H_{\omega_2}}_{\textit{sc}}(\emptyset) }
\newcommand{\lMPc}{ \textup{\textsf{MP}}^{H_{\omega_2}}_{<\omega_1\textup{-closed}}(\emptyset) }
\newcommand{\bflMPc}{ \textup{\textsf{MP}}^{H_{\omega_2}}_{<\omega_1\textup{-closed}}(H_{\omega_2}) }
\DeclareMathOperator{\RAsc}{\textup{\textsf{RA}}_{\textit{sc}}}
\DeclareMathOperator{\RAc}{\textup{\textsf{RA}}_{<\omega_1\textup{-closed}}}
\DeclareMathOperator{\bfRAsc}{\textup{\textsf{\textbf{RA}}}_{\textit{sc}}}

\newcommand{\SH}{\mathcal{S}\textit{k} \,}
\newcommand{\sk}[3]{\SH^{#1}( {#2} \cup {\ran(#3)} ) }
\newcommand{\Sk}[3]{\SH^{#1}( {#2} \cup {#3} ) }

\newcommand{\TC}[1]{\mathrm{TC}(\{ #1 \})}

\usepackage[pdfauthor={Kaethe Lynn Bruesselbach Minden},
    pdftitle={On Subcomplete Forcing},
    hidelinks
]{hyperref}

\begin{document}

\frontmatter

\begin{titlepage}

\begin{center}

~\vspace{2in}

\textsc{On Subcomplete Forcing} \\[0.5in]
by \\[0.5in]
\textsc{Kaethe Lynn Bruesselbach Minden} 

\vspace{\fill}
A dissertation submitted to the Graduate Faculty in Mathematics in partial fulfillment of the requirements for the degree of Doctor of Philosophy, The City University of New York \\[0.25in]
2017

\end{center}

\end{titlepage}

\setcounter{page}{2}

\phantom{}\vspace{\fill}
\begin{center}
\copyright~2017\\
\textsc{Kaethe Lynn Bruesselbach Minden}\\
All Rights Reserved\\
\end{center}
\begin{center}
This manuscript has been read and accepted by the Graduate Faculty in Mathematics in satisfaction of the dissertation requirement for the degree of Doctor of Philosophy.
\end{center}

\vspace{0.75in}

\begin{tabular}{p{1.75in}p{0.5in}p{3.5in}}
~                                   & & \textbf{Professor Gunter Fuchs}\\
~                                   & & \\
\hrulefill                          & &\hrulefill \\
Date                                & & Chair of Examining Committee\\
~                                   & & \\
~                                   & & \textbf{Professor Ara Basmajian}\\
~                                   & & \\
\hrulefill                          & &\hrulefill \\
Date                                & & Executive Officer\\
\end{tabular}

\vspace{0.75in}

\begin{tabular}{l}
\textbf{Professor Gunter Fuchs} \\
\textbf{Professor Joel David Hamkins} \\
\textbf{Professor Arthur Apter} \\
Supervisory Committee \\
\end{tabular}

\vspace{\fill}
\begin{center}
\textsc{The City University of New York}
\end{center}
\renewcommand{\c}{\mathfrak{c}}
\begin{center}
Abstract \\
\textsc{On Subcomplete Forcing} \\
by \\
\textsc{Kaethe Lynn Bruesselbach Minden} \\[0.25in]
\end{center}

\vspace{0.25in}

\noindent Adviser: Professor Gunter Fuchs

\vspace{0.25in}

\noindent I survey an array of topics in set theory and their interaction with, or in the context of, a novel class of forcing notions: \textit{subcomplete forcing}. Subcomplete forcing notions satisfy some desirable qualities; for example they don't add any new reals to the model, and they admit an iteration theorem. While it is straightforward to show that any forcing notion that is countably closed is also subcomplete, it turns out that other well-known, more subtle forcing notions like Prikry forcing and Namba forcing are also subcomplete. 
Subcompleteness was originally defined by Ronald Bj\"orn Jensen around 2009. Jensen's writings make up the vast majority of the literature on the subject. Indeed, the definition in and of itself is daunting. I have attempted to make the subject more approachable to set theorists, while showing various properties of subcomplete forcing that one might desire of a forcing class.

It is well-known that countably closed forcings cannot add branches through $\omega_1$-trees. I look at the interaction between subcomplete forcing and $\omega_1$-trees. It turns out that subcomplete forcing also does not add cofinal branches to $\omega_1$-trees.
I show that a myriad of other properties of trees of height $\omega_1$ as explored in \cite{Hamkins:qf} are preserved by subcomplete forcing; for example, I show that the unique branch property of Suslin trees is preserved by subcomplete forcing.

Another topic I explored is the \textit{Maximality Principle} (\textsf{MP}).
Following in the footsteps of Hamkins \cite{Hamkins:2003jk}, Leibman \cite{Leibman:MP}, and Fuchs \cite{Fuchs:2008rt},\cite{Fuchs:2008ve}, I examine the subcomplete maximality principle. 
In order to elucidate the ways in which subcomplete forcing generalizes the notion of countably closed forcing, I compare the countably closed maximality principle ($\MPc$) to the subcomplete maximality principle ($\MPsc$). Again, since countably closed forcing is subcomplete, this is a natural question to ask. 
I was able to show that many of the results about $\MPc$ also hold for $\MPsc$; for example, the boldface appropriate notion of $\MPsc$ is equiconsistent with a fully reflecting cardinal. 
However, it is not the case that the subcomplete and countably closed maximality principles directly imply one another. 

I also explore the \textit{Resurrection Axiom} (\textsf{RA}). 
Hamkins and Johnstone \cite{Hamkins:2013qv} defined the resurrection axiom only relative to $H_\c$, and focus mainly on the resurrection axiom for proper forcing. They also show the equiconsistency of various resurrection axioms with an uplifting cardinal. I argue that the subcomplete resurrection axiom should naturally be considered relative to $H_{\omega_2}$, and showed that the subcomplete resurrection axiom is equiconsistent with an uplifting cardinal. 

A question reasonable to ask about any class of forcings is whether or not the resurrection axiom and the maximality principle can consistently both hold for that class. I originally had this question about the full principles, not restricted to any class, but in my thesis it was appropriate to look at the question for subcomplete forcing. I answer the question positively for subcomplete forcing using a strongly uplifting fully reflecting cardinal, which is a combination of the large cardinals needed to force the principles separately. I show that the boldface versions of $\MPsc + \RAsc$ both holding is equiconsistent with the existence of a strongly uplifting fully reflecting cardinal.

While Jensen \cite{Jensen:2012fr} shows that Prikry forcing is subcomplete, I long suspected that many variants of Prikry forcing that have a kind of genericity criterion are also subcomplete. After much work I managed to show that a variant of Prikry forcing known as \textit{Diagonal Prikry Forcing} is subcomplete, giving another example of subcomplete forcing to add to the list.

\chapter*{Acknowledgments}

Thank you Gunter, for meeting with me over the years; kindly and patiently explaining everything from the beginning, always in detail, and for dedicating all of the time and effort into editing my writing. Thank you Joel and Arthur for agreeing to serve on my committee, and for your helpful advice and suggestions. 

Thank you Miha for letting me defend first, for your friendship, and for our wonderfully long, sometimes productive, conversations about set theory. Thank you Kameryn for actually reading the proof of Theorem 5.0.3 and for pointing out a typo.

I would also like to thank my mom, my dad, my sister Janet, my cat Po Charles, Keith, Aradhana, and other friends and family for their support.

\tableofcontents

\mainmatter

\chapter*{Introduction}
\addcontentsline{toc}{chapter}{Introduction}
Subcomplete forcing is a class of forcing notions defined by Ronald Bj\"orn Jensen \cite{Jensen:2012fr}. How subcompleteness fits in to the picture of commonly used classes of forcing notions is illustrated in the diagram below (where $\sigma$-cl. stands for $\sigma$-closed, or countably closed).

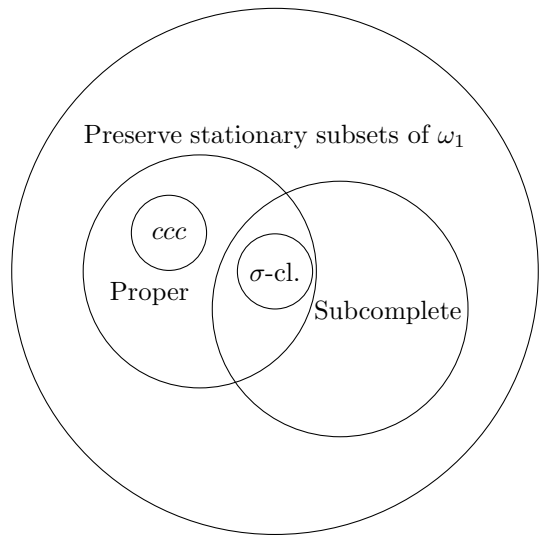
\begin{figure}[h!]\begin{center}
    \begin{tikzpicture}
      \draw (0:0cm) circle (3.5cm);
      \node at (90:1.8cm) {\footnotesize Preserve stationary subsets of $\omega_1$};
      \draw (90:0cm) circle (.5cm) node [text=black] {\footnotesize $\sigma$-cl.};
      \draw (180:1cm) circle (1.55cm) node [text=black,below left] {\footnotesize Proper};
       \draw (160:1.5cm) circle (.5cm) node [text=black] {\footnotesize $ccc$};
      \draw (330:1cm) circle (1.7cm);
      \node at (340:1.6cm){\footnotesize Subcomplete};
    \end{tikzpicture}\end{center}
\end{figure}

While all countably closed forcing is subcomplete, there are subcomplete forcing notions that are not proper - for example, Namba forcing (under $\CH$) and Prikry forcing. The aim here is to go deeper into the notion of subcompleteness. 

In Chapter \ref{chap:preliminaries} the terminology and background are given. In many ways this chapter serves as a kind of review or consolidation of parts of Jensen's notes, which I attempt to follow closely. I do give some results that may not be explicitly proved, but perhaps stated, in Jensen's notes; I attempt to only fill in the important gaps I will be relying upon in later sections. In Section \ref{sec:delta} the quantity $\delta(\P)$ is defined for posets $\P$, while relating the notion to chain conditions and showing how the quantity acts under dense homomorphisms. In Section \ref{sec:fullness} the stage is set for the sorts of models and elementary embeddings worked with in the definition of subcompleteness. In particular, I define a strengthening of transitivity for countable models that Jensen refers to as \textit{fullness}. I give some justification for working with these models, and give basic facts that will be refered to when working with the definition of subcompleteness. In Section \ref{sec:BarwiseTheory} the relevant background in Barwise Theory is given, which is primarily used in this context to show a certain forcing notion is subcomplete. In Section \ref{sec:Liftups} the liftup, an ultrapower-like construction, is defined. 

Finally subcompleteness is defined in Chapter \ref{chap:SCandrelatives}, but only in Section \ref{sec:SubcompleteForcing} after first introducing the weaker notion of subproperness in Section \ref{sec:SubproperForcing}. I show some of the defining properties of subcompleteness after defining the notion: I give Jensen's proof that subcomplete forcing doesn't add reals, that subcompleteness of $\P$ is absolute above the verification to subcompleness, that subcomplete forcing is subproper.
In \ref{subsec:SCForcingandCCForcing} it is shown that countably closed forcing is subcomplete by showing that the class of countably closed forcing notions is the same as that of complete forcing. In \ref{subsec:LevelsofSubcompleteness} the notion of subcompleteness above $\mu$ due to Jensen is defined and I show that if a forcing notion $\P$ is subcomplete above $\mu$ then it does not add new countable subsets of $\mu$.

The next two chapters spend time comparing countably closed forcing to subcomplete forcing. Chapter \ref{chap:PropertiesofSCForcing} focuses on subcomplete forcing's behavior with respect to $\omega_1$-trees (and some slightly larger trees) in Section \ref{sec:SCForcingAndTrees} and the iteration theory for subcomplete forcing in Section \ref{sec:IteratingSCForcing}. Furthermore in the former section we also give Jensen's argument that Suslin trees are not subcomplete, and we show that (nontrivial) $ccc$ forcing notions are not subcomplete.
Chapter \ref{chap:AxiomsaboutSC} looks at axioms about subcomplete forcing; focusing on the subcomplete maximality principle in Section \ref{sec:MP} and the subcomplete resurrection axiom in Section \ref{sec:RA}. The subcomplete maximality principle and the countably closed maximality principle are compared in \ref{subsec:SeparatingSCandCCMP}. More on the modal status of subcompleteness is explored in \ref{subsec:ModalLogicSC}. The local form of the maximality principle is also introduced in \ref{subsec:localMP} and its consistency strength analyzed in \ref{subsec:ConlocalMP}. In Section \ref{sec:RA+MP} it is shown that the maximality principle and the resurrection axiom may consistently be combined. 

Finally, Chapter \ref{chap:GenDiagonalPrikryForcing} is devoted to the proof that generalized diagonal Prikry forcing is subcomplete.

\chapter{Preliminaries}
\label{chap:preliminaries}
Before defining subcomplete forcing and showing consequences of it, some preliminary information is necessary. This can all be found in Jensen's lecture notes from the 2012 AII Summer School in Singapore. For the published version refer to \cite{Jensen:2012fr}. I will also refer to unpublished handwritten notes from Jensen's website.

First, a brief outline of some of the notation used in what follows:
\begin{itemize}
	\item Forcing notions $\P = \langle \P, \leq \rangle$ are taken to be partial orders that are separative and contain a ``top" element weaker than all elements of $\P$, denoted $\mathbbm 1$.
	
	\item We will be working with transitive models of $\ZFC^-$, the axioms of Zermelo-Fraenkel Set Theory without the axiom of \textsf{Powerset}, and with the axiom of \textsf{Collection} instead of \textsf{Replacement}. Usually we will name these models $N$, $M$, or $\N$. 

	\item I will follow Jensen to use the notation $\sigma: N \prec M$ for when $\sigma$ is an elementary embedding, but I will add some notation not used by Jensen, letting $N \preccurlyeq M$ denote and emphasize that $N$ is an elementary substructure of $M$. 

	\item Let $N$ be a transitive $\ZFC^-$ model. I write $\height(N)$ to mean $\Ord \cap N$. Let $\alpha$ be an ordinal. Then I write $\alpha^N$ for $\alpha \cap N$.
	
	\item Let $A$ be a set of ordinals. Then $\lim(A)$ is the set of limit points in $A$.  

	\item Let $\theta$ be a cardinal. $H_\theta$ refers to the collection of sets hereditarily of size less than $\theta$. Relativizing the concept to a particular model of set theory, $M$, I write $H_\theta^M$ to mean the collection of sets in $M$ that are hereditarily of size less than $\theta$ in $M$. In this case, if $\theta$ is determined by some computation, I mean for that computation to take place in $M$.
	
	\item Let $\tau$ be a cardinal. With abuse of notation we write $L_\tau[A]$ to refer to the structure $\langle L_\tau[A]; \in, A \cap L_\tau[A] \rangle$. 
	
	\item If $\Gamma$ is a class of forcings that preserve stationary subsets of $\omega_1$, then the $\Gamma$-forcing axiom ($\textsf{FA}_\Gamma$) posits that for all $\P \in \Gamma$, if $\mathcal D$ is a collection of $\omega_1$-many dense subsets of $\P$, then there is a $\mathcal D$-generic filter (a filter that intersects every element of $\mathcal D$ nontrivially). We write $\MA$ for Martin's Axiom, $\PFA$ for the Proper Forcing Axiom, and $\MM$ for Martin's Maximum, the forcing axiom for classes of forcing notions that preserve stationary subsets of $\omega_1$.
	
	\item We often make use of the following abbreviation: if a map $\sigma$ satisfies $\sigma(\overline a)=a$ and $\sigma(\overline b)=b$, we write $\sigma(\overline a,\overline b)=a,b$.
\end{itemize}	

\section{The Weight of a Forcing Notion}
\label{sec:delta}

\begin{defn} For a forcing notion $\P$, we write $\delta(\P)$ to denote the least cardinality of a dense subset in $\P$. This is sometimes referred to as the \textit{\textbf{weight}} of a poset. \end{defn}

Although Jensen defines $\delta(\P)$ for Boolean algebras, it's also relevant for posets. As Jensen states, for forcing notions $\P$, the weight can be replaced with the cardinality of $\P$, or even $\P$, for the purpose of defining subcompleteness. However, $\delta(\P)$ and $|\P|$ are not necessarily the same, since there could be a large set of points in the poset that all have a common strengthening. 

More can be said about the weight of forcing notions, and I give some basic results below.

\begin{lem} Let $\P$ be a poset. Then maximal antichains in $\P$ have size at most $\delta(\P)$. \end{lem}
\begin{proof}
Supposing $\P$ has a maximal antichain $A$ of size $\kappa$, any dense set $D$ (including one of smallest size) in $\P$ needs to have the property that for each element $a \in A$, there is an element $d \in D$ below $a$, satisfying $d \leq a$. Let's assume that $D$ has size $\delta(\P)$; then $D$ must have size at least $\kappa$ since $A$ is a maximal antichain; indeed we have a function $f:A \to D$, where $f(a)$ is an element of $D$ below $a$. It must be that $f$ is an injection, since otherwise there would be an element $d \in D$ below two distinct elements of $A$, contradicting the fact that $A$ is an antichain. Thus $\kappa =|A| \leq |D| = \delta(\P)$ as desired.
\end{proof}

Below we define a homomorphism between partial orders, following Schindler \cite[Definition 6.47]{schindler2014set}.

\begin{defn}  Let $\langle \P, \leq_\P \rangle$ and $\langle \Q, \leq_\Q \rangle$ be posets. A map $\pi: \P \to \Q$ is a \textbf{\emph{homomorphism}} so long as it preserves order and incompatibility: \begin{enumerate}
	\item For all $p,q \in \P$ we have that $p \leq_\P q \implies \pi(p) \leq_\Q \pi(q)$.
	\item For all $p,q \in \P$ we have that $p \perp q \implies \pi(p) \perp \pi(q)$.
\end{enumerate}

A homomorphism $\pi:\P \to \Q$ is said to be \textbf{\emph{dense}} so long as for every $q \in \Q$ there is some $p \in \P$ such that $\pi(p) \leq q$, i.e., $\ran(\pi)$ is dense in $\Q$. \end{defn}

We give some immediate remarks on our above definition. First of all, we claim property \textbf{1} automatically entails that for all $p,q \in \P$, if $\pi(p) \perp \pi(q)$ then $p \perp q$, since clearly if $p$ is compatible with $q$ we have that $\pi(p)$ is compatible with $\pi(q)$, which is the contrapositive of the claim. Furthermore, if $\P$ is separative (which indeed we assume of all of our forcing notions) then homomorphisms are automatically \textit{strong}, in that the implications of items \textbf{1} and \textbf{2} may both be reversed and we have the following:
\begin{enumerate}
	\item For all $p,q \in \P$ we have that $p \leq_\P q \iff \pi(p) \leq_\Q \pi(q)$.
	\item For all $p,q \in \P$ we have that $p \perp q \iff \pi(p) \perp \pi(q)$.
\end{enumerate}
Item \textbf{1} holds since if otherwise, meaning if $\pi(p) \leq \pi(q)$ for $p,q \in \P$ and $p$ is not stronger than $q$, then there is a $p^*\leq p$ that is incompatible with $q$ by separativity, which means $\pi(p^*) \leq p$ and $\pi(p^*) \perp \pi(q)$, a contradiction. Thus, if $\P$ is a forcing notion, by property \textbf{1}, homomorphisms from $\P$ to $\Q$ are automatically embeddings.

\begin{lem} \label{lem:deltasize} Let $\P$ and $\Q$ be forcing notions. If $\pi: \P \to \Q$ is a dense homomorphism (which means $\pi$ is a dense embedding, as explained above) then $\delta(\P)=\delta(\Q)$. \end{lem} 
\begin{proof}
Suppose that $D \subseteq \P$ is dense and $|D|=\delta(\P)$. Then $\pi ``\P$ is dense in $\Q$ and $\pi``D$ is dense in $\pi``\P$, since $\pi$ is a homomorphism and preserves the partial order. This means that $\pi``D$ is dense in $\Q$. Any other dense set in $\Q$ has to have size at least $\delta(\Q)$, and thus $\delta(\Q) \leq |\pi``D| \leq |D| = \delta(\P)$ as desired. 

We now use the remarks showing that our homomorphisms are strong to find a dense subset of $\P$ that has size at most $\delta(\Q)$. Let $\Delta \subseteq \Q$ be dense satisfying $|\Delta|=\delta(\Q)$. Since $\pi `` \P$ is dense in $\Q$, for each $q \in \Delta$ there is a $p_q \in \P$ such that $\pi(p_q) \leq q$. Let $D^*$ be the set of such $p_q$'s, where there is only one $p_q$ chosen for each $q \in \Delta$. To show that $D^*$ is dense, let $p \in \P$. Then as $\Delta$ is dense, there is $q \in \Delta$ satisfying $q \leq \pi(p)$. Thus there is $\pi(p_q) \leq q \leq \pi(p)$ where $p_q \in D^*$, so we have $p_q \leq p$ showing that $D^*$ is dense. This means that $\delta(\P) \leq \delta(\Q)$ as desired.
\end{proof} 

If $\P$ is a dense subset of $\Q$, then, of course, $\delta(\P) = \delta(\Q)$.

\section{Fullness}
\label{sec:fullness}
In the definition of subcompleteness, in lieu of working directly with $H_\theta$ and its well-order for ``large enough" cardinals $\theta$ as our standard setup (as is often done for proper forcing, for example), I will follow Jensen and work with models $N$ of the form: 
	$$H_\theta \subseteq N = L_\tau[A] \models \ZFC^-,$$
where $\tau>\theta$ is a cardinal that is not necessarily regular. Such $H_\theta$ will need to be large enough so that $N$ has the correct $\omega_1$ and $H_{\omega_1}$. One justification for working with these models is that such $N$ will naturally contain a well order of $H_\theta$, along with its Skolem functions and other useful bits of information we would like to have at our disposal. Additionally a benefit of working with models of the form $L_\tau[A]$ is that $L_\tau[A]$ is easily definable in $L_\tau[A][G]$, if $G$ is generic, using $A$.
	
In the standard fashion we will look at countable elementary substructures $X$ of the $N$ as above; $X \preccurlyeq N$. We then take the countable transitive collapse of such an $X$, and write $\N \cong X$. We will refer to these embeddings by writing 	$$\sigma: \N \cong X \preccurlyeq N.$$
Often we will write $$\sigma: \N \prec N$$ to suppress mention of $X$, the range of $\sigma$.
	
In fact, for our purposes it will not be quite enough for such an $\N$ to be transitive, we need a bit more, exactly given by the property of \textit{fullness}. Fullness of a model ensures that it is not pointwise definable, so that there can be many elementary maps between the smaller structure and the larger structure $N$ in our setup. Before defining fullness exactly, let us give some more of the concepts and definitions we will be working with.

For the embeddings above that we will be working with, it is not hard to see what the critical point is. Given $\sigma: \N \prec N$ where $\N$ is countable and transitive, $\cp(\sigma)$ is exactly $\omega_1^{\N} = \sigma^{-1}(\omega_1^N)$, since $\N$ is countable.

\begin{fact} \label{fact:CPofourEmbeddings} Let $\N$, $N$ be transitive $\ZFC^-$ models, where $\N$ is countable and $H_{\omega_1} \subseteq N$, with $\sigma: \N \prec N$. Then $\cp(\sigma)=\omega_1^{\N}$ and $\sigma \rest (H_{\omega_1})^{\N} = \id$. \end{fact}
\begin{proof}
Let $\alpha = \cp(\sigma)$. Then $\alpha$ is a cardinal in $\N$, and $\alpha > \omega$ as both $\N$ and $N$ are models of $\ZFC^-$. 

It must be that $\alpha = \omega_1^\N$, since $\omega_1^\N$ is a countable ordinal in $V$ and thus in $N$, but 
	$\sigma(\omega_1^\N)=\omega_1^N>\omega_1^{\N}.$
Every $x \in H_\alpha^\N=(H_{\omega_1})^\N$ may be coded as a subset of $\omega$ and thus as an ordinal less than $\omega_1$ by $\gamma=\TC{x}$, and $\sigma(\gamma)=\gamma$, and the desired result follows. 
\end{proof}

Let $N = L_\tau[A]$ for some cardinal $\tau$ and set $A$, be a transitive $\ZFC^-$ model, let $X$ be a set, and let $\delta$ be a cardinal. Our notation for the \textbf{\emph{Skolem hull}}, in $N$, closing under $\delta \cup X$, is the following: 
	$$\Sk{N}{\delta}{X} = \text{the smallest } Y \preccurlyeq N \text{ satisfying } X \cup \delta \subseteq Y.$$
	
We gather two immediate, basic results that we will refer to later below.

\begin{lem} \label{lem:subsethull}
Let $N=L_\tau[A]$ be a transitive $\ZFC^-$ model, $\delta$ and $\gamma$ be cardinals, and $X$ a set. If $\delta \leq \gamma$ then $\Sk{N}{\delta}{X} \subseteq \Sk{N}{\gamma}{X}$. 
\end{lem}
\begin{proof}
This is trivial; if $t \in \Sk{N}{\delta}{X}$ then $t$ is $N$-definable from some $\xi < \delta \leq \gamma$ and $\vec x \in X$ so $t \in \Sk{N}{\gamma}{X}$ as well.
\end{proof}

\begin{lem} \label{lem:hullequality} Let $N$ be a transitive $\ZFC^-$ model, $\delta$ a cardinal, and $X$, $Y$ sets. If $\Sk{N}{\delta}{X} = \Sk{N}{\delta}{Y}$ then $\Sk{N}{\gamma}{X} = \Sk{N}{\gamma}{Y}$ for all $\gamma \geq \delta$. \end{lem}
\begin{proof}
	Let $t \in \Sk{N}{\gamma}{X}$. Then there is some formula $\varphi$ where $t$ unique such that $N \models \varphi(t, \xi, x)$ for $\xi < \gamma$, $x \in X$. But since then $x \in X \subseteq \Sk{N}{\delta}{X}= \Sk{N}{\delta}{Y} \subseteq \Sk{N}{\gamma}{Y}$, it must be the case that $t \in \Sk{N}{\gamma}{Y}$ as well. Likewise for the reverse inclusion.
\end{proof}

Before defining fullness, we need one more definition.

\begin{defn} We say that a transitive model $N$ is \emph{\textbf{regular}} in a transitive model $M$ so long as for all functions $f: x \To N$, where $x$ is an element of $N$ and $f \in M$, we have that $f``x \in N$. \end{defn}
To elucidate the definition further, we immediately have the following lemma:

\begin{lem} \label{lem:regularityequiv}
Let $N, M \models \ZFC^-$ be transitive. $N$ is regular in $M$ iff $N = H_\gamma^M$, where $\gamma = \height(N)$ is a regular cardinal in $M$. 
\end{lem}
\begin{proof}
For the backward direction, suppose that $N=H_\gamma^M$ where $\gamma = \height(N)$ is a regular cardinal in $M$. Then for all $f: x \longrightarrow N$, with $x \in N$ and $f \in M$, certainly $f``x \in N$ as well.

For the forward direction, indeed $\gamma$ has to be regular in $M$ since otherwise $M$ would contain a cofinal function $f: \alpha \longrightarrow \gamma$ where $\alpha < \gamma$. By the transitivity of $N$, this implies that $\alpha \in N$. Thus $\cup f``\alpha$ is in $N \models \ZFC^-$ by regularity, so $\gamma \in N$, a contradiction.
We have that $N \subseteq H_{\gamma}^M$ since $N$ is a transitive $\ZFC^-$ model, so the transitive closure of elements of $N$ may be computed in $N$ and thus have size less than $\gamma$, so they are in $H_\gamma^M$ as $\gamma \in M$. To show that $H_{\gamma}^M \subseteq N$, let $x \in H_{\gamma}^M$. We assume by $\in$-induction that $x \subseteq N$. Then there is a surjection $f: \alpha \twoheadrightarrow x$ where $\alpha < \gamma$, in $M$. Hence by regularity, $x = f``\alpha \in N$ as desired.
\end{proof}

We now define fullness.

\begin{defn} A structure $M$ is \emph{\textbf{full}} so long as $M$ is transitive, $\omega \in M$, and there is a $\gamma$ such that $M$ is regular in $L_\gamma(M)$ where $L_\gamma(M) \models \ZFC^-$.
\end{defn}

Perhaps this property seems rather mysterious, but the fullness of a countable structure guarantees that it is not pointwise definable, which we will see is a necessary property of some of the models that come up in the definition subcomplete forcing.

\begin{lem} If $M$ is countable and full, then $M$ is not pointwise definable. \end{lem}
\begin{proof} Suppose toward a contradiction that $M$ is countable, full, and pointwise definable. By fullness there is some $L_\gamma(M) \models \ZFC^-$ such that $M$ is regular in $L_\gamma(M)$. By pointwise definability, for each element $x \in M$, we have attached to it some formula $\varphi(x)$ such that $M \models \varphi(x)$ uniquely, meaning that $\varphi(y)$ fails for every other element $y \in M$. Thus in $L_\gamma(M)$ we may define a function $f: \omega \cong M$, that takes the $n$th formula in the language of set theory to its unique witness in $M$. In particular we have that $L_\gamma(M)$ witnesses that $M$ is countable. However, this would allow $M$ to witness its own countability, since $M$ must contain $f$ as well by regularity. \end{proof} 

Let $N$ be a transitive $\ZFC^-$ model. Suppose $\N \cong X \preccurlyeq N$ where $X$ is countable and $\N$ is full. Then as we will find, there may possibly be more than one elementary embedding $\sigma: \N \prec N$. 
If there was only one unique embedding, we would be able to define $\N$ pointwise, by elementarity of the unique map and since $\N$ is countable.

The following lemma shows that fullness is in some sense not much harder than transitivity to satisfy.

\begin{lem} \label{lem:fullclubs}
Let $\theta>\omega_1$ satisfy $H_\theta \subseteq N = L_\tau[A] \models \ZFC^-$ with $\tau>\theta$ regular and $A \subseteq \tau$, and let $s \in N$. Then 
$$ \set{ \omega_1^\N }{ \text{there is $\sigma$ such that $\sigma: \N \prec N$ where $\N$ is countable and full, and $s \in \ran(\sigma)$} } \subseteq \omega_1$$ contains a club.
\end{lem} 

\begin{proof}
Let $\tau'=(\tau^+)^{L[A]}$. Let $\sigma' : \overline{L_{\tau'}[A]} \cong X \preccurlyeq L_{\tau'}[A]$ where $X$ is countable and $L_{\overline{\tau'}}[\overline A]$ is the Mostowski collapse of $X$. Let $\overline \tau$ be the largest cardinal of $L_{\overline{\tau'}}[\overline A]$. We would like to show that $\N = L_{\overline{\tau}}[\overline A]$ is full, by showing it is regular in $L_{\overline{\tau'}}[\overline A] = L_{\overline{\tau'}}(\N)$.
\begin{claim} $\N$ is regular in $L_{\overline{\tau'}}[\overline A]$. \end{claim}
\begin{proof}[Pf.] 
Let $f:x \longrightarrow \overline N$, where $x \in \N$ and $f \in L_{\overline{\tau'}}[\overline A]$, and indeed $\ran(f)= f``x \in L_{\overline{\tau'}}[\overline A]$. Firstly, $f`` x \subseteq L_\gamma[\overline A]$, for some $\gamma < \overline \tau$, because $\overline \tau$ is a regular cardinal in $L_{\overline{\tau'}}[A]$. We inductively define a sequence of Skolem hulls in order to see that ultimately $f``x$ must be an element of $\N$.
Let 
	$$X_0 = \Sk{L_{\overline{\tau'}}[\overline A]}{\gamma}{\{f``x\}}.$$
In $L_{\overline{\tau'}}[\overline A]$, $X_0$ has size $\gamma$, so in particular it has size less than $\overline \tau$. Also we have that $L_\gamma[\overline A] \subseteq X_0$. As an aside, we point out that one  might already want to take the transitive collapse of this structure to obtain something that looks like $L_{\overline{\overline{\tau'}}}[\overline{\overline A}]$. The $\overline{\overline A}$ arising from this is not very easy to work with; it is not necessarily true that in this case, $\overline{\overline A} = \overline A \cap {\overline{\overline{\tau'}}}$ as one would hope. To remedy this we will go on to define a hull whose transitive collapsed version of $\overline A$ is an initial segment of $\overline A$, ultimately making the collapsed structure definable in $L_{\overline \tau}[\overline A]$. To do this, inductively assume $X_n$ is defined with size less than $\overline \tau$, and set $\gamma_n = \sup(X_n) \cap \overline \tau$. We have that $|\gamma_n|<\overline \tau$, as $\overline \tau$ is regular in $L_{\overline{\tau'}}[\overline A]$. We let 
	$$X_{n+1} = \Sk{L_{\overline{\tau'}}[\overline A]}{\gamma_n}{X_n}.$$
This defines an elementary chain $\seq{X_n}{n<\omega}$ in $L_{\overline{\tau'}}[\overline A]$. So $X_\omega = \bigcup_{n<\omega}X_n$ is an elementary substructure of $L_{\overline{\tau'}}[\overline A]$. Additionally, let $\gamma_\omega:=X_\omega \cap \overline \tau = \sup_{n<\omega}{\gamma_n}$. 

Let $k:\overline{L_{\overline{\tau'}}[\overline{A}]} \cong X_\omega \preccurlyeq L_{\overline{\tau'}}[\overline A]$. Then $\overline{L_{\overline{\tau'}}[\overline A]}=L_{\overline{\overline{\tau'}}}[\overline{\overline{A}}]$ and by construction, $k \rest \gamma_\omega = \id$.
Additionally we now have $$\overline{\overline{A}} = k^{-1}``\overline A = k^{-1}``(\overline A \cap \gamma_\omega)=\overline A \cap \gamma_\omega.$$ Since $L_\gamma[\overline A] \subseteq X_0$, we have that $k^{-1}(f``x)=f``x$. Thus $f``x \in L_{\overline{\overline{\tau'}}}[\overline A \cap \gamma_\omega] \in L_{\overline{\tau}}[\overline A] = \N$, since $\overline{\overline{\tau'}} <\tau$ and $\overline A \cap \gamma_\omega < \overline \tau$.
Thus $\N$ is regular in $L_{\overline{\tau'}}[\overline A]$, proving the \textit{Claim}.
\end{proof}
Thus $\N$ is full. To show the claim, we need to show that there are club-many such $\N$s. But this is true because there are club-many relevant Skolem hulls from which the $\N$ arise.

In particular, let
\begin{align*} \Gamma = \Big\{ \N = L_{\overline \tau}[\overline A] \st & \text{ for some $\gamma<\omega_1$, } L_{\overline{\tau}'}[\overline A] \cong \Sk{L_{\tau'}[A]}{\gamma}{\{s\}} \preccurlyeq L_{\tau'}[A] \\
	&\text{ where } \overline \tau \text{ is the largest cardinal in } L_{\overline{\tau'}}[\overline A]\Big\}.
	\end{align*}

Then $C = \set{ \omega_1^\N }{ \N \in \Gamma } \subseteq \omega_1$ is club. 

Clearly $C$ is unbounded, since for any $\gamma < \omega_1$, there is $\sigma' : L_{\overline{\tau}'}[\overline A] \cong \Sk{L_{\tau'}[A]}{\gamma}{\{s\}} \preccurlyeq L_{\tau'}[A]$. Let $\sigma'(\overline s)=s$. Since $\overline \tau > \omega_1^\N$ by elementarity, we have that $\sigma=\sigma' \rest L_{\overline \tau}[\overline A]$ is an elementary embedding as well; $\sigma: \N \prec N$, and $\sigma(\overline s)=s$. Then the critical point of $\sigma$, namely $\omega_1^\N$, is above $\gamma$.

To see that $C$ is closed, suppose we have an unbounded set of of ${\N} \in \Gamma$ and an associated sequence of $\alpha=\omega_1^{{\N}}$. As we saw above, each of these ${\N}$ can be thought of as the domain of some elementary $\sigma: \N \cong X \preccurlyeq N$, where $X$ is countable. Then we can take the union of these $X$ to obtain $X \preccurlyeq N$. To form $\sigma: \N \prec N$ take the Mostowski collapse. The critical point of $\sigma$ will be the supremum of all of the $\alpha$'s.
\end{proof}

The following lemma will prove useful when showing that various forcing notions are subcomplete, since subcompleteness requires certain ground-model Skolem hulls to match those in forcing extensions.

\begin{lem}[Jensen] \label{lem:Ctrick} Let $\P$ be a forcing notion, and let $\delta=\delta(\P)$ be the smallest size of a dense subset in $\P$. Suppose that $\P \in H_\theta \subseteq N$ where $N=L_\tau[A] \models \ZFC^-$. Let $\sigma : \N \prec N$ where $\N$ is countable and full, and let $\sigma(\overline \P)=\P$.

Suppose $G \subseteq \P$ is $N$-generic and $\G \subseteq \overline \P$ is $\N$-generic, and that $\sigma``\G \subseteq G$, so $\sigma$ lifts (or extends) in $V[G]$ to an embedding $\sigma^*:\N[\G] \prec N[G]$. 

Then 
	$$N \cap \sk{N[G]}{\delta}{\sigma^*} = \sk{N}{\delta}{\sigma}.$$
\end{lem}
\begin{proof}
First, we establish that $\sk{N}{\delta}{\sigma} \subseteq \sk{N[G]}{\delta}{\sigma^*} \cap N$. To see this, let $x \in \sk{N}{\delta}{\sigma}$. Then $x$ is $N$-definable from $\xi<\delta$ and $\sigma(\overline z)$ where $\overline z \in \N$. Since $\sigma^*$ extends $\sigma$, this means that $x \in N$ and that $x$ is $N[G]$-definable from $\xi$ and $\sigma^*(\overline z)$. This is because $N=L_\tau[A]$ is definable in $N[G]$ using $A$.

For the other direction, let $x \in \sk{N[G]}{\delta}{\sigma^*} \cap N$. Then $x$ is $N[G]$-definable from $\xi < \delta$ and $\sigma^*(\overline z)$ where $\overline z \in \N[\G]$; and also $x \in N$. Letting $\dot{\overline z} \in \N^{\overline{\P}}$ such that $\overline z = \dot{\overline z}^{\G}$ we have $$\sigma^*(\overline{z}) = \sigma^*({\dot{\overline z}}^{\G})=\sigma(\dot{\overline{z}})^{G}$$ so we have that there is some formula $\varphi$ such that $x$ is the unique witness:
	$$x = \text{that } y \text{ where } N[G] \models \varphi(y, \xi, \sigma(\dot{\overline z})^{G}).$$ 
Take $f \in \N$ mapping $\overline \delta$ onto a dense subset of $\overline{\P}$. Then $\sigma(f)$ maps $\delta$ onto a dense subset of $\P$. Thus there is $\nu < \delta$ such that $\sigma(f)(\nu) \in G$ and 
	$\sigma(f)(\nu) \forces \varphi(\check x, \check \xi, \sigma(\dot{\overline{z}})).$ 
Thus 
	$$x = \text{that } y \text{ where } \sigma(f)(\nu) \forces^N_{\P} \varphi(\check y, \check \xi, \sigma(\dot{\overline{z}}))$$ 
so $x \in \sk{N}{\delta}{\sigma}$.
\end{proof}

\section{Barwise Theory}
\label{sec:BarwiseTheory}
In order to show that many posets are subcomplete, Jensen takes advantage of Barwise Theory and techniques using countable admissible structures to obtain transitive models of infinitary languages. Barwise creates an $M$-finite predicate logic, a first order theory in which arbitrary, but $M$-finite, disjunctions and conjunctions are allowed. The following is an outline of \cite[Chapter 1 \& 2]{Jensen:2012fr}. 

\begin{defn} Let $M$ be a transitive structure with potentially infinitely many predicates. A theory defined over $M$ is $M$-\textbf{\textit{finite}} so long as it is in $M$. A theory is $\Sigma_1(M)$, also known as \textbf{\textit{$M$-recursively enumerable}} or $M$-$re$, if the theory is $\Sigma_1$-definable, with parameters from $M$. \end{defn}
Of course we can generalize this to the entire usual Levy hierarchy of formulae, but for our purposes we only need to know what  it means for a theory to be $\Sigma_1(M)$. If $\mathcal L$ is a $\Sigma_1$(M)-definable language or theory, the rough idea is that to check whether a sentence is in $\mathcal L$, one should imagine enumerating the formulae of $\mathcal L$ to find the sentence and a witness to it in the structure $M$. Below we elaborate on what models we are working with. 

\begin{defn} A transitive structure $M$ is \textbf{\emph{admissible}} if it models the axioms of \textsf{Kripke-Platek Set Theory} (\textsf{KP}) which consists of the axioms of \textsf{Empty Set}, \textsf{Pairing}, \textsf{Union}, $\Sigma_0$-\textsf{Collection}, and $\Sigma_0$-\textsf{Separation}. \end{defn}

Jensen also makes use of models of $\textsf{ZF}^-$ that are not necessarily well-founded.
\begin{defn} Let $\mathfrak A = \langle A, \in_{\mathfrak A}, B_1, B_2, \dots \rangle$ be a (possibly) ill-founded model  of $\textsf{ZF}^-$, where $\mathfrak A$ is allowed to have predicates other than $\in$. The \textbf{\emph{well-founded core}} of $\mathfrak A$, denoted $\wfc(\mathfrak A)$, is the restriction of $\mathfrak A$ to the set of all $x \in A$ such that $\in_{\mathfrak A} \cap \mathcal C(x)^2$ is well founded, where $\mathcal C(x)$ is the closure of $\{x\}$ under $\in_{\mathfrak A}$. A model $\mathfrak A$ of $\textsf{ZF}^-$ is \textbf{\emph{solid}} so long as $\wfc(\mathfrak A)$ is transitive and $\in_{\wfc(\mathfrak A)}=\in \cap \wfc(\mathfrak A)^2$. \end{defn}

Jensen \cite[Section 1.2]{Jensen:2012fr} notes that every consistent set of sentences in $\textsf{ZF}^-$ has a solid model, and if $\mathfrak A$ is solid, then $\omega \subseteq \wfc(\mathfrak A)$. In addition,

\begin{fact}[Jensen] If $\mathfrak A \models \textsf{ZF}^-$ is solid, then $\wfc(\mathfrak A)$ is admissible. \end{fact}

\begin{defn} \label{def:InTheoriesAndBasicAxioms} The context for Barwise theory is countable admissible structures. If $M$ is admissible, we work with infinitary, axiomatized theories $\mathcal L$ in $M$-finitary logic, called \textit{\textbf{$\in$-theories}}, with a fixed predicate $\dot \in$ and \textbf{\emph{special constants}} denoted $\underline x$ for elements $x \in M$. Our underlying axioms for these $\in$-theories will always involve $\ZFC^-$ and some basic axioms ensuring that $\dot \in$ behaves nicely; 
the \textsf{Basic Axioms} are: \begin{itemize}
	\item \textsf{Extensionality}
	\item A statement positing the extensionality of $\dot \in$, which is a scheme of formulae defined for each member of $M$. For each $x \in M$, include an $M$-$re$ sentence (meaning it quantifies over $M$-finite sentence): 
	$$\forall v \left( v \ \dot \in\, \underline x \iff \bigdoublevee_{z \in x} v= \underline z \right).$$
\end{itemize} Here $\bigdoublevee$ denotes an infinite disjunction in the language.\end{defn}

An important fact ensured by our \textsf{Basic Axioms} is that the interpretations of these special constants in any solid model of the theory are the same as in $M$:

\begin{fact}[Jensen] Let $M$ be as in the above definition. Let $\mathfrak A$ be a solid model of the $\in$-theory $\mathcal L$. Then for all $x \in M$, we have that $\underline{x}^{\mathfrak A} = x \in \wfc(\mathfrak A)$. \end{fact}
\begin{proof}[Pf.] Shown by $\in$-induction. \end{proof}

Jensen uses the techniques of Barwise to come up with a proof system in this context, in which consistency of $\in$-theories can be discussed. In particular, $\in$-theories are correct: if we have a model of such a theory, then it is consistent. 
\begin{fact}[\textbf{\emph{Barwise Correctness}}] \label{fact:correctness} 
Let $\mathcal L$ be an $\in$-theory. If $A$ is a set of $\mathcal L$-statements and $\mathfrak A \models A$, then $A$ is consistent. \end{fact}

Furthermore, compactness and completeness are shown, relativized to the $M$-finite predicate logics that are used here. In our context, for countable admissible structures $M$, we will obtain solid models of consistent $\Sigma_1(M)$ $\in$-theories. In particular, the form of Barwise Completeness that we make use of here is stated below.

\begin{fact}[\textbf{\emph{Barwise Completeness}}] \label{fact:completeness} Let $M$ be a countable admissible structure. Let $\mathcal L$ be a consistent $\Sigma_1(M)$ $\in$-theory such that $\mathcal L \vdash \textsf{ZF}^-$. Then $\mathcal L$ has a solid model $\mathfrak A$ such that $$\Ord \cap \wfc(\mathfrak A) = \Ord \cap M.$$ \end{fact}


We will need the following definition, which is a generalization of fullness.
\begin{defn} A transitive $\ZFC^-$ model $N$ is \emph{\textbf{almost full}} so long as $\omega \in N$ and there is a solid $\mathfrak A \models \ZFC^-$ with $N \in \wfc(\mathfrak A)$ and $N$ is regular in $\mathfrak A$. \end{defn}
Clearly if $N$ is full, then $N$ is almost full.

A useful technique when showing a particular forcing is subcomplete, once many different embeddings can be constructed that approximate the embedding required for subcompleteness, is to be able to transfer the consistency of $\in$-theories over one admissible structure to another.
\begin{defn} If $N$ is a transitive $\ZFC^-$ model, let $\delta_N$ be the least $\delta$ such that $L_\delta(N)$ is admissible. \end{defn}

\begin{fact}[\textbf{\emph{Transfer}}]\label{fact:Transfer} Suppose that for $M$ admissible, $\mathcal L(M)$ is a $\Sigma_1(M)$ infinitary $\in$-theory. Let $N_1$ be almost full, and suppose that $k: N_1 \prec N_0$ cofinally. If both $\mathcal L(L_{\delta_{N_1}}(N_1))$ is $\Sigma_1$ over parameters $N_1, p_1, \dots, p_n \in N_1$ and $\mathcal L(L_{\delta_{N_0}}(N_0))$ is $\Sigma_1$ over parameters $N_0, k(p_1), \dots, k(p_n)$, then  
if $\mathcal L(L_{\delta_{N_1}}(N_1))$ is consistent, it follows that $\mathcal L(L_{\delta_{N_0}}(N_0))$ is consistent as well. \end{fact} 

\section{Liftups}
\label{sec:Liftups}
The following definitions are to describe a method to obtain emeddings, a technique that is ostensibly the ultrapower construction. These embeddings facilitate the use of Barwise theory to obtain the consistency of the existence of desirable embeddings. We follow \cite[Chapter 1]{Jensen:2012fr} here.

\begin{defn} Let $\N$ and $N$ be transitive $\ZFC^-$ models. We say that an elementary embedding $\sigma: \N \prec N$ is \emph{\textbf{cofinal}} so long as for each $x \in N$ there is some $u \in \N$ such that $x \in \sigma(u)$. 

Let $\alpha \in \N$. We say that $\sigma$ is \emph{\textbf{$\alpha$-cofinal}} so long as every such $u$ has size less than $\alpha$ as computed in $\N$. \end{defn}
	
\begin{defn} Let $\alpha > \omega$ be a regular cardinal in $\N$, a transitive $\ZFC^-$ model. Let 
	$$\overline \sigma: H^{\N}_\alpha \prec H \text{ cofinally,}$$ where $H$ is transitive. By a \emph{\textbf{transitive liftup}} of $\langle \N, \overline \sigma \rangle$ we mean a pair $\langle N_* , \sigma_* \rangle$ such that 
\begin{itemize} 
	\item $N_*$ is transitive
	\item $\sigma_*:\N \prec N_*$ $\alpha$-cofinally
	\item $\sigma_* \upharpoonright H_{\alpha}^{\N}= \overline \sigma$ \qedhere
\end{itemize}	
\end{defn}

Reminiscent of ultrapowers, transitive liftups can be characterized in the following way:
\begin{lem}[Jensen] \label{lem:liftupchar} Let $\N$, $N$ be transitive $\ZFC^-$ models with $\sigma: \N \prec N$. Then,
$\sigma$ is $\alpha$-cofinal $\iff$ elements of $N$ are of the form $\sigma(f)(\beta)$ for some $f: \gamma \to \overline N$ where $\gamma < \alpha$ and $\beta < \sigma(\gamma)$. \end{lem} 
\begin{proof} We show each direction of the equivalence separately.
	\begin{paragraph}{``$\Longrightarrow$":} Let $x \in N$, and take $u \in \overline N$ with $x \in \sigma(u)$ such that $|u| <\alpha$ in $\overline N$. Let $|u|=\gamma$, and take $f:\gamma \to u$ a bijection in $\overline N$. Then $\sigma(f):\sigma(\gamma) \to \sigma(u)$ is also a bijection in $N$ by elementarity. Since $x \in \sigma(u)$ we also have that $x$ has a preimage under $\sigma(f)$, say $\beta$. So $\sigma(f)(\beta)=x$ as desired. \end{paragraph}
	\begin{paragraph}{``$\Longleftarrow$":} Let $x=\sigma(f)(\beta)$ be an element of $N$, for $f:\gamma \to \overline N$ where $\gamma < \alpha$ in $\overline N$ and $\beta < \sigma(\gamma)$. Define $u = f``\gamma$. Then in $\overline N$ we have that $|u|<\alpha$. In addition we have that $x \in \sigma(u)$, since $\sigma(u)$ is in the range of $\sigma(f)$, where $x$ lies. \end{paragraph}
\end{proof}

Furthermore, Jensen shows that transitive liftups exist so long as an embedding already exists, using an ultrapower-like construction, and have a uniqueness property.

\begin{fact}[\textbf{\emph{Interpolation}}]\label{fact:Interpolation} Let $\sigma:\N \prec N$ with $\N \models \ZFC^-$ transitive, and let $\alpha \in \N$ be a regular cardinal. Then: \begin{enumerate}
	\item The transitive liftup $\langle N_*, \sigma_* \rangle$ of $\langle \N, \sigma \rest H^{\N}_\alpha \rangle$ exists.
	\item There is a unique $k_*:N_* \prec N$ such that $k_* \circ \sigma_* = \sigma$ and $k_* \rest \bigcup \sigma `` H^{\N}_\alpha = \id$.
\end{enumerate}
\end{fact} 

For the following useful lemma, we will need to define the following, more general, notion of liftups. Of course the rich theory is established by Jensen, and is explored in detail in his notes. We will only use this more general definition for the following lemma, which is why we did not introduce liftups in this way from the beginning.
\begin{defn} Let $\mathfrak A$ be a solid model of $\ZFC^-$ and let $\tau \in \wfc(\mathfrak A)$ be an uncountable cardinal in $\mathfrak A$. Let $$\sigma: H_{\tau}^\mathfrak A \prec H \text{ cofinally,}$$ where $H$ is transitive. Then by a \textit{\textbf{liftup}} of $\langle \mathfrak A, \sigma\rangle$, we mean a pair $\langle {\mathfrak A}_*, \sigma_* \rangle$ such that 
\begin{itemize}
	\item $\sigma_* \supseteq \sigma$
	\item ${\mathfrak A}_*$ is solid
	\item $\sigma_*: \mathfrak A \to_{\Sigma_0} {\mathfrak A}_*$ $\tau$-cofinally
	\item $H \in \wfc({\mathfrak A}_*)$ \qedhere
\end{itemize}
\end{defn}
\begin{fact}[Jensen]\label{fact:solidliftup} Let $\mathfrak A$ be a solid model of $\ZFC^-$. Let $\tau > \omega$, $\tau \in \wfc(\mathfrak A)$, and let $$\sigma: H_{\tau}^{\mathfrak A} \prec H \text{ cofinally,}$$  where $H$ is transitive. Then $\langle \mathfrak A, \sigma \rangle$ has a liftup $\langle \mathfrak A_*, \sigma_* \rangle$.
\end{fact}

The following lemma states that transitive liftups of full models are almost full.

\begin{lem}[Jensen]\label{lem:liftupfull} Let $N =L_\tau[A] \models \ZFC^-$ and $\sigma: \N \prec N$ where $\N$ is full. Suppose that $\langle N_*, \sigma_* \rangle$ is a transitive liftup of $\langle \N, \overline \sigma \rangle$. Then $N_*$ is almost full. \end{lem}
\begin{proof}
	Let $L_\gamma(\N)$ witness the fullness of $\N$. We will now apply Interpolation (\emph{Fact} \ref{fact:Interpolation}) to $\mathfrak A= L_\gamma(\N)$, which makes sense since certainly $\mathfrak A$ is a transitive model $\ZFC^-$. Additionally, by \textbf{Lemma \ref{lem:regularityequiv}} we have that $\N = H_\tau^{\mathfrak A}$, where $\tau = \height(\N)$. Since $\langle N_*, \sigma_* \rangle$ is a transitive liftup, we have that $$\sigma_*: H_\tau^{\mathfrak A} \prec N_* \text{ cofinally,}$$ where $N_*$ is transitive. Thus since $\mathfrak A$ is transitive, $\langle \mathfrak A, \sigma_* \rangle$ has a liftup $\langle \mathfrak A_*, {\sigma_*}_* \rangle$, where $\mathfrak A_* \models \ZFC^-$ since $\mathfrak A$ does, $\mathfrak A_*$ is solid, where $${\sigma_*}_* : \mathfrak A \prec \mathfrak A_* \text{ $\tau$-cofinally.}$$ 
	We have that  $N_* \subseteq \wfc({\mathfrak A}_*)$ and $\tau_*= {\sigma_*}_*(\tau)=\height(N_*)$ is regular since $\tau$ is. Furthermore, we will show that $N_* = H_{\tau_*}^{\mathfrak A_*}$, completing the proof:
	
	Certainly it is the case that $N_* \subseteq H_{\tau_*}^{\mathfrak A^*}$. 
	But if $x \in H_{\tau_*}$ in $\mathfrak A_*$, then by regularity we have that $x \in {\sigma_*}_*(u)$ in $\mathfrak A_*$, where $u \in \mathfrak A$, and $|u| < \tau$ in $\mathfrak A$. Let $v=u \cap H_\tau$ in $\mathfrak A$. Then $v \in \N$, since $\N$ is regular in $\mathfrak A$. But then $x \in \sigma_*(v) \in N_*$. So $x \in N_*$.
\end{proof}

\chapter{Subcomplete Forcing and Its Relatives}
\label{chap:SCandrelatives}

Before defining subproper forcing, it is useful to first give the relevant definition of proper forcing. Then we will see how the prefix ``sub" alters that definition to arrive at subproper forcing. Later on, we will draw comparisons to complete forcing and subcomplete forcing; showing how the ``sub" prefix alters the definition of complete forcing to finally obtain subcomplete forcing. 

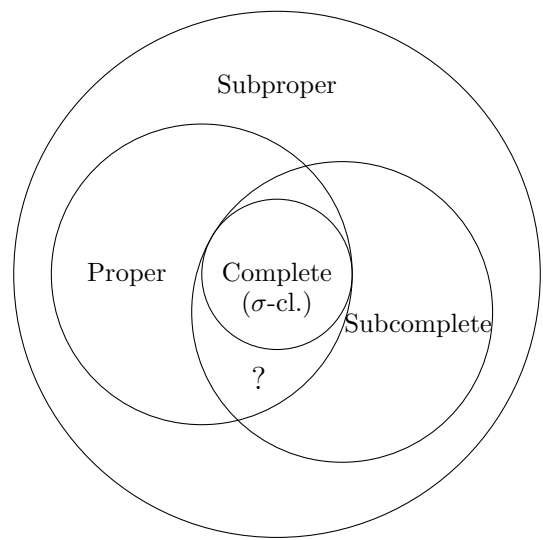
\begin{figure}[h!]\begin{center}
    \begin{tikzpicture}
      \draw (0:0cm) circle (3.5cm);
      \node at (90:2.5cm) {\footnotesize Subproper};
      \draw (90:0cm) circle (1cm) node [text=black] {\footnotesize Complete};
      \node at (90:-0.4cm){\footnotesize ($\sigma$-cl.)};
      \node at (80:-1.4cm){?};
      \draw (180:1cm) circle (2cm);
      \node at (180:2cm) {\footnotesize Proper};
      \draw (330:1cm) circle (2cm);
      \node at (340:2cm){\footnotesize Subcomplete};
    \end{tikzpicture}\end{center}
\end{figure} 

The above diagram roughly represents the picture of how these forcing notions fit together. We shall see in  \ref{subsec:SCForcingandCCForcing} that complete forcing gives rise to the same class of forcing notions as countably closed forcing, so in particular, this means that countably closed forcing notions are subcomplete, and complete forcings are proper.

From the diagram it looks as though there might be forcing notions that are both subcomplete and proper but not countably closed. It is not clear whether this is the case, and is a topic of future study.

\begin{question} Is there a forcing notion that is both proper and subcomplete but not countably closed? \end{question}

\section{Subproper Forcing}
\label{sec:SubproperForcing}
The following relevant characterization of properness, roughly speaking, replaces the notion of genericity below a master condition with the lifting of an embedding below a master condition.

First recall a standard definition of properness as in \cite{Jech:2002qr}:
\begin{defn} A forcing notion $\P$ is \textbf{\emph{proper}} so long as for sufficiently large $\theta$, for all $\P \in X \preccurlyeq \langle H_\theta, < \rangle$ where $X$ is countable and $<$ is a well-order of $H_\theta$, and all $p \in \P \cap X$ there is a \textbf{\textit{master condition}} $p^* \leq p$ that is $X$-generic. Namely, whenever $p^* \in G$ is $\P$-generic, $G \cap X$ is a $\P$-generic filter over $X$. 

Alternatively $\P$ is proper so long as there is a large enough $\theta$ for which there is a club of such countable $X \preccurlyeq \langle H_\theta, < \rangle$ where $X \in \P$.
\end{defn}

The following proposition shows how the definition of properness is understood in the context of lifting embeddings on countable transitive structures. 

\begin{prop} 
A forcing notion $\P$ is proper $\iff$
for sufficiently large $\theta$ we have that; letting: \begin{enumerate}
	\item $\P \in H_\theta \subseteq N = L_\tau[A] \models \ZFC^-$ where $\tau>\theta$ and $A \subseteq \tau$
	\item $\sigma: \N \cong X \preccurlyeq N$ where $X$ is countable and $\N$ is transitive
	\item $\sigma(\overline \theta, \overline{\P}, \overline s)=\theta, \P, s$ for some $s \in N$;
\end{enumerate}
for any $\overline p \in \overline{\P}$ there is $p^* \leq \sigma(\overline p)$ in $\P$ such that whenever $G \ni p^*$ is $\P$-generic, $\sigma^{-1}``G = \G \subseteq \overline{\P}$ is an $\N$-generic filter. 

Thus in particular, below $p^*$ we have that $\sigma$ lifts to an elementary embedding $\sigma^*:\N[\G] \prec N[G]$.
\end{prop}
\begin{proof} We show each direction of the biconditional separately.
\begin{paragraph}{``$\Longrightarrow$":} Let $\P$ be proper. 
In order to show that the alternative characterization is also satisfied, let $\theta$ be large enough so that we are in the scenario in which items $\boldsymbol{\mathit{1}} - \boldsymbol{\mathit{3}}$ hold, where $\pi: \N \cong X$.
	Let $\overline p \in \overline{\P}$. Then by the properness of $\P$ there is a master condition $p^* \leq \pi(\overline p)$. Let $G$ be generic for $\P$ with $p^* \in G$. Then $G \cap X$ is a $\P$-generic filter over $X$. Since $\pi$ is an isomorphism, it is clear that $\G = \pi^{-1}``(G \cap X)$ is $\overline \P$-generic over $\N$ since $G \cap X$ is generic over $X$.
\end{paragraph}
\begin{paragraph}{``$\Longleftarrow$":} 
Let $\theta$ be large enough so that the alternative characterization is satisfied for $\P$. 
To show that $\P$ is proper, we show that there are club-many countable $X \preccurlyeq \langle H_\theta, < \rangle$ such that $\P \in X$ and for all $p \in \P \cap X$ there is a master condition $p^* \leq p$ that is $X$-generic. 

To do this, first take $\langle H_\theta, < \rangle \in N = L_\tau[A] \models \ZFC^-$ with $\tau>\theta$ and $A \subseteq \tau$. We will show that 
	$$C=\set{ X \preccurlyeq \langle H_\theta, < \rangle }{ \text{$X$ is countable, $\P \in X$, and $\SH^N(X) \cap \P = X \cap \P$}} \subseteq \mathcal P_{\omega_1}(H_\theta) \text{ is club.}$$

\begin{claim} $C$ is club. \end{claim}
\begin{proof}[Pf.]
Let's show that 	$C$ is unbounded. Let $Z$ be a countable subset of $H_\theta$. We construct $X \supseteq Z$ such that $X \preccurlyeq \langle H_\theta, < \rangle$, $\P \in X$, and $\SH^N(X) \cap \P = X \cap \P$. We define $X$ inductively, by defining a sequence of $X_n \preccurlyeq H_\theta$ and $Y_n =\SH^N(X_n) \preccurlyeq N$.  Let $X_0 = \Sk{H_\theta}{Z}{\{\P\}}$. Assuming $X_n$ has been defined, let $X_{n+1}=\Sk{H_\theta}{X_n}{(Y_n \cap \P)}$. Let $X= \cup_{n<\omega} X_n$. Then we claim that $X \in C$ is as desired. Of course $\P \in X$ as $\P \in X_0$. Additionally, letting $Y:=\cup_{n<\omega}Y_n$, we have that $Y=\SH^N(X)$ as $X \subseteq Y$ and $Y$ is the smallest elementary substructure of $N$ containing $X$. So to show that $Y \cap \P = X \cap \P$, let $y \in Y \cap \P$. Then $y \in Y_n \cap \P$ for some $n<\omega$. Thus $y \in X_{n+1}$. And the same argument shows that $x \in X \cap \P$ means $x \in X_n$ for some $n<\omega$, meaning $x \in Y_{n+1}$. 

$C$ is also closed, since supposing there is a countable elementary chain of $X_n$ for $n<\omega$, the union $X_\omega=\cup_{n<\omega}$ is also a countable, elementary substructure of $H_\theta$. Furthermore, $\P \in X$ as $\P \in X_0$, and $\SH^N(X_\omega) = \cup_{n<\omega} \SH^N(X_n)$, so in an argument similar to the above, $\SH^N(X_\omega) \cap \P = X_\omega \cap \P$ as desired.

This completes the proof of the \textit{Claim}.
\end{proof}
To see that $\P$ is indeed proper, let $X \in C$, and let $p \in \P \cap X$. We want to find a master condition $p^*$ below $p$. We now use the alternative characterization. Indeed, let \begin{itemize}
	\item $\sigma:\N \cong \SH^N(X)=Y \preccurlyeq N$ where $\N$ is transitive.
	\item $\sigma(\overline \theta, \overline{\P}, \overline p)=\theta, \P, p$.	
\end{itemize}
By the alternative characterization, there is $p^* \leq p$ such that whenever $p^* \in G$ where $G$ is $\P$-generic over $N$, we have that $\sigma^{-1}``G=\G$ is $\overline \P$-generic over $\N$. Then $\sigma``\G=G\cap Y$ is $\P \cap Y$-generic over $Y$. But since $\P \cap Y = \P \cap X$, we have that $G \cap X$ is $\P \cap X$-generic over $X$ as desired.
\end{paragraph}
\end{proof}
Subproperness is a weakening of properness which asks that below a condition, there is \textit{some} embedding in the extension that lifts, not necessarily the one we started with. Extra conditions for the embedding and the structures it deals with are needed; namely the replacement of transitivity with fullness is crucial, since otherwise there may not consitently be more than one such elementary embedding.

While proper forcings satisfy the \textit{\textbf{countable covering property}}, meaning that every countable set of ordinals in $V[G]$ is included in a set in $V$ that is countable in $V$, subproper forcings do not necessarily have to have this property in general. We shall see that subcomplete forcing notions are clearly all subproper, and there are subcomplete forcing notions that do not satisfy the countable covering property, e.g. Prikry forcing.

\begin{defn} 
A forcing notion $\P$ is \emph{\textbf{subproper}} so long as
for sufficiently large $\theta$ we have that whenever we are in the following standard setup: \begin{itemize}
	\item $\P \in H_\theta \subseteq N = L_\tau[A] \models \ZFC^-$ where $\tau>\theta$ and $A \subseteq \tau$
	\item $\sigma: \N \cong X \preccurlyeq N$ where $X$ is countable and $\N$ is full
	\item $\sigma(\overline \theta, \overline{\P}, \overline s)=\theta, \P, s$ for some $s \in N$;
\end{itemize}
then, for any $\overline q \in \overline{\P}$ there is $p \leq \sigma(\overline q)$ in $\P$ such that whenever $G \ni p$ is $\mathbb P$-generic, there is an embedding $\sigma'\in V[G]$ satisfying: \begin{enumerate}
	\item $\sigma': \N \prec N$
	\item $\sigma'(\overline \theta, \overline{\P}, \overline s)=\theta, \P, s$
	\item $\sk{N}{\delta(\P)}{\sigma'} = \Sk{N}{\delta(\P)}{X}$
	\item $\G = \sigma'^{-1}``G$ is an $\N$-generic filter.
\end{enumerate}
In particular, below the condition $p$ the  embedding $\sigma' \in V[G]$ \textit{lifts} by \textbf{4}, since $\sigma``\G \subseteq G$, to an embedding $\sigma'^* \in V[G]$ where $\sigma'^*:\N[\overline G] \prec N[G]$.

We say that such a $\theta$ as above \textit{\textbf{verifies the subproperness}} of $\P$.

Often we write $\delta$ instead of $\delta(\P)$ when there should be no confusion as to which poset $\P$ we are working with.\footnote{See Section \ref{sec:delta} for more on $\delta(\P)$.}
\end{defn}
Condition \textbf{3} ensures that subproper forcings may be iterated. Indeed, Jensen proves that there is an iteration theorem for subproper forcing, and uses it to show the consistency of $\textsf{SuPFA}$, the subproper forcing axiom, given the existence of a supercompact cardinal.

\begin{thm} \label{thm:SubproperMM}
Subproper forcings preserve stationary subsets of $\omega_1$.
\end{thm}
\begin{proof}
Suppose not. Let $S \subseteq \omega_1$ be stationary. Let $\P$ be subproper, and suppose toward a contradiction that we have $\dot C \in V^{\P}$ such that there is $p \in \P$, 
	$$p \Vdash ``\dot C \subseteq \check \omega_1 \text{ is club}" \ \land \ \check S \cap \dot C = \emptyset.$$
Choose $\theta$ large enough so that $\P \in H_\theta$, and suppose we have $H_\theta \subseteq N = L_\tau[A] \models \ZFC^-$ for some $\tau > \theta$ and $A \subseteq \tau$. 
	Recall that $$B=\set{ \omega_1^{\N} }{ \text{where } \sigma: \overline N \prec N \text{ and $\N$ is countable, full with } \sigma(\overline \theta, \overline{\P}, \overline S, \overline{\dot C}, \overline p) = \theta, \P, S, \dot C, p } \subseteq \omega_1$$ contains a club by \textbf{Lemma \ref{lem:fullclubs}}.

Take $\alpha \in S \cap B$. Thus $\alpha = \omega_1^{\N}$ for some $\N$ such that: \begin{itemize} 
	\item $\sigma: \N \cong X \preccurlyeq N$ where $X$ is countable and $\N$ is full.
	\item $\sigma(\overline \theta, \overline{\P}, \overline S, \overline{\dot C}, \overline p) = \theta, \P, S, \dot C, p$.
\end{itemize}
Furthermore by elementarity, we have in $\N$ that $\overline p \forces ``\overline{\dot C} \subseteq \delta \text{ is club}."$ Apply subproperness to obtain $q \leq \sigma(\overline p) = p$ such that whenever $G \ni q$ is $\P$-generic, there is $\sigma' \in V[G]$ satisfying: \begin{itemize}
	\item $\sigma': \N \prec N$
	\item $\sigma'(\overline \theta, \overline{\P}, \overline S, \overline{\dot C}, \overline p) = \theta, \P, S, \dot C, p$
	\item $\overline G = \sigma'^{-1}``G$ is an $\N$-generic filter. 
\end{itemize}
So $\sigma'$ lifts to an embedding $\sigma^*: \N[\G] \prec N[G]$ in $V[G]$. 
Let $\overline C = \overline{\dot C}^{\G}$, $C = \dot C^G$. Since $q \in G$ is stronger than $p$, we have that $C \subseteq \omega_1$ is club and $S \cap C = \emptyset$ in $N[G]$ by our original assumption. However, $\alpha$ is the critical point of the embedding $\sigma^*$, so $\alpha$ is a limit ordinal and below $\alpha$, the club $C$ is fixed. Thus $\overline C = C \cap \alpha$ is unbounded in $\alpha$, in $N[G]$. Since $C$ is club and thus closed under limit points, this means $\alpha \in C$, a contradiction since $\alpha$ was taken to be in $S$.
\end{proof}
Thus, we may make sense of writing $\textsf{SubPFA}$, the subproper forcing axiom, and we have that Martin's Maximum (\textsf{MM}) implies \textsf{SubPFA}.

\section{Subcomplete Forcing}
\label{sec:SubcompleteForcing}
Subcomplete forcing is a class of forcing notions that we shall see do not add reals, but may potentially alter cofinalities to $\omega$. Examples of subcomplete forcing include Prikry forcing and Namba forcing (under $\CH$). This separates subcomplete forcings from proper forcings, that have countable covering and thus cannot change cofinalities to $\omega$. 

Before giving the definition of subompleteness, as with the transition from properness to subproperness, subcompleteness should be seen as a weakening of the class of complete (or countably closed forcings, as we shall see in \ref{subsec:SCForcingandCCForcing}). We follow \cite[Chapter 3]{Jensen:2012fr} for the following definitions and the proofs that we cite as due to Jensen.
 
\begin{defn} A forcing notion $\P$ is \emph{\textbf{complete}} so long as
for sufficiently large $\theta$ we have; letting: \begin{itemize}
	\item $\P \in H_\theta \subseteq N = L_\tau[A] \models \ZFC^-$ where $\tau>\theta$ and $A \subseteq \tau$
	\item $\sigma: \N \cong X \preccurlyeq N$ where $X$ is countable and $\N$ is transitive
	\item $\sigma(\overline \theta, \overline{\P}, \overline s)=\theta, \P, s$ for some $s \in N$;
\end{itemize}
if $\G$ is $\overline{\P}$-generic over $\N$ then there is $p \in \P$ forcing that whenever $G \ni p$ is $\P$-generic, $\sigma ``\, \G \subseteq G$. 

In particular, below the condition $p$ we have that $\sigma$ lifts to an embedding $\sigma^*:\N[\G] \prec N[G]$.
We say that such a $\theta$ as above \textit{\textbf{witnesses the completeness of $\P$}}.
\end{defn}
The adjustment made to get subcomplete forcings is to not necessarily insist the the original embedding lifts in the forcing extension. Instead subcompleteness asks for there to be an embedding in the extension, an embedding that is sufficiently similar to the original embedding, and lifts. However, as discussed earlier, the domain of the embedding must be \textit{full} to ensure that there can even consistently be more than one embedding. The definition is given below.

\begin{defn} \label{defn:SC}
A forcing notion $\P$ is \emph{\textbf{subcomplete}} so long as,
for sufficiently large $\theta$ we have that whenever we are in a situation where: \begin{itemize}
	\item $\P \in H_\theta \subseteq N = L_\tau[A] \models \ZFC^-$ where $\tau>\theta$ and $A \subseteq \tau$
	\item $\sigma: \N \cong X \preccurlyeq N$ where $X$ is countable and $\N$ is full
	\item $\sigma(\overline \theta, \overline{\P}, \overline s)=\theta, \P, s$ for some $s \in N$;
\end{itemize}
then we have that if $\G$ is  $\overline{\P}$-generic over $\N$ then there is $p \in \P$ such that whenever $G \ni p$ is $\P$-generic, there is $\sigma' \in V[G]$ satisfying: \begin{enumerate}
	\item $\sigma': \N \prec N$
	\item $\sigma'(\overline \theta, \overline{\P}, \overline s)=\theta, \P, s$
	\item $\sk{N}{\delta(\P)}{\sigma'} = \Sk{N}{\delta(\P)}{X}$
	\item $\sigma'``\, \G \subseteq G$.
\end{enumerate}
In other words and in particular, the condition $p$ forces that there is an embedding $\sigma'$ in the extension $V[G]$ that lifts, by \textbf{4}, to an embedding $\sigma'^*:\N[\G] \prec N[G]$ in $V[G]$.

We say that such a $\theta$ as above \textit{\textbf{verifies the subcompleteness}} of $\P$.

Often we write $\delta$ instead of $\delta(\P)$ when there should be no confusion as to which poset $\P$ we are working with.\footnote{See Section \ref{sec:delta} for more on $\delta(\P)$.}
\end{defn}
Condition \textbf{3} ensures that subcomplete forcings may be iterated, and is rarely used in the following proofs. Thus we will tend to leave out of the discussion if we are showing properties of subcomplete posets.

Immediately we obtain the following.
\begin{prop}[Jensen]\label{prop:noreals}
Subcomplete forcing does not add countable subsets of countable sets. In particular, subcomplete forcing does not add reals.
\end{prop}
\begin{proof}
Let $\P$ be subcomplete. Suppose toward a contradiction that $\P$ adds a new real, let $p \in \P$ force that $\dot r: \check \omega \rightarrow \omega$ is new. Let $\theta$ be large enough and let $\sigma$, $\N$, and $N$ satisfy: \begin{itemize} 
	\item $\P \in H_\theta \subseteq N = L_\tau[A] \models \ZFC^-$ where $\tau>\theta$ and $A \subseteq \tau$
	\item $\sigma: \N \cong X \preccurlyeq N$ where $X$ is countable and $\overline N$ is full
	\item $\sigma(\overline \theta, \overline{\P}, \overline p, \overline{\dot r})=\theta, \P, p, \dot r.$
\end{itemize}
Let $\G \subseteq \overline{\P}$ be an $\N$-generic filter with $\overline p \in \G$. By the subcompleteness of $\P$, there is $q \in \P$ forcing that whenever $G \ni q$ is $\P$-generic, there is $\sigma' \in V[G]$ such that \begin{itemize}
	\item $\sigma': \N \prec N$
	\item $\sigma'(\overline \theta, \overline{\P}, \overline p, \overline{\dot r})=\theta, \P, p, \dot r$
	\item $\sigma'``\, \G \subseteq G$.
\end{itemize} 
Thus $\sigma'$ lifts to $\sigma':\N[\G] \prec N[G]$ in $V[G]$. Let $r = \dot r^G$, $\overline r = \overline{\dot r}^{\G}$. Then for each $n < \omega$ we have that 
	$$r(n) = \sigma'(\overline r(n)) = \overline r(n).$$ 
So $r=\overline r \in V$, contradicting our original assumption that $p \in G$ forces $\dot r$ to be new. 
\end{proof}

\begin{prop}
Subcomplete forcings are subproper.
\end{prop}
\begin{proof}
Let $\delta(\P)=\delta$. If $\P$ is subcomplete, then to see that $\P$ is also subproper assume we are in the standard setup from the definition of subproperness: \begin{itemize}
	\item $\P \in H_\theta \subseteq N = L_\tau[A] \models \ZFC^-$ where $\tau>\theta$ and $A \subseteq \tau$
	\item $\sigma: \N \cong X \preccurlyeq N$ where $X$ is countable and $\overline N$ is full
	\item $\sigma(\overline \theta, \overline{\P}, \overline s)=\theta, \P, s$ for some $s \in N$.
\end{itemize}
Let $\overline q \in \overline{\P}$. Then $q = \sigma(\overline q) \in X$.
Let $\G$ be generic for $\overline{\P}$ over $\N$ such that $\overline q \in \G$. Then by subcompleteness of $\P$, we have that there is $p \in \P$ such that whenever $G$ is generic and $G$ contains $p$, then there is $\sigma' \in V[G]$ satisfying: \begin{enumerate}
	\item $\sigma': \N \prec N$
	\item $\sigma'(\overline \theta, \overline{\P}, \overline s, \overline q)=\theta, \P, s, q$
	\item $\sk{N}{\delta}{\sigma'} = \Sk{N}{\delta}{X}$
	\item $\sigma'``\, \G \subseteq G$.
\end{enumerate}
In particular, 
	$$q = \sigma(\overline q)=\sigma'(\overline q) \in G.$$ Thus $q$ is compatible with $p$, and so there is $r \in G$ satisfying $r \leq q$ and $r \leq p$. But then this $r$ verifies the subproperness of $\P$, since this $r$ is also below $p$.
\end{proof}
By \textbf{Theorem \ref{thm:SubproperMM}}, this means that subcomplete forcings also preserve stationary subsets of $\omega_1$. Thus the class of subcomplete forcing is not the same as the class of countably distributive forcing notions (those that don't add new reals). The following lemma tells us that the verification of subcompleteness is absolute to suitably large structures.

\begin{lem}[Jensen] \label{lem:scabsolute} Suppose that $\P \in H_\theta$ and let $\eta>|H_\theta|$. 

$\P$ is subcomplete as verified by $\theta$ $\iff$ $\P$ is subcomplete in $H_\eta$ as verified by $\theta$.  \end{lem} 
\begin{proof}
If $\P$ is subcomplete in $H_\eta$ then $\P$ is subcomplete in $V$ because in the definition of subcompleteness, only $N=L_\tau[A] \in H_\eta$ need to be considered. To see why, assume that $\P$ is subcomplete as verified by $\theta$ in $H_\eta$.  Then we show that $\P$ actually is subcomplete as verified by $\theta$.
\begin{claim} Whenever we are in a situation where \begin{itemize}
	\item $\P \in H_\theta \subseteq N = L_\tau[A] \models \ZFC^-$ where $\tau>\theta$ and $A \subseteq \tau$
	\item $\sigma: \N \cong X \preccurlyeq N$ where $X$ is countable and $\N$ is full
	\item $\sigma(\overline \theta, \overline{\P}, \overline s)=\theta, \P, s$ for some $s \in N$;
\end{itemize}
then we have that if $\G$ is  $\overline{\P}$-generic over $\N$ then there is $p \in \P$ such that whenever $G \ni p$ is $\P$-generic, there is $\sigma' \in V[G]$ satisfying: \begin{enumerate}
	\item $\sigma': \N \prec N$
	\item $\sigma'(\overline \theta, \overline{\P}, \overline s)=\theta, \P, s$
	\item $\sk{N}{\delta}{\sigma'} = \Sk{N}{\delta}{X}$
	\item $\sigma'`` \G \subseteq G$.
\end{enumerate} \end{claim}
\begin{proof}[Pf.]
We find a suitable embedding, similar to $\sigma$, but in $H_\eta$. Let $ \pi: N_0 \cong \Sk{N}{H_\theta}{X \cup \delta \cup \{s\}} =Y \preccurlyeq N$. Then $N_0 \in H_\eta$. Since $X \subseteq Y$ and $X \preccurlyeq N$, we have that $X \preccurlyeq Y$. Thus, letting $X_0=\pi^{-1}``X$, we have that $X \cong X_0 \preccurlyeq N_0$. So taking the Mostowski collapse $\N_0$, we have that $\N_0=\N$. Thus we have:
\begin{itemize}
	\item $\sigma_0:\N \cong X_0 \preccurlyeq N_0$ in $H_\eta$ 
	\item $\sigma_0(\overline \theta, \overline{\P}, \overline s)=\sigma(\overline \theta, \overline{\P}, \overline s)=\theta, \P, s$. 
	\end{itemize}
So by the subcompleteness of $\P$ in $H_\eta$, we have that if $\G \subseteq \overline{\P}$ is generic over $\N$ then there is $p \in \P$ such that whenever $G_0 \ni p$ is $\P$-generic over $H_\eta$, there is $\sigma_0'\in H_\eta[G_0]=V^{H_\eta}[G_0]$ satisfying \begin{enumerate}
	\item $\sigma_0': \N \prec N_0$
	\item $\sigma_0'(\overline \theta, \overline{\P}, \overline s)=\theta, \P, s$
	\item $\sk{N_0}{\delta}{\sigma_0'} = \Sk{N_0}{\delta}{X_0}$
	\item $\sigma_0'`` \G \subseteq G_0$.
\end{enumerate}
If $G \ni p$ is $\P$-generic over $V$, then $G$ is also $\P$-generic over $H_\eta$. We thus have that items \textsl{\textbf{1}} and \textsl{\textbf{2}} of the \textit{Claim} hold, since letting $\sigma'=\sigma_0' \circ \pi \in V[G]$ we have $\sigma':\N \prec N$ and $\sigma'(\overline \theta, \overline{\P}, \overline s)=\theta, \P, s$.  Moreover item \textsl{\textbf{3}} holds since $\pi``\sk{N_0}{\delta}{\sigma_0'} = \sk{N}{\delta}{\sigma'}$ and $\pi``\Sk{N_0}{\delta}{X_0} = \Sk{N}{\delta}{X}$. For property \textbf{\textsl{4}}, the point is that $\pi \rest H_\theta = \id$, so $\pi \circ \sigma_0' `` \G = \sigma_0'``\G \subseteq G$. This finishes the proof of the \textit{Claim}.
 \end{proof}

Finally, if $\P$ is subcomplete in $V$ as verified by $H_\theta$, then $\P$ is subcomplete in $H_\eta$ since given $\sigma \in H_\eta$ where $\sigma: \N \prec N$, we have of course that $\N, N \in H_\eta$. Additionally $\sigma \in V$, and we may apply subcompleteness in $V$ to obtain a condition $p \in \P$ such that whenever $p \in G$, there is a $\sigma' \in V[G]$ having certain properties, namely \textbf{1} through \textbf{4} of the definition. Such generic filters over $V$ are also generic over $H_\eta$, and furthermore $\sigma' \in H_\eta[G] = V^{H_\eta}[G]$ since $\N$ and $N$ are in $H_\eta$. So the rest of the properties are also clear, and $\P$ is subcomolete in $H_\eta$ as desired.
\end{proof}

\begin{remark}\label{remark:VerifyingSC} If $\P$ is subcomplete as verified by $\theta$, then $\P$ is subcomplete as verifed by $\theta' > \theta$ since $\P \in H_\theta$ obviously implies that $\P \in H_{\theta'}$. Thus if we are in a situation where $\P \in H_{\theta'} \subseteq N=L_\tau[A]$ for $\tau>\theta'$ then we may reduce down to the case where $\P \in H_{\theta} \subseteq N=L_\tau[A]$ and $\tau>\theta'>\theta$ using the same $\tau$ and apply subcompletenss. In particular, $\P$ is subcomplete as long as it is verified to be subcomplete by some $\theta$. So we may replace ``sufficiently large $\theta$" with ``some $\theta$" in the definition of subcompleteness.\footnote{See \cite[Section 3.1 Lemma 2.4]{Jensen:2012fr}.} \end{remark}

In addition we have that subcomplete forcing is closed under lottery sums, which is also true for subproper forcing. But since subcomplete forcing is the focus, we will only give the proof for subcomplete forcing. 

\begin{defn} For a family of forcing notions $\mathcal P$, the \textit{\textbf{lottery sum}} poset is defined as follows: 
	$$\oplus \mathcal P=\{ \mathbbm 1 \} \cup \set{ \langle \P, p \rangle }{ \P \in \mathcal P \ \land \ p \in \P }$$
with $\mathbbm 1$ weaker than everything and $\langle \P, p \rangle \leq \langle \P',p' \rangle$ if and only if $\P = \P'$ and $p \leq_\P p'$. \end{defn}

\begin{lem}\label{lem:sclottery} 
Lottery sums of subcomplete forcings are subcomplete.
\end{lem}
\begin{proof}
Let $\mathbb P$ be the lottery sum of $\mathcal Q = \set{\Q_\alpha}{\alpha < \kappa}$, where each $\Q_\alpha$ is subcomplete. Let $\theta$ be large enough to verify the subcompleteness of each $\Q_\alpha$. \begin{itemize}
	\item $\P \in H_\theta \subseteq N = L_\tau[A] \models \ZFC^-$ where $\tau>\theta$ and $A \subseteq \tau$
	\item $\sigma: \N \cong X \preccurlyeq N$ where $X$ is countable and $\N$ is full
	\item $\sigma(\overline \theta, \overline{\P}, \overline{\mathcal Q}, \overline s)=\theta, \P, \mathcal Q, s$ for some $s \in N$.
\end{itemize}
Let $\G$ be $\overline{\P}$-generic over $\N$. Note that if $\overline p \in \overline{\Q}_\alpha$, $\overline q \in \overline{\Q}_\beta$, $\alpha \neq \beta$, then $\langle \overline{\Q}_\alpha, \overline p \rangle$ and $\langle \overline{\Q}_\alpha, \overline q \rangle$ are incompatible. Since elements of $\G$ are pairwise compatible, it therefore must be the case that $\G \subseteq \{\mathbbm 1\} \cup (\{\overline{\Q}_\alpha\} \times \overline{\Q}_\alpha)$ for some $\Q_\alpha$ in $\mathcal Q$.

Since each poset in $\mathcal Q$ is subcomplete, there is $p \in \Q_\alpha$ such that whenever $G \ni p$ is $\Q_\alpha$-generic, there is $\sigma' \in V[G]$: \begin{enumerate}
	\item $\sigma': \N \prec N$
	\item $\sigma'(\overline \theta, \overline{\Q}_\alpha, \overline{\P}, \overline{\mathcal Q}, \overline s)=\theta, \Q_\alpha, \P, \mathcal Q, s$
	\item $\sk{N}{\delta(\Q_\alpha)}{\sigma'} = \Sk{N}{\delta(\Q_\alpha)}{X}$
	\item $\sigma'`` \, \G \subseteq G$.
\end{enumerate}

Now suppose that $p \in G \subseteq \P$ is generic over $V$. Then $G \subseteq \{\mathbbm 1\} \cup (\{\Q_\alpha\}\times\Q_\alpha)$, by the same argument as above and as $p \in \Q_\alpha$. Now all that's left to show is that $\sigma' \in V[G]$ satisfies 
	\begin{equation}\label{eqn-SkSigma'=SkX} \sk{N}{\delta(\P)}{\sigma'}= \Sk{N}{\delta(\P)}{X}.\end{equation}
To see this, note that $\delta(\Q_\alpha) \leq \delta(\P)$, since any dense set in $\P$ of smallest size will have to contain a dense set of the smallest cardinality in $\Q_\alpha$. So by item \textbf{3} above and by $\textbf{Lemma \ref{lem:hullequality}}$, we have that (\ref{eqn-SkSigma'=SkX}) holds as well.
\end{proof}

\begin{defn} Two posets $\P$ and $\Q$ are said to be \textit{\textbf{forcing equivalent}} if whenever $G \subseteq \P$ is generic, there is $H \subseteq \Q$ generic such that $V[G] = V[H]$, and similarly if $H \subseteq \Q$ is generic, there is $G \subseteq \P$ generic such that $V[G]=V[H]$. \end{defn}

For example, if $\P$ densely embeds into $\Q$, then $\P$ and $\Q$ are forcing equivalent. 

\begin{lem} \label{lem:deltalottery}
If $\kappa> \delta(\P)$, then we may find a forcing equivalent poset $\P_\kappa$ such that $\delta(\P_\kappa) = \kappa$. 
\end{lem}
\begin{proof}
Let $\P_\kappa = \oplus_\kappa \P$ be the $\kappa$-sized lottery sum of $\P$ with itself. The lottery sum of a forcing with itself is forcing equivalent to that forcing, but artificially will have larger dense sets. 
\end{proof} 


Being subcomplete does not appear to always be closed under forcing equivalence. However, we have the following.
\begin{prop}\label{prop:scde} If $\P$ is subcomplete and $\P$ densely embeds into $\Q$, then $\Q$ is subcomplete. \end{prop}
\begin{proof}
Let $\pi:\P \longrightarrow \Q$ be a dense embedding. We have that $\delta(\P)=\delta(\Q) = \delta$, by \textbf{Lemma \ref{lem:deltasize}}.
If $\P$ is subcomplete, then to see that $\Q$ is also subcomplete assume $\theta$ is large enough to verify the subcompleteness of $\P$ and that we are in the standard setup: \begin{itemize}
	\item $\P, \Q \in H_\theta \subseteq N = L_\tau[A] \models \ZFC^-$ where $\tau>\theta$ and $A \subseteq \tau$
	\item $\sigma: \N \cong X \preccurlyeq N$ where $X$ is countable and $\N$ is full
	\item $\sigma(\overline \theta, \overline \P, \overline \Q, \overline \pi, \overline s)=\theta, \P, \Q, \pi, s$ for some $s \in N$.
\end{itemize} 
We have by elementarity that $\overline \pi:\overline \P \longrightarrow \overline \Q$ is a dense embedding. 
Toward showing that $\Q$ is subcomplete, let $\overline H \subseteq \overline \Q$ be generic for $\overline \Q$ over $\N$. 
\begin{claim} There is $q \in \Q$ such that whenever $H \subseteq \Q$ is generic and $q \in H$, then there is $\sigma' \in V[H]$ satisfying: \begin{enumerate}
	\item $\sigma': \N \prec N$
	\item $\sigma'(\overline \theta, \overline{\Q}, \overline s)=\theta, \Q, s$
	\item $\sk{N}{\delta}{\sigma'} = \Sk{N}{\delta}{X}$
	\item $\sigma'`` \overline H \subseteq H$.
\end{enumerate}
\end{claim}
\begin{proof}[Pf.]
Let $\G =\overline \pi^{\ -1} `` \overline H$. Then $\G \subseteq \overline \P$ is generic over $\N$.
So by the subcompleteness of $\P$ there is $p \in \P$ such that whenever $G \subseteq \P$ is generic satisfying $p\in G$ there is an embedding $\sigma' \in V[G]$ satisfying: \begin{itemize}
	\item $\sigma': \N \prec N$
	\item $\sigma'(\overline \theta, \overline{\P}, \overline{\Q}, \overline \pi, \overline s)= \theta, \P, \Q, \pi,  s$
	\item $\sk{N}{\delta}{\sigma'} = \Sk{N}{\delta}{X}$
	\item $\sigma'`` \, \G \subseteq G$.
\end{itemize}
Indeed this is the $\sigma'$ that is required. Let $q=\pi(p)$. Then for $H=\pi``G$ we have that for $\overline q \in \overline H$, $\overline q = \overline \pi(\overline p)$ for some $\overline p \in \G$. But then 
	$$\sigma'(\overline q) = \sigma'(\overline \pi(\overline p)) \in \pi``G=H.$$ 
Since all of the other properties needed to satisfy the \textit{Claim} are true of $\sigma'$, we are done.	
\end{proof} 
Thus $\Q$ is subcomplete as desired.
\end{proof}

\subsection{Subcomplete Forcing and Countably Closed Forcing}
\label{subsec:SCForcingandCCForcing}
In order to explore the relationship between countably closed, complete, and subcomplete forcing, let's define the following, seemingly weaker class of forcings than given by completeness, where our countable transitive collapses are taken to be full as in the definition of subcompleteness.
\begin{defn} 
A forcing notion $\P$ is \emph{\textbf{weakly complete}} so long as
for sufficiently large $\theta$ we have; letting: \begin{itemize}
	\item $\P \in H_\theta \subseteq N = L_\tau[A] \models \ZFC^-$ where $\tau>\theta$ and $A \subseteq \tau$
	\item $\sigma: \N \cong X \preccurlyeq N$ where $X$ is countable and $\N$ is full
	\item $\sigma(\overline \theta, \overline{\P}, \overline s)=\theta, \P, s$ for some $s \in N$;
\end{itemize}
if $\G$ is $\overline{\P}$-generic over $\N$ then there is $p \in \P$ such that whenever $G \ni p$ is $\P$-generic, $\sigma ``\, \G \subseteq G$. 

In particular, below $p$ we have that $\sigma$ lifts to an embedding $\sigma^*:\N[\G] \prec N[G]$.
We say that $\theta$ as above \textit{\textbf{verifies the weak completeness of $\P$}}. As with subcomplete forcing (see \textit{Remark} \ref{remark:VerifyingSC}), we may replace ``sufficiently large $\theta$" with ``some $\theta$".
\end{defn}

\begin{lem}\label{lem:BA(w)c} If $\P$ is (weakly) complete, then $\BA(\P)$ is (weakly) complete. \end{lem}
\begin{proof}
Let $i:\P \To \BA(\P)$ be the canonical dense embedding from $\P$ to its Boolean algebra. It is not hard to see that the proof of \textbf{Proposition \ref{prop:scde}} gives us that if $\P$ is (weakly) complete then $\BA(\P)$ is (weakly) complete: instead of working with the lifted $\sigma'$ we may work with $\sigma$ by (weak) completeness and show the same properties hold, as desired.
\end{proof}

In fact, weakly complete forcing and complete forcing give rise to the same class: that of countably closed forcing.
Our below proofs directly follow those of Jensen \cite[Ch.~2 p.~3]{Jensen:2009wc}.
\begin{thm}[Jensen] The following give rise to the same classes of forcing notions, up to forcing equivalence:
\begin{enumerate}
	\item Countably closed
	\item Complete
	\item Weakly complete
\end{enumerate}
\end{thm}	
\begin{proof}
First we show that countably closed forcings are complete. Then we will show that both complete forcings and weakly complete forcings are forcing equivalent to countably closed posets.

\begin{paragraph}{If $\P$ is countably closed, then $\P$ is complete (and weakly complete):}
\begin{proof}[Pf.]
Below we show that if $\P$ is countably closed, then $\P$ is complete, and the same proof shows that $\P$ is weakly complete with the modifications in parentheses. Let $\theta$ be sufficiently large so that: \begin{itemize}
	\item $\P \in H_\theta \subseteq N = L_\tau[A] \models \ZFC^-$ where $\tau>\theta$ and $A \subseteq \tau$
	\item $\sigma: \N \cong X \preccurlyeq N$ where $X$ is countable and $\N$ is transitive (full)
	\item $\sigma(\overline \theta, \overline{\P}, \overline s)=\theta, \P, s$ for some $s \in N$.
\end{itemize}
Let $\G \subseteq \overline{\P}$ be generic over $\N$. Since $\overline{\P}$ is countable, we may generate $\G$ with a chain, $\seq{ \overline{p}_n }{ n < \omega }$. So use the countable closedness of $\P$ to take $p \in \P$ satisfying 
	$\forall n \ p \leq \sigma(\overline p_n).$
Then $\P$ is complete (weakly complete) as desired, since any generic $G$ containing $p$ must contain $\sigma``\,\G$ as the sequence of $\overline p_n$'s generate $\G$.
\end{proof}
\end{paragraph}

\begin{paragraph}{If $\P$ is complete, then $\P$ is weakly complete:}
\begin{proof}[Pf.]
Suppose that $\P$ is complete. To show that $\P$ is weakly complete, we need to show that if we are in a situation where \begin{itemize}
	\item $\P \in H_\theta \subseteq N = L_\tau[A] \models \ZFC^-$ where $\tau>\theta$ and $A \subseteq \tau$
	\item $\sigma: \N \cong X \preccurlyeq N$ where $X$ is countable and $\N$ is full
\end{itemize}
then there is $p \in \P$ such that below $p$, in the extension $\sigma$ lifts. But this is already true because if $\N$ is full then $\N$ is transitive, so the rest goes through as $\P$ is complete.
\end{proof}
\end{paragraph}

\begin{paragraph}{If $\P$ is complete, then $\P$ is forcing equivalent to a countably closed poset:}
\begin{proof}[Pf.]
Suppose that $\P$ is complete. Then its Boolean algebra $\B=\BA(\P)$ is also complete by \textbf{Lemma \ref{lem:BA(w)c}}. We show that $\B$ is forcing equivalent to a countably closed poset. Let $\theta$ verify the completeness of $\B$. Define a poset $\Q$ consisting of conditions $q = \langle X_q, G_q \rangle$ such that $\B, \theta \in X_q \preccurlyeq N$ where $X_q$ is countable and $G_q \subseteq \B$ is $X_q$-generic. Let 
	$$q \leq r \iff X_q \supseteq X_r \ \text{and} \ G_r = G_q \cap X_r.$$ 
We may immediately see that $\Q$ is countably closed since each $X_q$ is countable and the generics cohere. Define a dense homomorphism $\pi: \Q \to \B$ from $\Q$ to $\B$ as follows: $\pi(q) = \Meet G_q$ (the meet of all elements of $G_q$). We show that $\pi$ is as desired below:
	\begin{paragraph}{$\pi$ is well defined:} Let $q \in \Q$. That $\Meet G_q \neq 0$, where $0$ is the bottom element, follows from the completeness of $\B$. Let $\G_q$ be the isomorphic copy of $G_q$, $\sigma^{-1}``G_q$ in $\N$, where $\N$ is the Mostowski collapse of $X_q$. The completeness of $\B$ guarantees that there is $b \in \B$ such that whenever $G \ni b$ is $\B$-generic, $\sigma``\G=G_q \subseteq G$. This implies that $b \leq \Meet G_q$. (Otherwise $b \meet \neg \Meet G_q$ is in some generic $G^*$, contradicting $G_q\subseteq G$.) \end{paragraph}
	\begin{paragraph}{$q \leq r \implies \Meet G_q \leq \Meet G_r$:} If $q \leq r$ then $G_r \subseteq G_q$, so the desired result follows. \end{paragraph}
	\begin{paragraph}{$q \parallel r \iff \Meet G_q \meet \Meet G_r \neq 0$:} For the forward direction, if $s \leq q, r$ then $\Meet G_s \leq \Meet G_q \meet \Meet G_r$ since $G_s \supseteq G_q,G_r$. For the other direction, let $\Meet G_q \meet \Meet G_r \neq 0$. Then let $X \preccurlyeq N$ with $X$ countable, $X_q \cup X_r \subseteq X$, and $\Meet G_q \meet \Meet G_r \in X$. Since $X$ is countable we may obtain $G \subseteq \P$ that is $X$-generic, with $\Meet G_q \meet \Meet G_r \in G$. Then $\langle X,G \rangle \leq q,r$. \end{paragraph}
	\begin{paragraph}{$D=\set{ \cap G_q }{ q \in \Q }$ is dense in $\P$:} Let $b \in \B$. Let $X \preccurlyeq N$ with $X$ countable and $b \in X$. Then if $G$ is $X$-generic and contains $b$, we have $b \geq \Meet G$ and $\Meet G \in D$. \qedhere
\end{paragraph}
\end{proof} 
\end{paragraph}

\begin{paragraph}{If $\P$ is weakly complete, then $\P$ is forcing equivalent to a countably closed poset:}
\begin{proof}[Pf.]
Suppose that $\P$ is weakly complete. Then its Boolean algebra $\B=\BA(\P)$ is also complete by \textbf{Lemma \ref{lem:BA(w)c}}. We show that $\B$ is forcing equivalent to a countably closed poset. Let $\theta$ be large enough to verify the weak completeness of $\P$. Define a poset $\Q$ consisting of conditions $q = \langle X_q, G_q \rangle$ such that there is a $Y_q \preccurlyeq L_{\tau^+}[A]$ countable where $X_q = Y_q \cap N$ and $G_q \subseteq \B$ is $X_q$-generic, where $\B \in H_\theta \subseteq N = L_\tau[A]$ with $\tau >\theta$ and $A \subseteq \tau$.
Let 
	$$q \leq r \iff X_q \supseteq X_r \ \text{and} \ G_r = G_q \cap X_r.$$ 
We've guaranteed with our definition that if $\sigma: \N \cong X_q$ is the Mostowski collapse, then $\N$ is full, since then $\N = (L_{\overline \tau}[\overline A])^{L_{\overline{\tau^+}}[\overline A]}$ where $L_{\overline{\tau^+}}[\overline A]$ is the Mostowski collapse of $Y$, and we may reason as in the proof of \textbf{Lemma \ref{lem:fullclubs}}.


We wish to show that $\Q$ is countably closed. Define a dense homomorphism $\pi: \Q \to \BA(\P)$ as follows: $\pi(q) = \Meet G_q$ (the meet of all of the elements of $G_q$). Now we show that $\pi$ is as desired:

\begin{paragraph}{$\pi$ is well defined:} Let $q \in \Q$. That $\Meet G_q \neq 0$, where $0$ is the bottom element, follows from the weak completeness of $\B$. Let $\G_q$ be the isomorphic copy of $G_q$, $\sigma^{-1}``G_q$ in $\N$, where $\N$ is the Mostowski collapse via $\sigma$ of $X_q$ as above. The weak completeness guarantees that there is $b \in \B$ such that whenever $G \ni b$ is $\B$-generic, $\sigma``\G=G_q \subseteq G$. This implies that $b \leq \Meet G_q$. \end{paragraph}
	\begin{paragraph}{$q \leq r \implies \Meet G_q \leq \Meet G_r$:} If $q \leq r$ then $G_r \subseteq G_q$, so the desired result follows. \end{paragraph}
	\begin{paragraph}{$q \parallel r \iff \Meet G_q \meet \Meet G_r \neq 0$:} For the forward direction, if $s \leq q, r$ then $\Meet G_s \leq \Meet G_q \meet \Meet G_r$ since $G_s \supseteq G_q,G_r$. For the other direction, let $\Meet G_q \meet \Meet G_r \neq 0$. Then let $Y \preccurlyeq L_{\tau^+}[A]$ with $Y$ countable, $Y_q \cup Y_r \subseteq Y$, and $\Meet G_q \meet \Meet G_r \in Y$. Then letting $X=Y\cap N$ we may construct $G \subseteq \P$ that is $X$-generic, with $\Meet G_q \meet \Meet G_r \in G$. Then $\langle X,G \rangle \leq q,r$. \end{paragraph}
	\begin{paragraph}{$D=\set{ \Meet G_q }{ q \in \Q }$ is dense in $\P$:} Let $b \in \B$. Let $X \preccurlyeq N$ with $X$ countable and $b \in X$. Then if $G$ is $X$-generic and contains $b$, we have $b \geq \Meet G$ and $\Meet G \in D$. \qedhere
	\end{paragraph}
\end{proof}
\end{paragraph}
\end{proof}

\begin{thm}\label{thm:smallsc} If $\P$ is subcomplete and $|\P| = \omega_1$, then $\P$ is weakly complete, and is thus equivalent to a countably closed poset.
\end{thm} 
\begin{proof} 
We may assume that $\P \subseteq \omega_1$, and note that $\delta(\P) \leq \omega_1$. Let $\theta$ verify the subcompleteness of $\P$, and suppose we are in the standard setup: \begin{itemize}
	\item $\P \in H_\theta \subseteq N = L_\tau[A] \models \ZFC^-$ where $\tau>\theta$ and $A \subseteq \tau$
	\item $\sigma: \N \cong X \preccurlyeq N$ where $X$ is countable and $\N$ is full
	\item $\sigma(\overline \theta, \overline{\P},  \overline s)=\theta, \P, s$ for some $s \in N$.
\end{itemize}
Toward showing that $\P$ is complete, let $\G$ be $\overline{\P}$-generic over $\N$. 
By subcompleteness, there is $p \in \P$ so that whenever $G \ni p$ is $\P$-generic, there is $\sigma' \in V[G]$ such that: \begin{itemize}
	\item $\sigma': \N \prec N$
	\item $\sigma'(\overline \theta, \overline{\P}, \overline s)=\theta, \P, s$
	\item $\sk{N}{\omega_1}{\sigma'} = \Sk{N}{\omega_1}{X}$
	\item $\sigma'``\, \G \subseteq G$.
\end{itemize}
Of course $\G \subseteq \overline{\P}$ and thus $|\G| = \omega_1^{\N}$. However, because $\cp(\sigma) = \cp(\sigma') = \omega_1^{\N}$ we have that $\sigma``\,\G = \sigma'``\,\G$ so below $p$ we have $\sigma``\,\G \subseteq G$ as desired.
\end{proof}

\subsection{Levels of Subcompleteness}
\label{subsec:LevelsofSubcompleteness}
The notion of subcompleteness above $\mu$ is an attempt to measure where exactly subcompleteness kicks in; in some sense it tells you at what level the forcing fails to be complete. The following definition is from Jensen \cite[Chapter 2 p.\ 47]{Jensen:2009fe}.
\begin{defn} 
Let $\mu$ be a cardinal. We say that a forcing notion $\P$ is \emph{\textbf{subcomplete above $\mu$}} so long as for sufficiently large $\theta > \mu$, whenever we are in the standard setup, where: \begin{itemize}
	\item $\P \in H_\theta \subseteq N = L_\tau[A] \models \ZFC^-$ with $\tau>\theta$ and $A \subseteq \tau$
	\item $\sigma: \N \cong X \preccurlyeq N$ where $X$ is countable and and $\N$ is full
	\item $\sigma(\overline \theta, \overline \mu, \overline{\P}, \overline s)=\theta, \mu, \P, s$ for some $s \in N$;
\end{itemize}
then, for any $\G \subseteq \overline{\P}$, there is $p \in \P$ such that whenever $G \ni p$ is $\P$-generic, then there is $\sigma' \in V[G]$ satisfying: \begin{enumerate}
	\item $\sigma': \N \prec N$
	\item $\sigma'(\overline \theta, \overline{\P}, \overline s)=\theta, \P, s$
	\item $\sk{N}{\delta}{\sigma'} = \Sk{N}{\delta}{X}$
	\item $\sigma'`` \, \G \subseteq G$
	\item $\sigma' \rest H_{\overline \mu}^{\N} = \sigma \rest H_{\overline \mu}^{\N}$.
\end{enumerate}
As usual, this means that in particular below the condition $p$ there is an embedding $\sigma'$ that lifts by \textbf{4} to an embedding $\sigma'^*:\N[\G] \prec N[G]$ in $V[G]$. As usual we say that such $\theta$ as above \textbf{\textit{verifies the subcompleteness above $\mu$ of $\P$}}, and we may say that $\P$ is subcomplete if there is a $\theta$ that verifies its subcompleteness above $\mu$.
\end{defn}

Immediately we have the following:
\begin{prop} If $\P$ is subcomplete, then $\P$ is subcomplete above $\omega_1^{\N}$. \end{prop}
\begin{proof} By \textit{Fact} \ref{fact:CPofourEmbeddings}, we have that for any elementary embedding such as $\sigma$ or $\sigma'$ from countable $\N$ to $N$, $\sigma' \rest H_{\overline{\omega_1}}^{\N} = \sigma \rest H_{\overline{\omega_1}}^{\N} = \text{id}$. \end{proof}

\begin{thm} If $\P$ is subcomplete above $\mu$ then $\P$ does not add new countable subsets of $\mu$. \end{thm}
\begin{proof}
Suppose not. Let $\P$ be subcomplete above $\mu$. Let $\dot B \subseteq \mu$ be countable with 
	$$p \forces \dot f: \check \omega \cong \dot B \ \land\ \dot f \notin V.$$ 
Take $\theta > \mu$ large enough to verify the subcompleteness of $\P$. \begin{itemize}
	\item $\P \in H_\theta \subseteq N = L_\tau[A] \models \ZFC^-$ where $\tau>\theta$ and $A \subseteq \tau$
	\item $\sigma: \N \cong X \preccurlyeq N$ where $\N$ is countable and full
	\item $\sigma(\overline \theta, \overline \mu, \overline{\P}, \overline p, \overline{\dot B}, \overline{\dot f})=\theta, \mu, \P, p, \dot B, \dot f$.
\end{itemize}
Let $\G$ be $\overline{\P}$-generic over $\N$ with $\overline p \in \overline G$. By the subcompleteness of $\P$ above $\mu$ there is $p \in \P$ such that whenever $p \in G$ where $G$ is $\P$-generic, there is $\sigma' \in V[G]$ satisfying: \begin{enumerate}
	\item $\sigma': \N \prec N$
	\item $\sigma'(\overline \theta, \overline \mu, \overline{\P}, \overline p, \overline{\dot B}, \overline{\dot f})=\theta, \mu, \P, p, \dot B, \dot f$
	\item $\sk{N}{\delta}{\sigma'} = \Sk{N}{\delta}{X}$
	\item $\sigma'``\G \subseteq G$
	\item $\sigma' \rest \overline \mu = \sigma \rest \overline \mu$.
\end{enumerate}
Let $\overline f = \overline{\dot f}^{\G}$ and $f = \dot f^G$. By \textbf{4} and \textbf{5}, for each $n$, $$f(n)=\sigma'(\overline f(n)) = \sigma(\overline f(n)),$$ meaning that $f \in V$ since $\sigma \in V$, which is a contradiction.
\end{proof}
Thus if $\P$ is subcomplete above $\mu$, then $\mu$'s cardinality, and even its cofinality, cannot be altered to be $\omega$ via $\P$.

\chapter{Properties of Subcomplete Forcing}
\label{chap:PropertiesofSCForcing}
We've seen that countably closed forcings are subcomplete. Countably closed forcings have many nice properties: they preserve branches through $\omega_1$-trees, or any particular $\lozenge$-sequence. Both may be framed as as $\Pi_1^1$-statements, using the sequence or the tree as predicate. In Section \ref{sec:SCForcingAndTrees} we explore questions related to the interactions between subcomplete forcing and various properties of trees of height $\omega_1$. 
Section \ref{sec:SCForcingAndTrees} is inspired by the work of Fuchs \cite{Fuchs:2008rt}, where it is shown that the maximality principle for closed forcings (an axiom that we will define in Section \ref{sec:MP}) implies countably closed-generic $\mathbf{\Sigma}_2^1(H_{\omega_1})$-absoluteness (defined below). In fact, Fuchs defines the generic absoluteness notion for $<\!\kappa$-closed forcing, where $\kappa$ is regular. Since the idea was to look at whether there are analogous results for subcomplete forcing, we define the more general notion, for reasonable forcing classes (what is meant by reasonable is described in the beginning of Chapter \ref{chap:AxiomsaboutSC}), below.

\begin{defn} \label{def:GenericAbsoluteness} Let $n,m$ be natural numbers, let $\Gamma$ be a class of forcing notions, and let $M$ be a transitive set (usually either $\omega_1$ or $H_{\omega_1}$). Then \textbf{\emph{$\Gamma$-generic $\Sigma_n^m(M)$-absoluteness}} with parameters in $S \subseteq \mathcal P(M)$ is the statement that for any $\Sigma_n^m$-sentence $\varphi(a)$ where $a \in S \cap M$ and predicate symbols $\vec{\dot A}$, the following holds: Whenever $\vec A \in S \cap \mathcal P(M)$, $\P \in \Gamma$, and $G$ is $\P$-generic over $V$, then $$(\langle M, \in, \vec A \rangle \models \varphi(a))^V \iff (\langle M, \in, \vec A \rangle \models \varphi(a) )^{V[G]},$$
where $\vec{\dot A}$ is meant to be interpreted in $M$ as $\vec A$, and the satisfaction is order $m+1$.

The case where $S = \mathcal P(M)$ is the boldface version, which we will denote \textbf{\emph{$\Gamma$-generic $\mathbf{\Sigma}_n^m(M)$-absoluteness}}. Without parameters is the lightface version, which we shall denote \textbf{\emph{$\Gamma$-generic}} \textbf{\emph{$\Sigma_n^m(M)$-absoluteness}}.

We will refer to a $\Sigma_2^1(H_{\omega_1})$-property of some object in $H_{\omega_1}$ that is subcomplete-generic absolute as being \textbf{\emph{$sc$-absolute}}, using the object as a parameter, and meaning that whether or not subcomplete forcing was performed does not affect whether this property holds or not.
\end{defn}

As a first approximation answering whether or not subcomplete forcing is available for a comparable absoluteness result, in an attempt to perhaps further concretely differentiate results about subcomplete forcing from countably closed forcing, it was only natural to look at whether subcomplete forcing preserves various examples of $\Sigma_2^1(H_{\omega_1})$ properties that we know to be preserved by countably closed forcing. As we see in the following two sections, many of the combinatorial properties and properties of $\omega_1$-trees that are preserved by countably closed forcing are also preserved by subcomplete forcing. 

\section{Subcomplete Forcing and Trees}
\label{sec:SCForcingAndTrees}
Before we talk about how subcomplete forcing interacts with trees, we should first introduce some (hopefully) familiar terminology. 

\begin{defn} \hfill \begin{itemize}
	\item A \textbf{\emph{tree}} is a partial order $T=\langle T, <_T \rangle$ in which the predecessors of any node are well-ordered and there is a unique minimal element called the root. We will tend to conflate the tree with its underlying set. 
	\item The \textbf{\emph{height of a node}} $t \in T$ is the order type of its predecessors. We write $T_\alpha$ for the $\alpha$th level of $T$, the set of nodes having height $\alpha$. The \emph{\textbf{height of a tree}} $T$, $\height(T)$, is the strict upper bound of the height of its nodes. 
	\item We write $T\rest \alpha$ for the subtree of $T$ of nodes having height less than $\alpha$. An \textbf{\emph{$\omega_1$-tree}} is a normal tree of height $\omega_1$ where all levels are countable. A tree of height $\omega_1$ is \textbf{\emph{normal}} if every node has (at least) two immediate successors, nodes on limit levels are uniquely determined by their sets of predecessors, and every node has successors on all higher levels up to $\omega_1$. 
	\item We write $T_t$ to denote the subtree of $T$ consisting of the nodes $s \in T$ with $s \geq_T t$. For nodes $t \in T$, by $\Succ_T(t)$ we mean the set of successors $s \geq_T t$ in the tree. 
	\item A \textbf{\emph{branch}} $b$ in $T$ is a maximal linearly ordered subset of $T$, and the length of the branch is its order type. By $b(\alpha)$ we mean the node on level $\alpha$ of the branch. We write $[T]$ for the set of \textbf{\emph{cofinal branches}}, those branches containing nodes on every level. If $t \in T$ is a node, then we write $b_t$ to mean the ``branch" below $t$; $b_t = \set{ s \in T}{ t \geq_T s }$.
	\item An $\omega_1$-tree is \textbf{\emph{Aronszajn}} if it has no cofinal branches. Two nodes $t$ and $s$ in $T$ are compatible, written $s \parallel t$, if there is $r \in T$ such that $r \geq_T t$ and $r \geq_T s$. With trees, this is the same as demanding that either $s <_T t$, $s>_T t$, or $s=t$, or that $s$ and $t$ are \textit{\textbf{comparable}}. Otherwise, they are incompatible, written $s \perp t$. An antichain in a tree is a set of pairwise incompatible elements. A \textbf{\emph{Suslin tree}} is an $\omega_1$-tree with no uncountable antichain. When forcing with a tree, we reverse the order, so that stronger conditions are higher up in the tree. Consequently, Suslin trees are $ccc$ as notions of forcing. A \textbf{\emph{Kurepa tree}} is an $\omega_1$-tree with $\omega_2$-many cofinal branches. \qedhere
\end{itemize} \end{defn}

Countably closed forcing does not add cofinal branches through $\omega_1$-trees, so it is natural to wonder whether other subcomplete forcing cannot do this either. Indeed, we see below that no subcomplete forcing can add a branch through a ground model's $\omega_1$-tree. Subcomplete forcing clearly cannot add a new branch of height less than $\omega_1$, since subcomplete forcing can't add reals (\textbf{Proposition \ref{prop:noreals}}). Below we show that subcomplete forcing cannot add cofinal branches either.

\begin{thm}\label{thm:scbranch} Let $T$ be an $\omega_1$-tree. If $\P$ is subcomplete and $G$ is $\P$-generic then $[T]=[T]^{V[G]}$.
\end{thm}
\begin{proof} 
Assume not. Let $\dot b$ be a name for a new cofinal branch through $T \subseteq H_{\omega_1}$; let $p$ be a condition forcing that $\dot b$ is a new cofinal branch through $\check T$. 
Let $\theta$ verify the subcompleteness of $\P$, and let's place ourselves in the standard setup: \begin{itemize} 
	\item $\P \in H_\theta \subseteq N = L_\tau[A] \models \ZFC^-$ where $\tau>\theta$ and $A \subseteq \tau$
	\item $\sigma: \N \cong X \preccurlyeq N$ where $X$ is countable and $\N$ is full
	\item $\sigma(\overline \theta, \overline{\P}, \overline T, \overline p, \overline{\dot b}) = \theta, \P, T, p, \dot b$.
\end{itemize}
Let $\alpha = \omega_1^{\N}.$ By elementarity, we have that $\overline p$ forces $\overline{\dot b}$ to be a new cofinal branch over $\N$. As we construct a generic $\G$ for $\overline{\P}$ over $\N$, we will use the countability of $\N$ to diagonalize against all ``branches" through the tree as seen on level $\alpha$ of the tree $T$ in $N$, thereby obtaining a contradiction. The dotted branch in the diagram below is what we wish to construct.

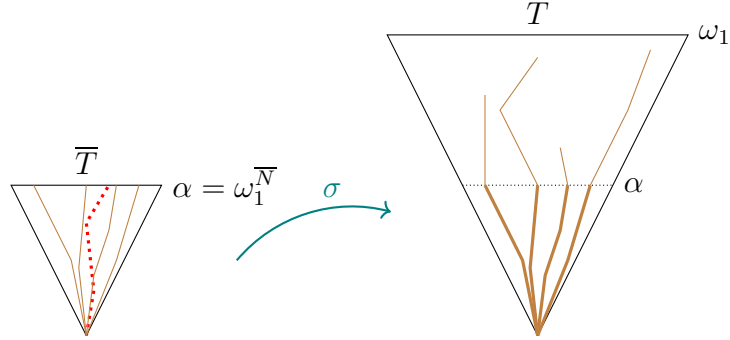
\begin{center}
\begin{tikzpicture}
\draw (-3,0)--(-4,2)--(-2,2)--(-3,0);
\node [above] at (-3,2) {$\overline T$};
\draw[very thick, dotted, color=red] (-3,0)--(-2.9,0.6)--(-3,1.5)--(-2.7,2);
\node [right] at (-2,2) {$\alpha = \omega_1^{\overline N}$};
\draw[color=brown] (-3,0)--(-3.2,1)--(-3.7,2);
\draw[color=brown] (-3,0)--(-3.1,0.9)--(-3,2);
\draw[color=brown] (-3,0)--(-2.9,0.8)--(-2.7,1.4)--(-2.6,2);
\draw[color=brown] (-3,0)--(-2.6,1)--(-2.3,2);
\draw [->, color = teal, thick] (-1,1) arc [radius=2, start angle=140, end angle= 75];
\node [above right, color=teal] at (0,1.7) {$\sigma$};
\draw (3,0)--(1,4)--(5,4)--(3,0);
\draw[densely dotted] (2, 2)--(4,2);
\draw[very thick, color=brown] (3,0)--(2.8,1)--(2.3,2);
\draw[color=brown] (2.3,2)--(2.3,3.2);
\draw[very thick, color=brown] (3,0)--(2.9,0.9)--(3,2);
\draw[color=brown] (3,2)--(2.5, 3)--(3,3.7);
\draw[very thick, color=brown] (3,0)--(3.1,0.8)--(3.3,1.4)--(3.4,2);
\draw[color=brown] (3.4,2)--(3.3, 2.5);
\draw[very thick, color=brown] (3,0)--(3.4,1)--(3.7,2);
\draw[color=brown] (3.7,2)--(4.2,3)--(4.5,3.8);
\node [above] at (3,4) {$T$};
\node [right] at (4,2) {$\alpha$};
\node [right] at (5,4) {$\omega_1$};
\end{tikzpicture}
\end{center}

Toward this end, enumerate the dense sets $\seq{\overline D_n}{ n < \omega}$ of $\overline{\P}$ that belong to $\N$. Also denote the sequence of downward closures of nodes on level $\alpha$ of $T$, the ``branches" through $T\rest \alpha$ that extend to have nodes of higher height in $T$, as $\seq{b_n}{n < \omega}$. Now define a sequence of conditions of the form $\overline p_n$ for $n < \omega$ that decide values of $\overline{\dot b}$ in $\overline T$ differently from $b_n$. Ensure along the way that for all $n$, \begin{itemize}
	\item $\overline p_{n+1} \in \overline D_n$ and 
	\item $\overline p_{n+1} \leq \overline p_n$. 
\end{itemize}
	
The construction (in $V$) may go as follows:

Let $\overline p_0 := \overline p \in \N$. For each $n < \omega$, note that there must be two conditions $\overline p^0_{n+1} \perp \overline p^1_{n+1}$, both extending $\overline p_n$, that decide the value of the branch $\overline{\dot b}$ to differ on some value. This always has to be possible since these conditions always extend $\overline p$, that forces $\overline{\dot b}$ to be new. Say $\overline p^1_{n+1} \forces \check x_{n} \in \overline{\dot b}$ and $\overline p^0_{n+1} \forces \check x_n \notin \overline{\dot b}$. Let $\overline p_{n+1}$ be a condition in $\overline D_n$ extending $\overline p^1_{n+1}$ if $x_n \notin b_n$, or a condition in $\overline D_n$ extending $\overline p^0_{n+1}$ otherwise.

Let $\G$ be the generic filter generated by the $\seq{ \overline p_n}{ n < \omega}$, let $\overline{\dot b}^{\G} = \overline b$. Since $\P$ is subcomplete, there is a condition $q \in \P$ such that whenever $G$ is $\P$-generic with $q \in G$, by subcompleteness we have $\sigma' \in V[G]$ such that: \begin{itemize}
	\item $\sigma': \N \prec N$
	\item $\sigma'(\overline \theta, \overline{\P}, \overline T, \overline p, \overline{\dot b}) = \theta, \P, T, p, \dot b$
	\item $\sigma'``\, \G \subseteq G$.
\end{itemize} 
So below $q$ there is a lift $\sigma^*:\N[\G] \prec N[G]$ extending $\sigma'$ with $\sigma^*(\overline b) = \sigma'(\overline{\dot b})^G=\dot b^G = b$, and $\sigma^*(\overline T) = \sigma'(\overline T)^G=T$. The point is that since $\overline p \in \G$, we have $N[G] \models p \in G$, so $b$ is a branch through $T$.

Furthermore, $\alpha$ is the critical point of the embedding $\sigma^*$. So below $\alpha$ the tree $T$, and thus the branch $b$, is fixed. In particular, in $N[G]$, $b \rest \alpha = \overline b$. However, $\overline b$ was constructed so as to not be equal to any of the $b_n$s, so it cannot be extended to become a branch through $T$, since it can't have a node on the $\alpha$th level. 

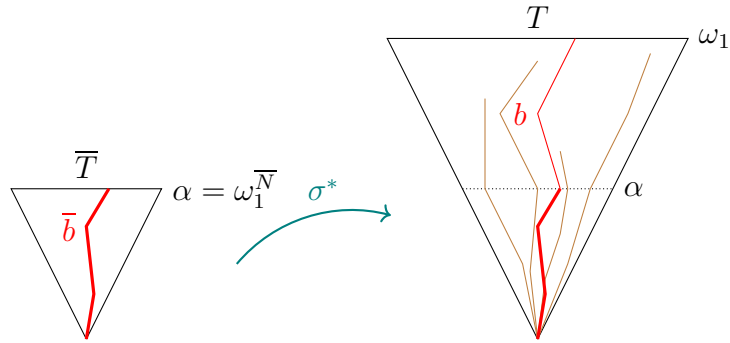
\begin{center}
\begin{tikzpicture}
\draw (-3,0)--(-4,2)--(-2,2)--(-3,0);
\draw[very thick,color=red] (-3,0)--(-2.9,0.6)--(-3,1.5)--(-2.7,2);
\node [left,color=red] at (-3,1.5) {$\overline b$};
\node [above] at (-3,2) {$\overline T$};
\node [right] at (-2,2) {$\alpha = \omega_1^{\overline N}$};
\draw [->, color = teal, thick] (-1,1) arc [radius=2, start angle=140, end angle= 75];
\node [above right, color=teal] at (-0.2,1.7) {$\sigma^*$};
\draw (3,0)--(1,4)--(5,4)--(3,0);
\draw[densely dotted] (2, 2)--(4,2);

\draw[color=brown] (3,0)--(2.8,1)--(2.3,2);
\draw[color=brown] (2.3,2)--(2.3,3.2);
\draw[color=brown] (3,0)--(2.9,0.9)--(3,2);
\draw[color=brown] (3,2)--(2.5, 3)--(3,3.7);
\draw[color=brown] (3,0)--(3.1,0.8)--(3.3,1.4)--(3.4,2);
\draw[color=brown] (3.4,2)--(3.3, 2.5);
\draw[color=brown] (3,0)--(3.4,1)--(3.7,2);
\draw[color=brown] (3.7,2)--(4.2,3)--(4.5,3.8);

\draw[very thick, color=red] (3,0)--(3.1,0.6)--(3,1.5)--(3.3,2);
\draw[color=red] (3.3,2)--(3, 3)--(3.5,4);
\node [left, color=red] at (3,3) {$b$};
\node [above] at (3,4) {$T$};
\node [right] at (4,2) {$\alpha$};
\node [right] at (5,4) {$\omega_1$};
\end{tikzpicture}
\end{center}
This is a contradiction.
\end{proof}

Immediately we see the following.
\begin{cor} Subcomplete forcing preserves the properties of being Aronszajn and the property of failing to be Kurepa of an $\omega_1$-tree. \end{cor}

Further investigating the motivating question of whether $sc$-generic $\mathbf{\Sigma}_2^1(H_{\omega_1})$-absoluteness follows from the maximality principle for subcomplete forcing, one might look at the countably closed case in order to simulate the argument for subcompleteness. Fuchs' argument used the fact that $\mathbf{\Sigma}_1^1(\omega_1)$-statements are absolute for countably closed forcings. This is true because countably closed forcings cannot add branches through wider trees (as in, trees with larger levels than $\kappa$-trees), and because $\mathbf{\Sigma}_1^1(\omega_1)$-absoluteness can be formulated as stating that branches cannot be added through such ``wider" Aronszajn trees. More precisely, we define what is meant by a ``wide" tree below:

\begin{defn} Let $\kappa$ and $\lambda$ be cardinals. We shall say that $T$ is a \textbf{\emph{$(\kappa, \leq \lambda)$-tree}} if $T$ is a tree of height $\kappa$ with levels of size less than or equal to $\lambda$. We shall refer to the size restriction on the levels in the tree in the second coordinate as the tree's width. 

An \textbf{\emph{Aronszajn $(\kappa, \leq \lambda)$-tree}} is a $(\kappa, \leq \lambda)$-tree with no cofinal branch. A \textbf{\emph{Kurepa $(\kappa, \leq \lambda)$-tree}} is a $(\kappa, \leq \lambda)$-tree with $\kappa^+$-many cofinal branches.\end{defn}

With this terminology, we can state the following lemma, giving an equivalent characterization of $\mathbf{\Sigma}^1_1(\omega_1)$-absoluteness.

\begin{lem} Assume $\CH$. A forcing notion $\P$ adds no new reals and no new cofinal branches to an $(\omega_1, \leq\omega_1)$-Aronszajn tree if and only if $\mathbf{\Sigma}_1^1(\omega_1)$-statements are absolute for $\P$. \end{lem}
\begin{proof}
Clearly if $\mathbf{\Sigma}_1^1(\omega_1)$-statements are absolute for $\P$, then $\P$ adds no new reals and no new branches to an $(\omega_1, \leq\omega_1)$-Aronszajn tree.

For the other direction, suppose that $\P$ adds no new reals, and no new branches to any $(\omega_1, \leq \omega_1)$-Aronszajn trees. Since we are assuming $\CH$, we may assume that $\P$ adds no new branches to any $(\omega_1, \leq 2^\omega)$-trees. Let $G \subseteq \P$ be generic. Upward absoluteness between $V$ and $V[G]$ clearly holds for existential statements.
To show downward absoluteness, let $A \subseteq \omega_1$, $A \in V$. 

Let $\psi(A)$ be the following statement:
	$$\exists X \ (\omega_1, A, X) \models \varphi(a),$$
where $\varphi(a)$ is a first order sentence in the language of set theory using $A$ and $X$ as predicates, and $a < \omega_1$. Suppose that $\psi(A)$ is true in In $V[G]$, as witnessed by $X$. 

Let $T$ be the tree consisting of nodes of the form $(\alpha, x)$ such that $x \subseteq \alpha$, and $(\alpha, A \cap \alpha, x) \models \varphi(a),$ where the tree ordering is defined so that
	$$(\alpha, x) \leq (\beta, y) \iff (\alpha, A\cap \alpha, x) \preccurlyeq (\beta, A \cap \beta, y).$$
By a standard Lowenheim-Skolem style argument, the set 
	$$C = \set{ \alpha < \omega_1 }{ (\alpha, A \cap \alpha, X \cap \alpha) \preccurlyeq (\omega_1, A, X) }$$
is club in $V[G]$.	
Thus the set $\set{ (\alpha, X \cap \alpha) }{ \alpha \in C }$ defines a cofinal branch through $T$ in $V[G]$, since for all countable $\alpha$, we have $X \cap \alpha \in V$ as $\P$ doesn't add reals. Since additionally $\P$ can't add new branches to $T$, it follows that $T$ has a cofinal branch, call it $b$, in $V$. Finally, let
	$$X' = \bigcup \set{ x }{ \exists \alpha < \omega_1 \ (\alpha, x) \in b }.$$ 
Since $(\omega_1, A, X')$ is the union of an elementary chain of models satisfying $\varphi(a)$, it must also satisfy $\varphi(a)$ in $V$, and thus $\psi(A)$ holds in $V$ as witnessed by $X'$ as desired.
\end{proof}

Thus it is natural to ask whether subcomplete forcing can add cofinal branches to wider Aronszajn trees, in order to see if absoluteness results similar to those for countably closed forcing hold for subcomplete forcing.

\begin{question} \label{ques:aronszajnbranch?} Can subcomplete forcing add cofinal branches to $(\omega_1, \leq \omega_1)$-Aronszajn trees? \end{question} 

If we omit the requirement that the wider tree is Aronszajn, we already know of some trees of height $\omega_1$ that  subcomplete forcing adds branches to: 

\begin{prop} Subcomplete (or even countably closed) forcing may add a cofinal branch to an $(\omega_1, \leq 2^{\omega})$-tree. \end{prop}
\begin{proof}
The point here is that the poset $\Add(\omega_1,1)$ is subcomplete since it is countably closed, but it may be viewed as a tree of height $\omega_1$ that has levels up to and including size $2^{\omega}$. Of course this tree is not Aronszajn, it is Kurepa. Every cofinal branch through the tree corresponds to a subset of $\omega_1$, which there are already more than $\omega_1$-many of in the ground model.
\end{proof}

In general, countably closed forcing can't add a (cofinal) branch to any $(\omega_1, \leq \kappa)$-Aronszajn tree for cardinals $\kappa$.

If we allow the size of the levels of the height $\omega_1$-Aronszajn tree to be large, then we can obtain a slightly more complicated example of a tree that subcomplete forcing does add a cofinal branch to, using the forcing denoted by Jensen as $\P_A$. This forcing is designed to force Friedman's principle for $A$, where $A$ is a stationary subset of cofinality $\omega$ points in $\omega_2$, which we shall denote as $A\subseteq \omega_2 \cap \cof(\omega)$ stationary. 

\begin{defn}
Let $\kappa > \omega_1$ be regular, and let $A \subseteq \kappa \cap \cof(\omega)$. We write \emph{$\P_{A}$} to denote the forcing designed to shoot a cofinal, normal sequence of order type $\omega_1$ through $A$. The conditions of $\P_A$ consist of normal functions of the form $p: \nu +1 \to A$, where $\nu< \omega_1$, and extension is defined in the usual way, where $p \leq q$ if and only if $q \subseteq p$. If $(\omega_2 \cap \cof(\omega))\setminus A$ is stationary in $\omega_2 \cap \cof(\omega)$, then $\P_A$ is not countably closed. 

Then if $G$ is $\P_A$-generic, $\cup G : \omega_1 \to A$ is normal and cofinal in $\kappa$.

The forcing $\P_A$ is used to show that Friedman's Principle holds under the subcomplete forcing axiom $\SCFA$ , where \textbf{\emph{Friedman's Principle}} at a regular cardinal $\kappa > \omega_1$ states that if $A \subseteq \kappa$ is any stationary set of $\omega$-cofinal ordinals, then there is a normal function $f: \omega_1 \to A$.
\end{defn}

\begin{fact}[Jensen] $\P_A$ is subcomplete. \end{fact}

\begin{prop} Suppose that Friedman's Principle fails for $\omega_2$. Then subcomplete forcing may add a cofinal branch to an $(\omega_1, \leq \omega_2 \cdot 2^\omega)$-Aronszajn tree. \end{prop}
\begin{proof}
Let $A \subseteq \cof(\omega) \cap \omega_2$ witness the failure of Friedman's Prinicple, meaning that $A$ is stationary. Consider the forcing poset $\P_A$ as a tree. It has size $\omega_1 \cdot \omega_2^\omega = \omega_2^\omega$, since it consists of functions with domain in $\omega_1$, and each condition is from a countable ordinal to $\omega_2$. 

Considering $\P_A$ as a tree, it has height $\omega_1$. Each level has size less than or equal to $\omega_2^\omega = \omega_2 \cdot 2^\omega$. Moreover, since Friedman's Principle fails, the tree $\P_A$ has no cofinal branches and is thus Aronszajn, yet forcing with the tree will of course add a cofinal branch.
\end{proof}

However, we may slightly tweak the proof of \textbf{Theorem \ref{thm:scbranch}} to see that the levels of trees of height $\omega_1$ could have size less than $2^\omega$ and subcomplete forcing adds no branches through them, via virtually the same argument. The argument showing that countably closed forcing adds no branches through such trees is more or less a counting argument; refer to Fuchs~\cite{Fuchs:2008rt} for details, the proof is attributed to Silver. In the case of subcompleteness, we will obtain a contradiction by counting the number of possible generics that could add new branches and seeing that there are too many possibilities, similar to what is done in the countably closed case.

\begin{thm} Subcomplete forcing cannot add (cofinal) branches to $(\omega_1, < 2^\omega)$-trees. \end{thm} 
\begin{proof} 
Assume not. Let $\dot b$ be a name for a new branch through $T \subseteq H_{\omega_1}$ an $(\omega_1, < 2^\omega)$-tree; let $p$ be a condition forcing that $\dot b$ is a new cofinal branch through $\check T$. 
Let $\theta$ verify the subcompleteness of $\P$ and let's place ourselves in the standard setup: \begin{itemize} 
	\item $\P \in H_\theta \subseteq N = L_\tau[A] \models \ZFC^-$ where $\tau>\theta$ and $A \subseteq \tau$
	\item $\sigma: \N \cong X \preccurlyeq N$ where $X$ is countable and $\N$ is full.
	\item $\sigma(\overline \theta, \overline{\P}, \overline T, \overline p, \overline{\dot b}) = \theta, \P, T, p, \dot b$. 
\end{itemize}
Let $\alpha = \omega_1^{\N}$, the critical point of the embedding $\sigma$. By elementarity, we have that $\overline p$ forces $\overline{\dot b}$ to be a new branch over $\N$. We will construct continuum-many generics $\G_r$ for $\overline{\P}$ over $\N$, indexed by reals, each of which will correspond to continuum-many different values of the generic branch on level $\alpha$ of the tree $T$ in $N$, to obtain a contradiction.

Toward this end, enumerate the dense sets $\seq{D_n}{ n < \omega}$ of $\overline{\P}$ that belong to $\N$, so that $\overline p \in D_0$.

We would like to construct binary trees of conditions in $\overline \P$ and branches in $\overline T$: $ P= \seq{ \overline p_x }{ x \in 2^{<\omega} }$, $B = \seq{ \overline b_x }{ x \in 2^{< \omega} }$ such that, letting $|x|$ be the length of $x$, we have the following: \begin{itemize}
	\item $\overline p_x \in D_{|x|}$
	\item $x \subseteq y \implies \overline p_y \leq \overline p_x \leq \overline p$
	\item for some $\beta > |x|$, $p_x \forces \overline{\dot b} \rest \check \beta  = \overline b_x$
	\item $\overline b_{x^\frown \langle 0 \rangle} \neq \overline b_{x^\frown \langle 1 \rangle}$.
\end{itemize}

To do this, let $\overline p_0 := \overline p \in \N$. As we noted in the proof of \textbf{Theorem \ref{thm:scbranch}}, there must be two conditions $\overline q^0_{1} \perp \overline q^1_{1}$, both extending $\overline p_0$, that decide the value of the branch $\overline{\dot b}$ differently. This always has to be true since these conditions extend $\overline p$, that forces that there is a new cofinal branch $\overline{\dot b}$. In particular, by our reasoning there is $\beta< \overline{\omega_1}$, and two conditions $\overline q_1^0 \leq \overline p_0$ and $\overline q_1^1 \leq \overline p_0$ such that $$\overline q_1^0 \forces \overline{\dot b}(\beta) = t^0 \text{ \ \ and \ \ } \overline q^1_{1} \forces \overline{\dot b}(\beta) = t^1,$$ where $t^0 \neq t^1$. Let $\overline b_{\langle 0 \rangle}$ be the branch in $\overline T$ below $t^0$ and $\overline b_{\langle 1 \rangle}$ be the branch in $\overline T$ below $t^1$. It must be that $\overline b_{\langle 0 \rangle} \neq \overline b_{\langle 1 \rangle}$. Moreover we may extend $\overline q^0_1$ and $\overline q^1_1$ to conditions $\overline p_{\langle 0 \rangle}, \overline p_{\langle 1 \rangle} \in \overline D_1$. 

Continuing in such a fashion, we may recursively continue to define $\overline p_x$ for $x \in 2^{< \omega}$. 

In particular, suppose $\overline p_x$ and $b_x$ are defined for $x$ of length $n$. By our reasoning above there is $\beta > |x|$, and two conditions above each $\overline p_x$, such that $\overline q_x^0 \leq \overline p_x$ and $\overline q_x^1 \leq \overline p_x$ so that 
	$$\overline q_x^0 \forces \overline{\dot b}(\beta) = t^0 \text{\ \ and \ \ } \overline q^1_x \forces \overline{\dot b}(\beta) = t^1,$$
where $t^0 \neq t^1$. Let $\overline b_{x^\frown \langle 0 \rangle}$ be the branch in $\overline T$ below $t^0$ and $\overline b_{x^\frown \langle 1 \rangle}$ be the branch in $\overline T$ below $t^1$. Extending each of these incompatible conditions so as to land in the $(n+1)$th dense set, we find $\overline p_{x^\frown \langle 0 \rangle} \leq \overline q^0_x$ and $\overline p_{x^\frown \langle 1 \rangle} \leq \overline q^1_x$ such that $\overline p_{x^\frown \langle 0 \rangle}, \overline p_{x^\frown \langle 1 \rangle} \in D_{n+1}$. 
Then, as desired, for each $x$ of length $n$ we have that $\overline p_{x^\frown \langle 0 \rangle}, \overline p_{x^\frown \langle 1 \rangle} \in D_{n+1}$. We've also designed it so that $\overline p_{x^\frown \langle 1 \rangle} \leq \overline p_x \leq \overline p$ and $\overline p_{x^\frown \langle 0 \rangle} \leq \overline p_x \leq \overline p$. Since we know that there is some $m'$ such that $\overline p_x \forces \overline{ \dot b} \rest m' = \overline b_x$ we know that $p_{x^\frown \langle 1 \rangle}$ and $p_{x^\frown \langle 0 \rangle}$ force the same thing since they both extend $\overline p_x$. Thus our construction gives us that $\overline p_{x^\frown \langle 0 \rangle} \forces \overline{ \dot b} \rest m = \overline b_{x^\frown \langle 0 \rangle}$ and $\overline p_{x^\frown \langle 1 \rangle} \forces \overline{ \dot b} \rest m' = \overline b_x$.

So we have our binary trees $P$ and $B$ as desired. Any chain of conditions in the binary tree $P$ will generate a generic filter $\overline G_r$; every real $r: \omega \to 2$ codes a path in the binary tree of conditions generating the generic. This is because our conditions were chosen to meet all of the dense sets in our list. Moreover, each generic filter $\overline G_r$ corresponds to a branch $\overline b_r$, where for each initial segment $t$ of $r$ satisfies that for some $\beta > |x|$, $\overline b_r \rest m = \overline b_x$. Because of how we chose $P$, this gives us that $$\N[\G] \models {\overline{\dot b}}^{\G} =\overline b_r.$$

For each $r$ let $\overline{\dot b}^{\G_r} = \overline b_r$. Since $\P$ is subcomplete, for each $r$ there is a condition $q_r \in \P$ such that whenever $G$ is $\P$-generic with $q_r \in G$, by subcompleteness we have $\sigma_r \in V[G]$ such that: \begin{itemize}
	\item $\sigma_r: \N \prec N$
	\item $\sigma_r(\overline \theta, \overline{\P}, \overline T, \overline p, \overline{\dot b}) = \theta, \P, T, p, \dot b$
	\item $\sigma_r``\, \G_r \subseteq G$.
\end{itemize} 
So below each $q_r$ there is a lift $\sigma_r^*:\N[\G_r] \prec N[G]$ extending $\sigma_r$ with $\sigma_r^*(\overline b_r) = \sigma_r(\overline{\dot b})^G=\dot b^G = b_r$, and $\sigma_r^*(\overline T) = \sigma_r(\overline T)^G=T$. 

This means that we may force over $N$ with $\P$ to obtain continuum many cofinal branches through the tree, $\seq{ b_r }{ r \in 2^\omega}$. Each of these branches must, of course, have a node on level $\alpha$. Since each cofinal branch $\overline b_r$ is unique, and since $\alpha$ is the critical point, this means that there are $2^\omega$-many nodes on level $\alpha$ of the tree $T$ in $N$, a contradiction.
\end{proof}

\begin{cor} The property of being Aronszajn of $(\omega_1, <2^\omega)$-trees is $sc$-absolute. \end{cor}

Thus if $2^\omega = \omega_2$, this result is optimal in the sense that it answers \textit{Question} \ref{ques:aronszajnbranch?} in the affirmative.

\subsection{Suslin Tree Preservation}
\label{subsec:SuslinTreePreservation}
As we saw in \textbf{Theorem \ref{thm:scbranch}}, ground-model Aronszajn trees are preserved by subcomplete forcing. Using a slightly different method than that of the previous theorems, Jensen shows in \cite[Chapter 3 p.~10]{Jensen:2009wc} that ground-model Suslin trees are also preserved by subcomplete forcing. We give the proof below.

\begin{fact}[Jensen] \label{fact:Suslinpres} Subcomplete forcing preserves the property of being Suslin of $\omega_1$-trees. \end{fact}
\begin{proof}
Let $T$ be a Suslin tree. Let $\P$ be subcomplete. Suppose toward a contradiction that $p \in \P$ forces that $\dot A$ is a maximal antichain of size $\omega_1$. Let $\theta$ verify the subcompleteness of $\P$ and as usual we place ourselves in the standard setup: \begin{itemize} 
	\item $\P \in H_\theta \subseteq N = L_\tau[A] \models \ZFC^-$ with $\tau>\theta$ and $A \subseteq \tau$
	\item $\sigma: \N \cong X \preccurlyeq N$ where $X$ is countable and $\N$ is full
	\item $\sigma(\overline \theta, \overline{\P}, \overline T, \overline p, \overline{\dot A}) = \theta, \P, T, p, \dot A$. 
\end{itemize} 
Letting $\alpha = \omega_1^{\N}$, we have that $\overline T = T\rest \alpha$ as usual. Let $M$ be a countable, transitive $\ZFC^-$ model with both $\N, T\rest (\alpha+1) \in M$. Let $\G \subseteq \overline{\P}$ be generic over $M$ with $\overline p \in \G$. So $\G$ is also generic over $\N$.

We can now work below a condition in $G \subseteq \P$ generic to obtain a $\sigma' \in V[G]$ such that: \begin{itemize}
	\item $\sigma': \N \prec N$
	\item $\sigma'(\overline \theta, \overline{\P}, \overline T, \overline p, \overline{\dot A}) = \theta, \P, T, p, \dot A$
	\item $\sigma'``\, \G \subseteq G$.
\end{itemize} 
As usual we have a lift $\sigma^*:\N[\G] \prec N[G]$. Letting $\overline A=\overline{\dot A}^{\G}$ and $\dot A^G=A$ we have that $\sigma^*(\overline A)=A$. Let $\seq{b_t}{t \in T_\alpha}$ be the collection of partial branches below the nodes of level $\alpha$ of the tree $T$.

Every node in $T$ above level $\alpha$ has to have a predecessor in level $\alpha$. For each $t \in T_\alpha$, $\G$ is $\overline{\P}$-generic over $\N[b_t]$ since $b_t$ is $\overline T$-generic over $\N$ - as cofinal branches through Suslin trees are generic. By the product lemma, each $b_t$ is $\overline T$-generic over $\N[\G]$. Since $\overline A$ is maximal, $b_t \cap \overline A \neq \emptyset$.
Thus $\overline A$ is sealed in $\overline T = T \rest \alpha$, meaning it has no elements above level $\alpha$. But since $\overline A \subseteq A$ and $A$ is maximal, this means that $A$ is countable, so $T$ remains Suslin as desired. 
\end{proof}

Immediately we have the following:

\begin{cor} The property of being Suslin of $\omega_1$-trees is $sc$-absolute. \end{cor}

The previous theorem of course also shows that forcing with a Suslin tree is not subcomplete, even though it does not add a real. This is yet another example showing that subcomplete forcing is not the same as countably distributive forcing. Furthermore, this gives us the following, which was observed by Miha Habi\v c:

\begin{cor} \label{cor:CCCnotSC} Nontrivial $ccc$ forcings are not subcomplete. \end{cor}
\begin{proof} Suppose instead toward a contradiction that $\P$ is both $ccc$ and subcomplete. By subcompleteness, $\P$ does not add a real, making $\P$ countably distributive (meaning it doesn't add countable sequences of ordinals). This is because $ccc$ forcing notions have the countable covering property: supposing instead $\P$ adds a countable sequence of ordinals, this sequence can be covered by a countable set in the ground model. But using this cover, the new set of ordinals can be coded as a new real. This contradicts the assumption that no reals are added. Thus $\P$ doesn't add countable sequences of ordinals, and is countably distributive. 

But $ccc$ countably distributive forcings are also known as Suslin algebras, which always add a branch to to some Suslin tree~\footnote{This is stated in Jech \cite{Jech:2002qr} after Definition 30.19.}, so $\P$ can't be subcomplete. \end{proof} 

Moreover, the method of proof showing that Suslin trees are preserved is optimal in the sense that maximal antichains of $\omega_1$ trees are not necessarily preserved. Indeed the proof seems to rely on the fact that the tree is Suslin, in particular using the fact that cofinal branches through Suslin trees are always generic.

\begin{prop} If $T$ is an $\omega_1$-tree that is not Suslin, then $\Add(\omega_1, 1)$ adds a new maximal antichain to $T$. Indeed, any forcing adding a new subset to $\omega_1$ will add a new maximal antichain to $T$. \end{prop}
\begin{proof}
Let $A = \set{ a_\alpha }{ \alpha < \omega_1 }$ be a maximal antichain in $T$. Let $G \subseteq \omega_1$ be $\Add(\omega_1, 1)$-generic. Let $A' = \set{ a_\alpha }{\alpha \notin G} \cup \set{ t \in T }{ \exists \alpha \in G \ \ t \in \Succ_T(a_\alpha) }$. Then $A'$ is a maximal antichain in $T$ and as we have that $G = \set{ \alpha < \omega_1 }{ a_\alpha \notin A' }$, it must be that $A' \notin V$. 
\end{proof}

At this point one might wonder exactly how robust Suslin trees really are under subcomplete forcing. Suppose that $T$ is a Suslin tree. Then the following theorem shows that not only is it the case that $T$ is still Suslin after any subcomplete forcing, and not only are there no new cofinal branches added to $T$, but even after forcing with $T$ in a subcomplete extension, there are no more branches than if you had forced with $T$ without the presence of the subcomplete forcing to begin with.

\begin{thm} \label{thm:genericbranchpres} 
If $T$ is a Suslin tree and $\P$ is subcomplete, then $ [T]^{V^{\P \times T}} = [T]^{V^T}$. In other words, subcomplete forcing doesn't add to the collection of generic branches through ground-model Suslin trees. \end{thm}
\begin{proof}
Suppose not. Let $T$ be a Suslin tree. Let $\P$ be subcomplete. Let $\ddot b$ be a $\P$-name for a $\check T$-name for a new branch through $T$ and suppose we have $p \in \P$, $t \in T$ satisfying: 
	$$p \forces_{\P} \left( \check t \forces_{\check T} \left( \ddot b \in [\check T] \ \land \ \ddot b \notin [\check T]^{\check V[\Gamma_{\check T}]} \right) \right)$$
where $\Gamma_{\check T}$ is a ($\P$-name for a) $\check T$-name for a generic branch for $\check T$. In other words, whenever $G\times b \subseteq \P\times T$ is generic with $\langle p, t \rangle \in G \times b$ we have that $(\ddot b^G)^b \in [T]^{V[G][b]} \setminus [T]^{V[b]}$.	
Let $\theta$ verify the subcompleteness of $\P$, and let's get ourselves into the standard setup: \begin{itemize} 
	\item $\P \in H_\theta \subseteq N = L_\tau[A] \models \ZFC^-$ with $\tau>\theta$ and $A \subseteq \tau$
	\item $\sigma: \N \cong X \preccurlyeq N$ where $X$ is countable and $\N$ is full
	\item $\sigma(\overline \theta, \overline{\P}, \overline T, \overline p, \overline{\ddot b}, \overline t) = \theta, \P, T, p, \ddot b, t$. 
\end{itemize}
Let $\alpha = \omega_1^{\N}$, the critical point of $\sigma$. We have $\overline T = T \rest \alpha$ as usual. 

Additionally by elementarity, we have that $\overline p$ satisfies the same property as the condition $p$ does but relativized to $\N$:
	$$\overline p \forces_{\overline \P} \left( \check{\overline t} \forces_{\check{\overline T}} \left( \overline{\ddot b} \in [\overline{\check T}] \ \land \ \overline{\ddot b} \notin [\check{\overline T}]^{\check V[\Gamma_{\check{\overline T}}]} \right) \right)$$
Let $\vec b = \seq{b_n}{n < \omega}$ enumerate the branches below nodes on the $\alpha$th level of $T$, chosen so that $\overline t \in b_0$. Let $\seq{\overline D_n}{n< \omega}$ enumerate the dense subsets of $\overline{\P}$ in $\N$. 

Again the idea is to carefully construct a generic $\G \subseteq \overline{\P}$ over $\N$ by diagonalizing against branches in $\vec b$. We construct a $\leq_{\overline{\P}}$-sequence $\seq{\overline p_n}{ n< \omega }$ satisfying, for each $n<\omega$: \begin{enumerate}
	\item $\overline p_n \in \overline D_n$
	\item In $\N$, $\overline p_n \forces_{\overline \P} \left( \check{\overline{t'}} \forces_{\check {\overline{T}}} \overline{\ddot b}(\check \gamma) \neq (\check{b_n(\gamma)}) \right)$, for some $\gamma < \alpha$ and $\overline{t'} \in b_0$, $\sigma(\overline{t'}) = t' \geq_T t$. In other words, $\overline p_n$ forces that the canonical name for $\overline{t'}$ forces the value of the generic branch to be different from one of the ``branches" in our list in $N$.
	\end{enumerate}
In order to inductively define such a sequence $\vec{\overline p}$, let's suppose $\overline p_m$ have been defined for $m<n$. To get $\overline p_n$, choose $\overline q_n$ below each $\overline p_m$ for all $m<n$, satisfying $\overline q_n \in \overline D_n$. 

As $\overline T$ is Suslin in $\N$ and cofinal branches are generic over Suslin trees, we have that $\N[b_0]$ forms a generic extension over $\N$. Let $\G_0, \G_1$ be mutually $\overline{\P}$-generic over $\N[b_0]$ so that $\overline p, \overline q_n \in \G_0 \cap \G_1$. 	
Since both $\overline T, \overline{\P} \in \N$ we have that for $i=0,1$; $$\N[b_0][\G_i]=\N[\G_i][b_0].$$ For $i=0,1$ let $\overline c_i = (\overline{\ddot b}^{\G_i})^{b_0}$. 

Since for each $i$, the condition $\overline p \in \G_i$ and $\overline t \in b_0$, both of the $\overline c_i$ are cofinal branches through $\overline T$ and $\overline c_i \neq b_0$ by our inductive assumption of item \textbf{2} since namely $\overline q_n \leq \overline p_0$.
It follows from the mutual genericity of $\G_0$ and $\G_1$ that $\overline c_0 \neq \overline c_1$; otherwise, suppose that $\overline c = \overline c_0 = \overline c_1$. Then we'd have $$\overline c \in \N[\G_0][b_0] \cap \N[\G_1][b_0] = \N[b_0][\G_0] \cap \N[b_0][\G_1] = \N[b_0]$$ so $\overline c \in [\overline T]^{V[b_0]}$, a contradiction.

So let $\overline c \in \{ \overline c_0, \overline c_1 \}$ be such that $\overline c \neq b_n$. 
Thus we may choose $\gamma < \alpha$ so that the value of $\overline c$ on level $\alpha$ is not the same as $b_n(\gamma)$. Then this holds in some $\N[\G_i][b_0]$, and we can obtain a condition $\overline p_n \leq \overline q_n$ forcing this.

Now let $\G \subseteq \overline{\P}$ be generic over $\N$, generated by $\vec{\overline p}$. We can from now on work below a condition in $G \subseteq \P$ generic to obtain a $\sigma' \in V[G]$ such that: \begin{itemize}
	\item $\sigma': \N \prec N$
	\item $\sigma'(\overline \theta, \overline{\P}, \overline T, \overline p, \overline{\ddot b}, \overline t) = \theta, \P, T, p, \ddot b, t$.
	\item $\sigma'``\, \G \subseteq G$.
\end{itemize} 
As usual we have a lift $\sigma^*:\N[\G] \prec N[G]$. Now let $c$ be $T$-generic over $V[G]$, with $t \in c$, so that $b_0 \subseteq c$. Then $(\ddot b^G)^c=b$ is a cofinal branch and $c \neq b$ as before. So we have that for all $n$, there is $\gamma < \alpha$ such that $N[G][c] \models b(\gamma) \neq b_n(\gamma)$, a contradiction.
\end{proof}

The previous result yields the following result on Suslin trees with the unique branch property. The unique branch property is explored in more detail by Hamkins and Fuchs \cite{Hamkins:qf}. We give the definition here.

\begin{defn} A normal $\omega_1$-tree $T$ has the \textbf{\emph{unique branch property}} (\textit{ubp}) so long as $$\mathbbm 1 \forces_T `` \check T \text{ has exactly one new cofinal branch.}"$$ That is, after forcing with the tree, $T$ has  exactly one cofinal branch that was not in the ground model. \end{defn}

\begin{thm} \label{thm:UBP} 
Subcomplete forcing preserves the $ubp$ of Suslin trees. \end{thm}
\begin{proof}
Let $T$ be a normal Suslin tree with the $ubp$. Let $\P$ be subcomplete. Let $\ddot b$ be a $\P$-name for a $T$-name so that there is $p \in \P$, $t \in T$ satisfying: $$ p \forces_{\P} \left( \check t \forces_{\check T} `` \ddot b \text{ is a cofinal branch different from the canonical } \check T\text{-generic branch}" \right).$$ In other words, $\P$ adds a new cofinal branch through the tree, so that $$[T]^{V^T} \neq [T]^{V^{\P \times T}},$$ contradicting \textbf{Theorem \ref{thm:genericbranchpres}}. 
\end{proof}

\begin{cor} The unique branch property of Suslin trees is $sc$-absolute. \end{cor} 

We have seen that after subcomplete forcing, the set of generic branches through a Suslin tree, obtained via forcing with the tree, is not effected. What about the set of maximal antichains? We see below that the property of being Suslin off the generic branch is also $sc$-absolute. This is another property of Suslin trees that has been explored by Hamkins and Fuchs \cite{Hamkins:qf}.

\begin{defn} A Suslin tree is \textbf{\emph{Suslin off the generic branch}} so long as after forcing with $T$ to add a generic branch $b$, for any node $t$ not in $b$, the tree $T_t$ remains Suslin. \end{defn}

\begin{thm} \label{thm:suslinoffbranch}
The property of being Suslin off the generic branch is $sc$-absolute. In other words, if $T$ is a Suslin tree that is also Suslin off the generic branch, then $T$ is still Suslin off the generic branch after subcomplete forcing.
\end{thm}
\begin{proof}
Let $\P$ be subcomplete, and suppose that the desired result fails. Let $p \in \P$, with $t, t' \in T$ incompatible, and $\ddot A$ a $\P$-name for a $T$-name, so that $p$ forces that $t$ forces that $\ddot A$ is a maximal antichain in $\check{T_{t'}}$ and that $|\ddot A|=\check \omega_1$. Let $\theta$ verify the subcompleteness of $\P$, and assume we are in the following situation: \begin{itemize} 
	\item $\P \in H_\theta \subseteq N = L_\tau[A] \models \ZFC^-$ with $\tau>\theta$ and $A \subseteq \tau$
	\item $\sigma: \N \cong X \preccurlyeq N$ where $X$ is countable and $\N$ is full
	\item $\sigma(\overline \theta, \overline{\P}, \overline T, \overline{\ddot A}, \overline t, \overline{t'}, \overline p) = \theta, \P, T, \ddot A, t, t', p$. 
\end{itemize}

Let $\alpha=\omega_1^\N$, of course $\alpha = \cp(\sigma)$. 

Let $M$ be a countable, transitive $\ZFC^-$ model, with $\N, T \rest (\alpha+1) \in M$ so that $\set{ b_u }{ u \in T(\alpha) } \subseteq M$. 

Let $\G \subseteq \P$ be $M$-generic for $\overline{\P}$ with $\overline p \in \G$. By subcompleteness we can from now on work below a condition in $G \subseteq \P$ generic to obtain an elementary embedding $\sigma' \in V[G]$ such that: \begin{itemize}
	\item $\sigma': \N \prec N$
	\item $\sigma'(\overline \theta, \overline{\P}, \overline T, \overline{\ddot A}, \overline t, \overline{t'}, \overline p) = \theta, \P, T, \ddot A, t, t', p$
	\item $\sigma'``\, \G \subseteq G$.
\end{itemize} 
As usual we have a lift $\sigma^*:\N[\G] \prec N[G]$, and we have $\overline T = T \rest \alpha$. 

Let $u \in T(\alpha)$ be above the node $t$ in the tree, such that $b_u \in M$ is a branch through $\overline T = T \rest \alpha$. Since $\N \models ``\overline T \text{ is Suslin}"$, it follows that $b_u$ is $\overline T$-generic over $\N$, since cofinal branches through Suslin trees are generic. 

Let $b \subseteq T$ be generic over $V[G]$, with $u \in b$. Then $b_u \subseteq b$. Let $\overline A = (\overline{\ddot A}^{\G})^{b_u}$. Then $$\N[\G][b_u] \models ``\overline A \text{ is a maximal antichain in }\overline T_{t'} \text{ of size } \omega_1."$$ This is by elementarity, as $\overline p \in \G$. Indeed, letting $A = (\ddot A^G)^b$, we have that $(A \cap T) \rest \alpha = \overline A$. 

But the same can be done for $t'$; let $u' \in T(\alpha)$, with $u'$ above $t'$ in the tree. Then $b_{u'} \in M$. So $\G$ is $\overline{\P}$-generic over $\N[b_u][b_{u'}]$ since both are in $M$. Moreover, $b_u$ and $b_{u'}$ are mutually $\overline T$-generic over $\N$, since $$\N \models ``\overline T \text{ is Suslin off the generic branch."}$$ Thus $b_{u'}$ is $\overline T_{t'}$-generic over $\N[b_u][\G]=\N[\G][b_u]$. 

But $\overline A$ is a maximal antichain in $\overline T_{t'}$ and $\overline A \in \N[\G][b_u]$, so $\overline A \cap b_{u'} \neq \emptyset$. Thus $\set{ b_{u'} }{ u' \in T_{t'}(\alpha) }$ seals $\overline A$ in $T_{t'}$; meaning that $\overline A$ is bounded in the tree by level $\alpha$, which makes it countable. Thus by the maximality of $\overline A$ and $A$, it must be that $A$ is also countable, a contradiction.
\end{proof}

\begin{defn} Let $n$ be a natural number. A Suslin tree $T$ is \textbf{\emph{$n$-fold Suslin off the generic branch}} if after forcing with with the tree $n$ times, or forcing with $T^n$ that adds $n$ branches $b_1, \dots, b_n$, then $T_p$ remains Suslin for any $p$ not on any $b_i$. \end{defn}

It follows, by a similar argument to \textbf{Theorem \ref{thm:suslinoffbranch}}, that the property of an $\omega_1$-tree of being $n$-fold Suslin off the generic branch is $sc$-absolute.

\section{Iterating Subcomplete Forcing}
\label{sec:IteratingSCForcing}
The two-step iteration of subcomplete forcings is subcomplete, as is shown by Jensen~\cite[Chapter 4]{Jensen:2012fr}:
\begin{fact}[Jensen] If $\P$ is subcomplete and $\forces_\P ``\dot \Q \text{ is subcomplete}"$, then $\P * \dot \Q$ is subcomplete. \end{fact}

However, we can see now from previous results that the converse does not hold for subcomplete forcing, even though the converse does hold for countably closed forcing. Namely, we have the following:

\begin{prop} Suppose that $\P * \dot \Q$ is subcomplete. Then it is not necessarily true that $\forces_\P ``\dot \Q \text{ is subcomplete}"$, even if $\P$ is subcomplete.  \end{prop}
\begin{proof}
Let $\P$ be the forcing to add a Suslin tree, where the conditions are normal subtrees of successor height. Let $\dot{T}$ denote the forcing that adds a branch through the tree added by $\P$ (forcing with the tree). By a result of Kunen, which is explained in detail by Gitman and Welch \cite[Lemma 6.11]{Gitman:2011pd}, we have that there is a countably closed forcing $\Q$ such that:
	$$\pi: \Q \To \P * \dot T \text{ is a dense embedding.}$$ 
Recall that by \textbf{Proposition \ref{prop:scde}}, since $\Q$ is countably closed and thus subcomplete (indeed, $\Q$ is complete), we have that $\P * \dot{T}$ is thus also subcomplete. (However it is not countably closed, as countably closed forcings aren't preserved under dense embeddings.) In particular, we have a scenario in which $\P * \dot{T}$ is subcomplete but it is not the case that $\forces_{\P} ``\dot{T} \text{ is subcomplete}"$, since as we have seen earlier with \emph{Fact} \ref{fact:Suslinpres}, forcing with a Suslin tree is never subcomplete.
\end{proof} 

For infinite iterations of subcomplete forcing we use revised countable support, which is required to get a sensible iteration theory. If we used only purely countable support and, for example, if Namba forcing was involved in the iteration, then the cofinality of $\omega_2$ would become $\omega$ and reals would be added at stages of cofinality $\omega_2$, making the iteration not subcomplete by \textbf{Proposition~\ref{prop:noreals}}. Indeed, revised countable support has been used to iterate Namba forcing, which is semiproper under $\CH$. Revised countable support iterations were originally invented by Shelah, but the definition worked with by Jensen is due to Donder, and works with iterations of complete Boolean algebras.

\begin{defn}
A sequence $\B=\B_\kappa =\seq{ \B_\alpha }{ \alpha < \kappa }$ is an iteration of complete Boolean algebras so long as $\B_\alpha \subseteq \B_\beta$ for $\alpha \leq \beta <\alpha$ and for limit ordinals $\lambda < \kappa$ we have that $\B_\lambda$ is a complete Boolean algebra generated by $\bigcup_{\alpha<\lambda} \B_\alpha$. By a \textbf{\textit{thread}} in $\B$ we mean a sequence $b=\seq{b_\alpha}{\alpha<\lambda}$ in $\B$ such that $b_\beta \in \B_\beta \setminus \{0\}$ and for $\alpha\leq \beta <\kappa$ we have that $b_\alpha=\Meet \set{a \in \B_\alpha}{b_\beta \leq a}$.

A \textbf{\textit{revised countable support iteration}} ($rcs$ iteration) considers revised countable threads, such that for each thread $b$ and for every limit stage $\lambda \leq \kappa$, we have that either the support of $b$ is bounded in $\lambda$ or the cofinality of the limit stage is collapsed to $\omega$, i.e., there is $\alpha < \lambda$ with $b_\alpha \forces_{\B_\alpha} \cof(\check \lambda) = \check \omega$. The $rcs$ limit is then defined as is done for inverse limits, using $rcs$ threads, and an $rcs$ iteration uses the $rcs$ limit at all limit points.
\end{defn}

I will not give the proof here, but Jensen \cite[Chapter 4]{Jensen:2012fr} shows that the iteration theorem for subcomplete forcing holds, if $rcs$ is used. Here we state it for forcing notions instead of Boolean algebras.

\begin{fact}[\textit{\textbf{Subcomplete Iteration Theorem}}] Let $\P = \P_\kappa = \seq{ ( \P_\alpha, \dot{\Q}_\alpha ) }{ \alpha < \kappa }$ be an \emph{rcs}-iteration such that for all $\alpha < \kappa$: \begin{enumerate}
	\item $\Vdash_{\P_\alpha} \dot{\Q}_\alpha \text{ is subcomplete.}$
	\item $\Vdash_{\P_{\alpha+1}} \delta(\check{\P}_\alpha) \leq \omega_1$
\end{enumerate}
Then $\P$ is subcomplete.
\end{fact}

The following is a consequence of Proposition 16.30 from Jech's \textit{Set Theory} \cite{Jech:2002qr}:

\begin{fact}\label{sciterationbound} Let $\kappa$ be inaccessible. Let $\P = \P_\kappa = \seq{ ( \P_\alpha, \dot{\Q}_\alpha ) }{ \alpha < \kappa }$ be an $rcs$-iteration such that each $\P_\alpha$ has the $\kappa$-$cc$. Let $G_\alpha$ be $\P_\alpha$-generic, and let $G=G_\kappa$. Then if $X \subseteq \kappa$ and $|X| < \kappa$ in $V[G]$, we have that $X \in V[G_\beta]$ for some $\beta < \kappa$.
\end{fact}
\begin{proof}[Proof sketch.]
This is because $\P_\kappa$ satisfies the $\kappa$-chain condition, by a $\Delta$-system argument.
\end{proof}

The following defintion is a kind of blueprint for the sort of iterations we will be using in the following chapter. We will want to force various axioms and principles about subcomplete forcing, which will require the use of a specific large cardinal. The following definition is specifically useful when forcing $\SCFA$ from a supercompact cardinal.

\begin{defn} Suppose $P$ is a proposition about subcomplete forcing that we would like to show is consistent relative to $\ZFC$. Let $\kappa$ be inaccessible. \textbf{\emph{The subcomplete least-counterexample to $\mathbf{P}$ lottery sum $
\mathbf{rcs}$ iteration of length $\pmb{\kappa}$}} is of the form:  $\P = \P_\kappa = \seq{ ( \P_\alpha, \dot{\Q}_\alpha ) }{ \alpha < \kappa }$, a revised countable support iteration of length $\kappa$, defined so that at stage $\alpha$, 
letting $\mathcal P$ be the collection of subcomplete forcing posets in $V^{\P_\alpha}$ of \textit{minimal rank} for which the proposition $P$ fails, we take 
	\[\text{$\P_{\alpha+1} = \P_\alpha * \dot{\Q}_\alpha * \Coll(\omega_1, |\P_\alpha|)$ where $\dot{\Q}_\alpha$ is a term for the lottery sum $\oplus \mathcal P$.} \qedhere \]  
\end{defn}


Due to \textbf{Lemma \ref{lem:deltalottery}} we have that lottery sums of subcomplete forcings are subcomplete, so it follows that subcomplete least-counterexample lottery sum $rcs$-iterations of length $\kappa$ are subcomplete.

\begin{lem}\label{lem:lengthcollapse} Let $\kappa$ be inaccessible. Then nontrivial subcomplete least-counterexample lottery sum $rcs$-iterations of length $\kappa$ collapse $\kappa$ to be $\omega_2$. Moreover, these iterations will force $\CH$. \end{lem}
\begin{proof} It is dense in the nontrivial lottery sum poset to collapse cardinals between $\omega_1$ and $\kappa$ to be $\omega_1$, since this is done for every subcomplete poset of throughout the iteration by definition. 

Meanwhile, $\omega_1$ cannot be collapsed by subcomplete forcing. Subcomplete least-counterexample lottery sum $rcs$-iterations will force $\CH$ to hold in the extension, since at some point $\CH$ will be forced by the collapse forcing, but since no reals are added by subcomplete forcing, this $\CH$ will remain true after that point.
\end{proof}

Many results are comparable to those for proper and semiproper forcing. The difference is that with subcomplete  iterations, we will obviously never choose Cohen forcing - indeed no Cohen reals will ever appear since the whole iteration is itself subcomplete.

\chapter{Axioms about Subcomplete Forcing}
\label{chap:AxiomsaboutSC}
In this chapter we will discuss various axioms about subcomplete forcing. Forcing axioms and bounded forcing axioms are well known. As subcomplete forcing axioms have been already explored elsewhere in Jensen \cite{Jensen:2009wc}, ~\cite{Jensen:2012fr} and Fuchs~\cite{Fuchs:2016yo}, they will not be the focus of our discussion. We dive more deeply into comparing and contrasting the subcomplete maximality principle and the subcomplete resurrection axiom in the next few sections. In particular, we would like to compare the axioms for subcomplete forcing to those for countably closed forcing, in order to further elucidate the differences between properties of these forcing classes.

In the following we will let $\Gamma$ be a class of notions of forcing that is defined by some formula $\psi_\Gamma(x,p)$, where $p$ is a parameter. We will generally refer to the classes that contain trivial forcing and that have a two-step iteration theorem as \textit{reasonable}. Specifically we will be concerned mainly with the class of subcomplete forcing, but we may sometimes be interested in countably closed forcing. Of course one may also consider forcing classes such as $ccc$, properness, semiproperness, etc. 

We shall write $\SCFA$ to stand for the subcomplete forcing axiom $\textsf{FA}_{sc}$, or $\textsf{FA}_\Gamma$ where $\Gamma=\set{ \P }{ \P \text{ is subcomplete} }.$
Much like $\SPFA$ and $\PFA$, the subcomplete forcing axiom $\SCFA$ may be forced from the existence of a supercompact cardinal. Jensen~\cite{Jensen:2012fr} gives a Baumgartner-style argument for this.

\begin{fact}[Jensen] If there is a supercompact cardinal $\kappa$, then there is a forcing extension satisfying $\SCFA$. \end{fact}

As Jensen elaborates, $\SCFA$ has some similarities to $\SPFA$ and $\MM$ since it implies Friedman's Principle (discussed in Section \ref{sec:SCForcingAndTrees}).

Another well-known axiom about forcing is the bounded forcing axiom. Bounded forcing axioms were introduced by Martin Goldstern and Saharon Shelah~\cite{Goldstern:1995zr}. 

\begin{defn}
The \textbf{\emph{Bounded Forcing Axiom}} for a class of forcing notions $\Gamma$, denoted $\textsf{BFA}_{\Gamma}$, states that for any forcing $\P \in \Gamma$ and $\mathcal A$, a collection of maximal antichains of $\P$ such that $|\mathcal A|=\omega_1$ and furthermore for all $A \in \mathcal A$ we have $|A| \leq \aleph_1$, there exists a $\mathcal A$-generic filter on $\P$ -  meaning that there is a filter $G \subseteq \P$ meeting every antichain in $\mathcal A$.
\end{defn}

It is useful to give characterizations of the bounded forcing axiom in terms of generic absoluteness, which has been given by Bagaria \cite[Theorem 5]{Bagaria:2000fj}. The following characterization will be useful to keep in mind throughout the next few sections, since both the resurrection axiom and the maximality principle have connections to different characterizations of the bounded forcing axiom.

\begin{fact}[Bagaria] \label{Fact:BFA} For a class of forcing notions $\Gamma$, the following are equivalent:
\begin{enumerate}
	\item $\BFA_\Gamma$ 
	\item For every $\P \in \Gamma$, $G \subseteq \P$ generic,  $H_{\omega_2} \preccurlyeq_{\Sigma_1} H_{\omega_2}^{V[G]}.$ 
	\item \label{item:GammaGenericSigma1Absoluteness} $\Gamma$-generic $\Sigma_1(H_{\omega_2})$-absoluteness.\footnote{See \textbf{Definition \ref{def:GenericAbsoluteness}}.}
\end{enumerate}	
	\end{fact}

We shall write $\BSCFA$ to stand for $\textsf{BFA}_\Gamma$ where $\Gamma=\set{ \P }{ \P \text{ is subcomplete} }.$

The consistency strength of the bounded subcomplete forcing axiom is exactly that of a reflecting cardinal, shown by Fuchs \cite{Fuchs:2016yo}.

\begin{defn} A regular cardinal $\kappa$ is \textbf{\emph{reflecting}} so long as for any formula $\varphi(x)$ and any $a \in H_\kappa$, if there is a cardinal $\theta > \kappa$ with $H_\theta \models \varphi(a)$, then there is a cardinal $\gamma < \kappa$ with $a \in H_\gamma$ and $H_\gamma \models \varphi(a)$. \end{defn}


\begin{fact}[Fuchs] The axiom $\BSCFA$ is equiconsistent (over $\ZFC$) with the existence of a reflecting cardinal. \end{fact}
A key ingredient in showing that the bounded subcomplete forcing axiom is equiconsistent with a reflecting cardinal is the following fact about the definition of subcomplete forcing,
\begin{fact}[Fuchs] \label{fact:scSigma2} The statement ``$\P$ is subcomplete" is expressible by a $\Sigma_2$-formula. \end{fact}

\section{The Maximality Principle}
\label{sec:MP}
The Maximality Principle (\textsf{MP}) was originally defined by Stavi and V\"a\"an\"anen \cite{Stavi:2001qv} for the class of $ccc$ forcings. Maximality principles were defined in full generality by Hamkins~\cite{Hamkins:2003jk}, and expanded upon for different classes of forcing notions by Fuchs~\cite{Fuchs:2008rt}, \cite{Fuchs:2008ve},  and Leibman~\cite{Leibman:MP}.

One motivation behind maximality principles is their connection to modal logic. In the realm of set theory, as introduced by Hamkins~\cite{Hamkins:2003jk}, one interprets ``possible" as forceable, or true in some forcing extension, and ``necessary" as true in every forcing extension. Maximality principles have a strong connection to the Euclidean frame condition, which posits that if a statement is possible then it is necessarily possible. In modal terms, using the necessary ($\Box$) and possible ($\lozenge$) modal operators, this would state that for a sentence $\varphi$, 
	$$\lozenge \varphi \implies \Box \lozenge \varphi.$$ 
Assuming the modal theory $\textsf{S4}$, which requires some basic conditions on the language, like the duality between $\Box$ and $\lozenge$, by taking the contrapositive, the Euclidean frame condition is equivalent to adding the frame condition stating that if a sentence is possibly necessary, then it is necessary, i.e.: 
	$$\lozenge \Box \varphi \implies \Box \varphi.$$ 
Modulo $\textsf{S4}$, these two frame conditions are equivalent, so I will refer to them interchangeably as the Euclidean frame condition. Adding the Euclidean frame condition to $\textsf{S4}$ yields the modal theory $\textsf{S5}$. Since Hamkins and Leibman were exploring the realm of models of set theory as Kripke models, with the accessibility relation given by the forcing extension relation, it was naturally asked whether $\textsf{S5}$ would be satisfied. It is not hard to see that $\textsf{S4}$ is satisfied in this realm. Moreover, if trivial forcing is available, if a sentence $\varphi$ is necessary, it must be true in the trivial forcing extension, and is therefore true. The Euclidean frame condition was simplified to the following, coined as \textsf{The Maximality Principle} (\textsf{MP}): 
	$$\lozenge \Box \varphi \implies \varphi.$$ 
We have just shown that if the Euclidean frame condition holds in the realm of models of set theory with the forcing extension relation, so does $\textsf{MP}$. In fact, the two axioms are equivalent while considering any class of forcings including trivial forcings. To see why the other direction holds, suppose $\lozenge \Box \varphi$ holds for a sentence $\varphi$, i.e. $\varphi$ may be forced in such a way so as to hold in every further forcing extension. This is equivalent to saying that $\Box \varphi$ can be forced to hold in every forcing extension. Since under $\textsf{S4}$ it is clear that $\lozenge \Box \varphi \implies \lozenge \Box (\Box \varphi)$, and since $\textsf{MP}$ holds, $\lozenge \Box (\Box \varphi) \implies \Box \varphi$, putting this together we arrive at one of our equivalent formulations of the Euclidean frame conditions: $\lozenge \Box \varphi \implies \Box \varphi$ as desired.

We give the precise definition below, where the maximality principle is restricted to a particular class of forcings, thus enabling us to define the axiom restricted to only subcomplete forcings. 

\begin{defn}
Let $\Gamma$ be a class of notions of forcing that is defined by some formula $\psi_\Gamma(x,p)$, where $p$ is a parameter. In cascaded modal operator useages, this $\psi_\Gamma(x,p)$ is to be evaluated in the forcing extensions. 

We say that a sentence $\varphi(a)$ is \textbf{$\Gamma$-forceable} if there is $\P \in \Gamma$ such that for every $q \in \P$, we have that $q \forces \varphi(a)$. In other words, a statement is $\Gamma$-forceable if it is forced to be true in an extension by a forcing from $\Gamma$. 

A sentence $\varphi(a)$ is \textbf{$\Gamma$-necessary} if for all $\P \in \Gamma$ and all $q \in \P$, we have that $q \forces \varphi(a)$. So a sentence is $\Gamma$-necessary if it holds in any forcing extension by a forcing notion from $\Gamma$. If $\Gamma$ contains the trivial forcing then being $\Gamma$-necessary implies that $\varphi(a)$ is true. 

If $S$ is a term in the language of set theory, then $\textsf{MP}_{\Gamma}(S)$ is the scheme of formulae stating that every sentence with parameters from $S$ that is $\Gamma$-forceably $\Gamma$-necessary, i.e., the sentence ``$\varphi(a)$ is $\Gamma$-necessary" is $\Gamma$-forceable, is true.  

Let $\Box \textsf{MP}_{\Gamma}(S)$ be the \emph{necessary form of the principle itself}, stating that $\textsf{MP}_{\Gamma}(S)$ holds in every forcing extension obtained by a forcing notion in $\Gamma$.
\end{defn}
Here we focus on the maximality principle for subcomplete forcing notions. We are especially interested in comparing the subcomplete maximality principle to that for countably closed forcings. The maximality principle for closed forcings, (including countably closed) was explored by Fuchs \cite{Fuchs:2008rt}, and we use many of the same techniques and results used there.

We write $\MPsc$ to stand for $\textsf{MP}_\Gamma$ where $\Gamma=\set{ \P }{ \P \text{ is subcomplete} }$, and $\MPc$ in the case where $\Gamma=\set{ \P }{ \P \text{ is countably closed} }$.
From our previous comments, since both of these classes of forcings $\Gamma$ contain trivial forcing, $\textsf{MP}_\Gamma$ is equivalent to one of our characterizations of the Euclidean frame condition in this context, which would state that every sentence that is $\Gamma$-forceably $\Gamma$-necessary is $\Gamma$-necessary. 
 
Before moving on, it should be noted that since the maximality principle is a scheme, it does not technically make sense to write $\textsf{MP}_\Gamma \implies P$ for some proposition $P$ in the language of set theory. When something like this is written, it should be interpreted as saying instead $\ZFC+\textsf{MP}_\Gamma \proves P$.

First we analyze the parameter set $S$ that may be allowed in the definition. As in the case where we consider $\MPc$ as in \cite{Fuchs:2008rt}, the natural parameter set to use in our case, when considering subcomplete forcing, is $H_{\omega_2}$.
The next lemma follows Fuchs \cite[Theorem 2.4]{Fuchs:2008rt}.

\begin{lem} \label{lem:MPparams}
The subcomplete maximality principle cannot be consistently strengthened by allowing parameters that aren't in $H_{\omega_2}$. In particular, $$\MPsc(S) \implies S \subseteq H_{\omega_2}.$$\end{lem}
\begin{proof}
The point is that for any set $a$, it is $sc$-forceably $sc$-necessary for $a \in H_{\omega_2}$. Indeed, after forcing with $\Coll(\omega_1, \TC{a})$, which is subcomplete since it is countably closed, we have that $a \in H_{\omega_2}$ in the forcing extension. This must remain true in every further forcing extension. Thus, if $\MPsc(\{a\})$ holds, it follows that $a \in H_{\omega_2}$.
\end{proof}

Of course the same proof works for countably closed forcing, since the only forcing considered was the collapse forcing to $\omega_1$, which is also countably closed.

\subsection{Consistency of the Maximality Principle}
\label{subsec:ConMP}
In order to show the consistency of the maximality principle for subcomplete forcing, we use a proof technique that adapts arguments of Hamkins \cite[Lemma 1.22]{Hamkins:2003jk}. Hamkins has described the proof as ``running through the house and turning on all the lights", in the sense that the posets that are forced are those that push ``buttons", sentences that can be ``switched on" and stay on, in all forcing extensions. We will adopt the ``button" terminology in the following, where a \textit{\textbf{button}} in our case is a sentence $\varphi(a)$ that is $sc$-forceably $sc$-necessary.\footnote{Here one might imagine that the buttons are as you find in an elevator or crosswalk, where once the button is pushed it may not be subsequently undone by pushing it again, like an ``on/off" button.}

We show that the subcomplete maximality principle holds in a forcing extension, assuming there is a regular cardinal $\delta$ satisfying $V_\delta \preccurlyeq V$. As Hamkins \cite{Hamkins:2003jk} discusses in detail, the existence of a regular cardinal $\delta$ such that $V_\delta \preccurlyeq V$ is a scheme of formulas sometimes referred to as the ``L\'evy scheme," and is equiconsistent with the statement that ``$\Ord$ is Mahlo": the scheme insisting that every definable closed unbounded class of ordinals contain a regular cardinal. We will also refer to the L\'evy scheme as positing the existence of what we refer to as a \textit{\textbf{fully reflecting}} cardinal.

\begin{thm}\label{thm:forceMP} Let $\delta$ be fully reflecting. Then there is a subcomplete iteration $\P$ of length $\delta$ such that whenever $G \subseteq \P$ is generic, we have that $\delta=\omega_2$ in the extension and: $$V[G] \models \MPsc(H_{\omega_2}).$$ 
\end{thm}
\begin{proof} 
First note that $\delta$ must be inaccessible since it is regular and $V_\delta \preccurlyeq V$. 

Define the $\delta$-length subcomplete lottery sum $rcs$-iteration $\P = \P_\delta$ as follows: for $\alpha < \delta$ let $$\P_{\alpha+1} = \P_\alpha *\dot{\Q}_\alpha*\Coll(\omega_1, |\P_\alpha|),$$ where $\dot{\Q}_\alpha$ is a $\P_\alpha$-name for the lottery sum of all minimal rank posets that force some sentence to be $sc$-necessary. In particular, let $\Phi$ be the collection of sentences $\varphi(a)$ that use parameters from $H_{\omega_2}^{V_\delta^{\P_\alpha}}$ such that: 
	$$\text{$V^{\P_\alpha}_\delta \models $ ``$\varphi(a)$ is $sc$-forceably $sc$-necessary."}$$
So $\Phi$ is the set of all possible buttons available at stage $\alpha$. Define $\dot{\Q}_\alpha = \bigoplus_{\varphi(a) \in \Phi} \mathcal Q$ where: 
	$$\mathcal Q=\set{ \dot {\Q} \in V_\delta^{\P_\alpha} }{ \text{In $V_\delta^{\P_\alpha}$, $\dot{\Q}$ is of least rank such that $\dot{\Q}$ is $sc$ and forces `$\varphi(a)$ is $sc$-necessary.'} }$$ 
Since we will want the full iteration $\P$ to remain relatively small in size and to have the $\delta$-$cc$, notice that here we insist that the parameters for our sentences come from a fragment $H_{\omega_2}^{\P_\alpha}$ of the universe $V_\delta$, up to the stage we are in. We have also ensured that this iteration is always defined, since for any name for a set $\dot a$ in $V_\delta^{\P_\alpha}$ the sentence ``$|\dot a|=\omega_1$" is $sc$-forceably $sc$-necessary.

Now suppose that $G \subseteq \P$ is generic over $V$. Let's see that $V[G] \models \MPsc(H_{\omega_2})$. Assume $\varphi(x)$ and $a \in H_{\omega_2}$ satisfies:
	$$\text{$V[G] \models ``\varphi(a)$ is $sc$-forceably $sc$-necessary."}$$
Let $p \in G$ force that $\varphi(\dot a)$ is $sc$-forceably $sc$-necessary.	

Since $\P$ has the $\delta$-$cc$, at no stage in the iteration could $\delta$ be collapsed. This means that there has to be some stage where the parameter $a$ appears.
This is given by $\textbf{Theorem \ref{sciterationbound}}$: $a$ can be coded as a subset of $\omega_1$ and thus has size less than $\delta$.
So there is some earliest stage in the iteration, say $\alpha < \delta$, where $a$ is in $V[G_\alpha]$ and so that $\alpha$ is larger than the support of the condition $p$, which forced $\varphi(a)$ to be a button. This implies that $\varphi(a)$ is an available button at stage $\alpha$, since after the rest of the subcomplete iteration $\P_{tail}$, $\varphi(a)$ is $sc$-forceably $sc$-necessary, so already $\varphi(a)$ is $sc$-forceably $sc$-necessary at this stage. Since $V_\delta[G_\alpha] \preccurlyeq V[G_\alpha]$, as $\P_\alpha$ is a bounded part of the iteration satisfying $\P_\alpha \in V_\delta$, we have: 
	$V_\delta[G_\alpha] \models ``\varphi(a) \text{ is $sc$-forceably $sc$-necessary"}$ as well.
Moreover, $\varphi(a)$ remains a button past the support of $p$.
Thus it is dense, in $\P$, for $\varphi(a)$ to be ``pushed" at some point after stage $\alpha$, say $\beta$.
In other words, there is a $\beta < \delta$ such that there is a least rank subcomplete $\Q$ forcing $\varphi(a)$ to be $sc$-necessary in $V_\delta[G_\beta]$. Let $H \subseteq \Q$ be generic over $V[G_\beta]$ so that there is some $G_{tail}$ generic for the rest of the forcing satisfying $V[G_\beta][H][G_{tail}]=V[G]$. The sentence $\varphi(a)$ is now $sc$-necessary in $V_\delta[G_\beta][H]$. But then since $V_\delta[G_\beta][H] \preccurlyeq V[G_\beta][H]$, this means that $\varphi(a)$ is $sc$-necessary in $V[G_\beta][H]$, by elementarity. Thus since the rest of the iteration is subcomplete, $\varphi(a)$ is true in $V[G_\beta][H][G_{tail}]$ as desired.

In fact, we can analyze $V[G]$ further to see that since for any name for a set $\dot a$ the sentence ``$|\dot a|=\omega_1$" is $sc$-forceably $sc$-necessary, we have that every element of $V_\delta$ gets collapsed to have size $\omega_1$ - in the next stage, by the definition of the $rcs$ iteration being used. Since $\delta$ isn't collapsed by the iteration (it has the $\delta$-$cc$) it must be that $\delta=\omega_2$ in $V[G]$. If $\CH$ did not already hold at the beginning of the iteration, it will hold by the end, since then the sentence ``$|\dot{\c}| = \omega_1$" is $sc$-forceably $sc$-necessary. 
\end{proof}

As for the other direction of the relative consistency of $\MPsc(H_{\omega_2})$, we have the following result. Here we follow closely the analogous result for other classes of forcing notions, originally due to Hamkins \cite{Hamkins:2014yu}.

\begin{lem}\label{lem:MP->fullyreflectinginL} Assume that $\MPsc(H_{\omega_2})$ holds. Then $L_{\omega_2} \preccurlyeq L$. Thus, in particular, we have that $\omega_2^V$ is fully reflecting in $L$. \end{lem}
\begin{proof}
Assume $\vec{a} \in L_{\omega_2}$ and $L \models \exists z \ \varphi(z,\vec a)$. Consider the sentence: $$``\text{The least ordinal } \gamma \text{ such that there is } b \in L_\gamma \text{ with } \varphi^L(b, \vec a) \text{ has cardinality at most } \omega_1."$$ The parameters for the sentence come from $H_{\omega_2}$ and it is $sc$-forceably $sc$-necessary, since we can force $\gamma$ to have size $\omega_1$, and once the function witnessing this cardinality change is there, it cannot be altered by further forcing without adding a real. So the statement is true, meaning that there is a witness for the existential statement $\exists z \ \varphi(z, \vec a)$ in $L_{\omega_2}$ as desired.
\end{proof}

This implies that $\ZFC+\MPsc(H_{\omega_2})$ is equiconsistent with $\ZFC$ plus ``$\Ord$ is Mahlo". Also, it means that $\MPsc(H_{\omega_2})$ cannot hold in $L$, since then $L$ would think that $\omega_2$ is inaccessible in $L$, a contradiction.

Earlier, the necessary maximality principle was introduced, which posits that the maximality principle itself persists in every forcing extension. Like the situation with countably closed forcing, which is shown in Theorem 3.16 from \cite{Fuchs:2008rt}, this form is not consistent for subcomplete forcing, which we show below.

\begin{prop}\label{prop:Kurepa} Assume $\MPsc(S)$. Then there are no Kurepa trees in $S$. So if $S=H_{\omega_2}$ then there are no Kurepa trees.
\end{prop}
\begin{proof}
Assume that $T \in S$ is a Kurepa tree with $\lambda$-many branches, where $\lambda \geq \omega_2$. Then forcing with $\Coll(\omega_1, \lambda)$ doesn't add branches to $T$, since the forcing is of course countably closed. But now $T$ has $\omega_1$-many branches in the extension, so $T$ is no longer a Kurepa tree in the forcing extension. Indeed, $T$ can never become a Kurepa tree in any extension by a subcomplete forcing since subcomplete forcings cannot add branches to $\omega_1$-trees by \textbf{Theorem \ref{thm:scbranch}}. So $T$ is $sc$-forceably $sc$-necessarily \textit{not} a Kurepa tree. As $T$ is assumed to be in $S$ and is thus allowed as a parameter in $\MPsc(S)$, it follows that $T$ is not a Kurepa tree, a contradiction.
\end{proof}

\begin{prop}\label{prop:necessaryMPinconsistent} $\Box \MPsc(H_{\omega_2})$ is inconsistent with \ZFC. \end{prop}
\begin{proof}
Work in $\ZFC + \Box \MPsc (H_{\omega_2})$. There is a generic extension obtained by subcomplete forcing in which there is a Kurepa tree (since the forcing to add one is countably closed.) In this extension, $\MPsc(H_{\omega_2})$ has to still hold, since we are assuming the necessary form of the principle. But this is a contradiction to \textbf{Proposition \ref{prop:Kurepa}}.
\end{proof}

\subsection{The Local Maximality Principle}
\label{subsec:localMP}
Here we define the local version of the maximality principle, where the truth of a forceably necessary sentence will be checked not in $V$ but in a much smaller structure. This should be compared to one of the equivalent ways of defining the bounded forcing axiom, namely the third characterization given in \textit{Fact} \ref{Fact:BFA}.

\begin{defn} Let $\Gamma$ be a reasonable class of forcing notions, and let $S$ be a set of parameters. Let $M$ be a defined term for a structure to be evaluated in forcing extensions, and $S \subseteq M$. The \emph{\textbf{Local Maximality Principle}} relative to $M$ ($\MP^{M}_{\Gamma} (S)$) is the statement that for every $a \in S$ and every formula $\varphi(x)$, if $\varphi^{M}(a)$ is $\Gamma$-forceably $\Gamma$-necessary, then $\varphi^{M}(a)$ is true. 
\end{defn}

We consider $\lMPsc$, which stands for the local version of $\MPsc(H_{\omega_2})$, where $\Gamma$ is the class of subcomplete forcing notions, and $M$ and $S$ are both $H_{\omega_2}$. As discussed in \textbf{Lemma \ref{lem:MPparams}}, our choice of $H_{\omega_2}$ makes sense for the parameter set for the subcomplete maximality principle, and since the smallest model $M$ that makes sense to use for the local version has to at least contain the parameter set, $H_{\omega_2}$ is a natural choice.

Clearly, the following implication holds:
\begin{lem}\label{lem:globalMPislocal} $\MPsc(H_{\omega_2}) \implies \lMPsc$. \end{lem}
And the following expresses the relationship between the local maximality principles and bounded forcing axioms, here stated for subcomplete forcing.
\begin{prop} $\lMPsc \implies \BSCFA$. \end{prop}
\begin{proof}
Assume that $\lMPsc$ holds. We use characterization \ref{item:GammaGenericSigma1Absoluteness} of $\BSCFA$ from \textit{Fact} \ref{Fact:BFA}. To show that $sc$-generic $\Sigma_1(H_{\omega_2})$-absoluteness holds, let $\varphi(x)$ be a $\Sigma_1$-formula and let $a \in H_{\omega_2}$ and let $\P$ be subcomplete, satisfying $\forces_\P \varphi(\check a)$. Let $G \subseteq \P$ be generic. Since $\varphi(x)$ is $\Sigma_1$ and $H_{\omega_2} \preccurlyeq_{\Sigma_1} V$, we have that $\varphi^{H_{\omega_2}}(a)= \varphi^{H_{\omega_2}}((\check a)^G)$ holds in all future forcing extensions. Thus $\varphi^{H_{\omega_2}}(a)$ is $sc$-forceably $sc$-necessary, which means that $\varphi^{H_{\omega_2}}(a)$ is true (in $V$). Thus $\varphi(a)$ holds in $V$ as desired.
\end{proof}

So in particular, we have that $\MPsc(H_{\omega_2})$ implies $\BSCFA$. What do models of $\lMPsc$ look like? Below we see that models of the local subcomplete maximality principle naturally have Suslin trees, and imply $\CH$. 

\begin{prop} \label{prop:localproperties} The theory $\ZFC + \lMPsc$ proves the following: \begin{enumerate}
	\item There is a Suslin tree.
	\item $\lozenge$ holds.
	\item $\CH$ holds.
\end{enumerate}\end{prop} 
\begin{proof}
First we need to see that $H_{\omega_2}$ is enough to verify each of these properties.

For $\textbf{\textit{1}}$, note that the forcing to add a Suslin tree is countably closed, and thus is subcomplete. But Suslin trees are also preserved by subcomplete forcing, by Fact \ref{fact:Suslinpres}. In fact, any particular Suslin tree will continue to be a Suslin tree after any subcomplete forcing. Thus the existence of a Suslin tree is $sc$-forceably $sc$-necessary, and hence true by $\lMPsc$. 

For $\textbf{\textit{2}}$ note that Jensen \cite[Chapter 3 p.~7]{Jensen:2009wc} shows that $\lozenge$ will hold after performing subcomplete forcing if it held in the ground model (we haven't yet seen that any particular instance of a $\lozenge$-sequence will be preserved). Since forcing to add a $\lozenge$-sequence is countably closed, and it will continue to hold after any subcomplete forcing. Of course then $\textbf{\textit{3}}$ follows, since $\lozenge$ implies $\CH$. 
\end{proof}

We've seen that $\MPsc(H_{\omega_2})$ is consistent with $\ZFC$ and with any reasonable value of $2^{\omega_1}$. In light of \textbf{Lemma \ref{lem:globalMPislocal}}, we have the following result about models of $\MPsc(H_{\omega_2})$.

\begin{cor} $\ZFC + \MPsc(H_{\omega_2})$ prove the following: \begin{enumerate}
	\item There is a Suslin tree.
	\item $\lozenge$ holds.
	\item $\CH$ holds.
\end{enumerate}\end{cor}

\subsection{Separating and Combining the Subcomplete and Countably Closed Maximality Principles}
\label{subsec:SeparatingSCandCCMP}
Up until this point, it looks as though $\MPc$ and $\MPsc$ are very similar; all of the previous results about $\MPsc$ also hold for $\MPc$. Indeed, we have the following, which is a straightforward consequence of a result of Fuchs \cite[Theorem 2.10]{Fuchs:2008rt}:

\begin{fact} \label{fact:ForceMPc} Assume that $\delta$ is regular and $V_\delta \preccurlyeq V$. Then $\MPc(H_{\omega_2})$ holds in $V[G]$, where $G$ is generic for $\Coll(\kappa, <\delta)$. \end{fact}

Perhaps, one might think, it is just the case that the subcomplete and countably closed maximality principles imply each other? As it turns out, this is simply not the case, as we shall show through a sequence of results below. In this section we look at ways in which the principles may consistently be separated, ultimately summed up by the following diagram:

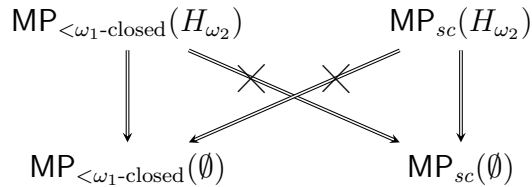
\begin{center}
\begin{tikzpicture}
  \matrix (m) [matrix of math nodes,row sep=3em,column sep=4em,minimum width=2em]{
  	\MPc(H_{\omega_2}) & \MPsc(H_{\omega_2}) \\
     	\MPc(\emptyset) & \MPsc(\emptyset) \\};
  \path[-stealth]
    (m-1-1) edge [double] (m-2-1)
    		edge [double] node[pos=0.3] {$\diagdown$} node[pos=0.3] {$\diagup$} (m-2-2)
    (m-1-2) edge [double] node[pos=0.3] {$\diagdown$} node[pos=0.3] {$\diagup$} (m-2-1)
    (m-1-2) edge [double] (m-2-2);
\end{tikzpicture}
\end{center}

It is easy to see that $\MPsc(H_{\omega_2})$ implies $\MPsc(\emptyset)$, since sentences without parameters are always available when considering sentences that potentially have a larger pool of parameters to draw from. The same of course is true for the maximality principle for countably closed forcing. We shall see below that it is not generally the case that $\MPc$ implies $\MPsc$. However, we will see that it is consistent for $\MPc(H_{\omega_2})$ and $\MPsc(\emptyset)$ to both hold. 

The following two lemmas exploit the fact that subcomplete forcing, unlike countably closed forcing, can add countable sequences. This will be relied upon to elucidate the difference between the countably closed and subcomplete maximality principles. The key tool of our analysis will be Namba forcing. We give the precise definition later in \textbf{Definition~\ref{def:namba}}, but for now let us just state that Namba forcing is a poset that adds a cofinal $\omega$-sequence to $\omega_2$ and is subcomplete if $\CH$ holds.

\begin{lem}\label{lem:scseqnec} It is $sc$-forceably $sc$-necessary that $$\Ord^\omega \neq \Ord^\omega \cap L.$$ \end{lem}
\begin{proof}
This is because after using countably closed forcing, namely $\Add(\omega_1, 1)$, to force $\CH$ if it doesn't already hold, Namba forcing is subcomplete, and adds a countable subset of $\omega_2$, which cannot be removed by further subcomplete forcing.
\end{proof}

\begin{lem}\label{lem:ctbleseqnec} If $V=L$, then it is $<\omega_1$-closed-forceably $<\omega_1$-closed-necessary that $$\Ord^\omega = \Ord^\omega \cap L.$$ \end{lem}
\begin{proof}
This is because countably closed forcings are countably distributive, and don't add countable sequences.
\end{proof}

From these two observational lemmas, we can now say something about what models of the combined maximality principles would have to look like. The following proposition is one more ingredient we will use to show the non-implications in the above diagram:

\begin{prop} \label{prop:L[a]notV} If $\MPsc(H_{\omega_2}) + \MPc(\emptyset)$ holds, then $V \neq L[a]$ for any set $a$. \end{prop}
\begin{proof} Suppose toward a contradiction that $V=L[a]$ for some set $a$ and both $\MPsc(H_{\omega_2})$ and $\MPc(\emptyset)$ hold. Then it is $<\omega_1$-closed-forceably $<\omega_1$-closed-necessary that $V$ is a countably closed forcing extension of a model of the form $L[\overline a]$, where $\overline a$ is a bounded subset of $\omega_2$ - indeed, $\overline a$ can be used as a parameter in $\MPsc(H_{\omega_2})$. To see this, first force with $\Coll(\omega_1, \TC{a})$, then every further countably closed extension will always look as described.
In particular, by $\MPc(\emptyset)$, it follows that $V=L[\overline a][G]$, where $G$ is a generic for a countably closed forcing and $\overline a \subseteq \alpha < \omega_2$ for some ordinal $\alpha$. However, it is also $sc$-forceably $sc$-necessary that there is a countable sequence not in $L[\overline a]$, since we have $\CH$ in $L[\overline a]$ and so Namba forcing is subcomplete, thus Namba forcing over $L[\overline a]$ will give rise to such a sequence. This is a contradiction to $V=L[\overline a][G]$, whch does not contain a countable sequence not in $L[\overline a]$. \end{proof}

In other words,

\begin{cor}\label{cor:L[a]notV} If $V=L[a]$, then it is not the case that $\MPsc(H_{\omega_2})$ and $\MPc(\emptyset)$ both hold. \end{cor}

We will now show both of the crossed out implications in the above diagram.

\begin{thm} \label{thm:ConBFMPscandNotLFMPc} It is consistent for $\MPsc(H_{\omega_2})$ to hold while $\MPc(\emptyset)$ fails. \end{thm}
\begin{proof} Let $\delta$ be fully reflecting in $L$. Force $\MPsc(H_{\omega_2})$ over $L$ to obtain $L[G] \models \MPsc(H_{\omega_2})$. By \textbf{Corollary \ref{cor:L[a]notV}}, this means that $\MPc(\emptyset)$ fails.
\end{proof}

Moreover, if we allow parameters for the countably closed maximality principle but not for the subcomplete maximality principle, it  is possible for the subcomplete maximality principle to fail while the countably closed principle holds, as shown below.

\begin{thm} \label{thm:ConBFMPcandNotLFMPsc} It is consistent for $\MPc(H_{\omega_2})$ to hold while $\MPsc(\emptyset)$ fails. \end{thm}
\begin{proof} We work over $L$, with $\delta$ fully reflecting in $L$. We force over $L$ in order to make $\MPc(H_{\omega_2})$ hold in $L[G]$. Then $L[G]$ has the same countable sequences of ordinals as $L$ does, by \textbf{Lemma \ref{lem:ctbleseqnec}}. Thus $\MPsc(\emptyset)$ fails; in particular, the sentence
	$$\exists\, \vec a \in \Ord^\omega \setminus L$$
is $sc$-forceably $sc$-necessary by \textbf{Lemma \ref{lem:scseqnec}}, but false in $L[G]$.  \end{proof}
However, it is relatively consistent for $\MPc(H_{\omega_2})$ and $\MPsc(\emptyset)$ hold at the same time.

\begin{thm}\label{thm:ConBFMPc+LFMPsc} It is consistent for $\MPc(H_{\omega_2})$ and $\MPsc(\emptyset)$ to both hold. \end{thm}
\begin{proof} Let $L_{\overline \delta} \preccurlyeq L_\delta \preccurlyeq L$, with $\delta$ regular. Force $\MPsc(\emptyset)$ with a $\overline \delta$-length iteration of subcomplete forcings. To do this, the iteration is the same as in \textbf{Theorem \ref{thm:forceMP}}, except simpler since we don't have to worry about parameters for our available sentences. Call the result $L[\G]$. This gives us $L_\delta[\G] \preccurlyeq L[\G]$, since $\overline \delta$ is small relative to $\delta$. Then force $\MPc(H_{\omega_2})$ by forcing with $\P= \Coll(\omega_1, <\delta)$ over $L[\G]$, which works by \textit{Fact} \ref{fact:ForceMPc}.
Letting $G \subseteq \P$ be generic over $L[\G]$, in $L[\G][G]$ we certainly have that $\MPc(H_{\omega_2})$ holds. The question is whether we have somehow killed $\MPsc(\emptyset)$. To see that this does not happen, let $\varphi$ be a sentence without parameters in $L[\G][G]$ that is $sc$-forceably $sc$-necessary. This means that there is a subcomplete forcing $\Q=\dot \Q^G$ in $L[\G][G]$ such that 
	$$L[\G][G] \models ``\Q \text{ is subcomplete" and }``\forces_{\Q}`\varphi \text{ is $sc$-necessary.'}"$$
But this means that there is a condition $p \in G$ such that 
	$$L[\G] \models ``p \forces `\dot \Q \text{ is subcomplete' and }`\forces_{\dot\Q} ``\varphi \text{ is $sc$-necessary.'''}"$$
In other words, letting $\P_{\leq p}$ denote the forcing $\P$ below the condition $p$, 
	$$L[\G] \models ``\forces_{\P_{\leq p}*\dot\Q} `\varphi \text{ is $sc$-necessary.'}"$$
We have that $\P_{\leq p} * \dot \Q$ is subcomplete by the two-step iteration theorem, since also 
	$$L[\G] \models ``\forces_{\P_{\leq p}} `\dot \Q \text{ is subcomplete.'}"$$
Thus, by $\MPsc(\emptyset)$ in $L[\G]$, $\varphi$ is $sc$-necessary. Hence $\varphi$ is true in $L[\G][G]$, since it is a subcomplete forcing extension of $L[\G]$. \end{proof} 

We can generalize part of the above proof to see that $\MPsc(H_{\omega_2})$ will continue to hold in further forcing extensions so long as the extensions are subcomplete and do not add any new parameters. This is equivalent to the claim that $\Box \MPsc(\emptyset)$ is equiconsistent with $\MPsc(\emptyset)$. The analogous result stating that $\Box \MPc(\emptyset)$ is consistent is implied by work of Fuchs \cite[Lemma 4.2]{Fuchs:2008rt}.

However, it remains unclear how to answer the following question:
\begin{question} Is it consistent for $\MPsc(H_{\omega_2})$ and $\MPc(\emptyset)$ to both hold? \end{question}

The reason this would be trickier is because subcomplete forcings are certainly not necessarily countably closed.
Already we see from the result of \textbf{Corollary \ref{cor:L[a]notV}} that the same proof technique we have used say in \textbf{Theorem \ref{thm:ConBFMPcandNotLFMPsc}} does not seem to work in this example.

In fact, we can use our proof of \textbf{Proposition \ref{prop:L[a]notV}} and set-theoretic geology to go a bit further. Set-theoretic geology was explored by Fuchs, Hamkins, and Reitz \cite{Fuchs:2015fy}. The authors make precise the notion of the definability of ground models for models of $\ZFC$. Moreover, Usuba \cite{Usuba:2016pi} has shown that if there are set-many grounds for a model of $\ZFC$, then the intersection of these grounds, called the \textit{mantle}, is itself a ground (in particular, it is a model of $\ZFC$).

\begin{thm} \label{thm:BFMPsc+LFMPcImpliesGrounds} If $\MPsc(H_{\omega_2}) + \MPc(\emptyset)$ holds, then there is a proper class of grounds. \end{thm}
\begin{proof}
Suppose toward a contradiction that there is not a proper class of grounds, and that both $\MPsc(H_{\omega_2})$ and $\MPc(\emptyset)$ hold. By Usuba's work on set-theoretic geology and proof of the downward directed grounds hypothesis \cite{Usuba:2016pi}, this means that $V=\M[G]$, where $\M$ is the intersection of all of the grounds (the mantle). Define the statement $\varphi$ as follows: 

\begin{center} $\varphi$ : ``For some $\alpha < \omega_2$, there is a $g \subseteq \alpha$ such that $V=\M[g][h]$, where $h$ is  generic for a countably closed forcing." \end{center}
Then $\varphi$ is $<\omega_1$-closed-forceably $<\omega_1$-closed-necessary, since we already have that $V= \M[G]$ but by taking $g=G*H$ in the above sentence, where $H \subseteq \Coll(\omega_1, \TC{G})$ is generic over $V=\M[G]$, it is then $<\omega_1$-closed-necessary for $V$ to be of the form $\M[g][h]$ for $h$ generic for a countably closed forcing. 

So without loss of generality we can assume $G$ is like that, ie. let $V=\M[g][h] = \M[G]$. Since $\MPc(\emptyset)$ (or indeed $\MPsc(H_{\omega_2})$ holds) in $V$, we also have $\CH$. Thus Namba forcing, which we shall denote as $\Namba$, is subcomplete. However, $\Namba$ adds a new $\omega$-sequence, call it $S$. In particular, the sentence $\psi$ holds in $V[S]=\M[G][S] = \M[g][h][S]$, where $\psi$ is defined as follows:
\begin{center} $\psi$: ``There is an $\omega$-sequence not in $\M[g].$" \end{center}

Since there is an $\omega$-sequence not in $\M[G]=V$, and $h$ couldn't have added it as $h$ is countably closed. Moreover, the sentence $\psi$ holds in $V$ by $\MPsc(H_{\omega_2})$, since it is $sc$-forceably $sc$-necessary. But this means that $\psi$ is true in $\M[G]=\M[g][h]$, a contradiction as $\M$ is forcing invariant and since $g$ and $h$ come from countably closed forcing that don't add new $\omega$-sequences.
\end{proof}

This immediately yields the following corollary, in light of Usuba's results.

\begin{cor}\label{cor:hyperhuge->not(BFMPsc+LFMPc)} If there is a hyper-huge cardinal, then $\MPsc(H_{\omega_2}) + \MPc(\emptyset)$ fails. \end{cor}
\begin{proof}
By \cite[Theorem 1.6]{Usuba:2016pi}, once is a hyper-huge cardinal there are only set many ground models. This contradicts \textbf{Theorem \ref{thm:BFMPsc+LFMPcImpliesGrounds}}. \end{proof}

To bring further doubt that the lightface version of the countably closed maximality principle and the boldface subcomplete maximality principle can be combined, we have the following result comparing the local subcomplete and countably closed maximality principles.

\begin{thm} \label{thm:LMPsc+LMPcImpliesSharpClosure} If $\lMPsc$ and $\lMPc$ both hold, then $V$ is closed under sharps. \end{thm}
\begin{proof}
	Suppose toward a contradiction that $\lMPsc$ and $\lMPc$ both hold and $V$ is not closed under sharps; in particular  suppose there is a set of ordinals $X$ such that $\neg X^\sharp$. Then we first claim that there is a $Y$ such that $L[Y]$ contains all countable sets of ordinals.
	\begin{claimno} \label{claim:L[Y]ContainsCtbleSets} There is $Y \subseteq \Ord$ such that $[\Ord]^\omega \subseteq L[Y]$. \end{claimno}
	\begin{proof}[Pf.]
		Let $\tilde X$ code all of the countable subsets of $\omega_2$, so that $[\omega_2]^\omega \subseteq L[\tilde X]$. Let $Y = X \oplus \tilde X$. Then $Y$ is as desired. To see this, let $a \in [\Ord]^\omega$. By $\neg X^\sharp$, Jensen's Covering Lemma holds in $L[X]$, so let $b \in L[X]$ with $a \subseteq b$, where $b$ is a set of ordinals satisfying $|b| \leq \omega_1$. Let $f\in L[X]$ be order preserving, satisfying 
	$$f:b \longrightarrow \otp(b) < \omega_2^V,$$
and let $\overline a = f``a$. Then $\overline a \in L[ \tilde X]$. Thus both $\overline a, f \in L[Y]$ and so $a = f^{-1} ``\overline a \in L[Y]$ as desired.
	\end{proof}
Moreover we claim that the statement of \textit{Claim} \ref{claim:L[Y]ContainsCtbleSets} will necessarily hold in the $H_{\omega_2}$ of any countably closed forcing extension after a suitable collapse forcing.
\begin{claimno} It is $<\omega_1$-closed-forceably $<\omega_1$-closed-necessary	 that \begin{equation} \label{equation:H_omega2modelsclosedunderctbleords} H_{\omega_2} \models \text{ ``there is a set $Y$ such that every countable set of ordinals is in $L[Y]$"} \end{equation}	
\end{claimno}
\begin{proof}[Pf.]
	Let $G \subseteq \Coll(\omega_1, \sup{Y})$ be generic and let $\Q$ be a countably closed poset. Let $H \subseteq \Q$ be generic over $V[G]$. We have to show (\ref{equation:H_omega2modelsclosedunderctbleords}) in $V[G][H]$.
	
	The main point is that $Y \in H_{\omega_2}^{V[G][H]}$. Moreover let $\kappa = \omega_2^{V[G][H]}$, and let $a \in [\kappa]^\omega$ in $V[G][H]$ . Then since countably closed forcing doesn't add countable sets of ordinals, $a \in V$. Thus $a \in L[Y]$, so $a \in L_\kappa[Y]$ as both $a$ and $Y$ are bounded subsets of $\kappa$, which is a cardinal in $L[Y]$. But $L_\kappa[Y]=(L[Y])^{H_{\omega_2}^{V[G][H]}}$. \end{proof}
	
	By $\lMPc$, let $Y \subseteq \omega_2$ satisfy (\ref{equation:H_omega2modelsclosedunderctbleords}) in $V$. Perform  Namba forcing, $\Namba$, and let $g \subseteq \Namba$ be generic. Then 
	$$H_{\omega_2}^{V[g]} \models [\Ord]^\omega \nsubseteq L[Y],$$ 
and this persists to further forcing extensions. So it is $sc$-forceably $sc$-necessary, and thus true in $V$. So 
	$$H_{\omega_2} \models [\Ord]^\omega \nsubseteq L[Y]$$
a contradiction to \textit{Claim} \ref{claim:L[Y]ContainsCtbleSets}.	
\end{proof}

\subsection{More on the Modal Logic of Subcomplete Forcing}
\label{subsec:ModalLogicSC}
While we are using maximality principles to compare the classes of countably closed forcing and subcomplete forcing, one might wonder whether the property of a poset being subcomplete is itself $sc$-forceably $sc$-necessary. Certainly, the property of being countably closed is $<\omega_1$-closed-forceably $<\omega_1$-closed-necessary, by the following basic fact:

\begin{fact} If $\P$ and $\Q$ are countably closed, then $\forces_\P ``\check \Q \text{ is countably closed."}$ \end{fact}

However, the analogous statement for proper forcings is not even true. Letting $\mathbb T$ be a poset that forces with a Suslin tree (in other words $\mathbb T$ is a Suslin tree) and $\mathbb S$ is the canonical $ccc$ specialization forcing, then $\mathbb T \times \mathbb S$ is not proper, since after adding a cofinal branch the specialization function will collapse $\omega_1$. Indeed, both $\mathbb T$ and $\mathbb S$ are $ccc$ so this very argument shows that the analogous statement is not true for $ccc$ forcings either. In the case of subcompete forcing, we cannot perform a similar argument using Suslin trees and $ccc$ specialization forcing since forcing with them is not subcomplete by \textit{Fact} \ref{fact:Suslinpres} and \textbf{Corollary \ref{cor:CCCnotSC}}. Instead we will work with Namba forcing.

\begin{defn} \label{def:namba} \textit{\textbf{Namba forcing}}, denoted by $\Namba$, a forcing notion consisting of subtrees $T \neq \emptyset$ of $\omega_2^{<\omega}$ ordered by inclusion, such that $T$ is downward closed in $\omega_2^{<\omega}$ and where each node in $T$ has $\omega_2$-many eventual successors in $T$. \end{defn}

Each condition in $\Namba$ has size $\omega_2$. Namba forcing adds a cofinal sequence $S: \omega \to \omega_2^V$ to the extension, a cofinal branch through $\omega_2^{<\omega}$, so there is no way it can be countably closed or proper. We will refer to such $S$ as a Namba sequence. 
Under $\CH$, Namba forcing adds no new reals. Furthermore Namba forcing is subcomplete, which Jensen shows in \cite[Section 3.3]{Jensen:2012fr}. 

Jensen \cite[Appendix: Lemma 1]{Jensen:2009cq} shows the following fact: 
\begin{fact}[Jensen]\label{fact:Namba} Let $S$ be a Namba sequence. Let $S' \in V[S]$ be a cofinal $\omega$-sequence in $\omega_2^V$. Then $S'$ is a Namba sequence and $V[S']=V[S]$. \end{fact}

\begin{thm}
The forcing $\Namba \times \Namba$ adds a new real, and is thus not subcomplete.
\end{thm}
\begin{proof}
I will think of $\Namba \times \Namba = \Namba * \check \Namba$ as an iteration. Let $S_1:\omega \to \omega_2^V$ be the first cofinal Namba sequence added by $\Namba$, and make sure that $S_1$ increasing, which we can do by the Namba fact (\textit{Fact} \ref{fact:Namba}). Let $S_2:\omega \to \omega_2^V$ be the second Namba sequence. In $V[S_1]$ we may define a function $f: \omega_2^V \To \omega$ by letting 
	$$\text{$f(\alpha) =$ the least $m <\omega$ such that $S_1(m) \geq \alpha$.}$$
Note that $f(S_1(n))=n$ since $S_1$ is increasing, and $f$ is surjective since $S_1$ has domain $\omega$. We would like to compare $S_1$ and $S_2$ by tracking how $S_2$ threads through the tree labeled by $f$. In particular, define the real $r: \omega \to \omega$ in $V[S_1][S_2]$ by the following: 
	$$r(n) = f(S_2(n)).$$
On each level, the idea is to have the real $r$ pick out the levels where $S_1(n) \geq S_2(n)$. 

It must be that $r$ is new. To see why, let $c: \omega \to \omega$ be an arbitrary ground model real. For any Namba condition $T$, there has to be a level, say $n < \omega$, with $\omega_2$-many nodes (otherwise, there is no way the condition could have size $\omega_2$). Moreover for any $\beta < \omega_2$, it is dense for conditions $T$ to satisfy that there is an $n$ such that
	$T \forces \dot{S}_2(n) > \beta,$ 
where $\dot S_2$ is the name for the generic sequence $S_2$. This is because it is always dense for Namba conditions to so-to-speak veer to the right; it is dense for some node in a condition to have a value greater than any fixed $\beta < \omega_2$, again since conditions must have size $\omega_2$. So the set of ground-model Namba conditions that force that for some $n$,
	$f(\dot S_2(n)) > c(n)$, is dense. 
Indeed, letting $T \in \Namba$, we may find a strengthening $T^*$ forcing that the least $m$ such that $S_1(m) \geq \dot S_2(n)$ is larger than $c(n)$. To do this, find a level of $T$ with $\omega_2$-many nodes, say level $\alpha$. Let $T^*$ strengthen $T$ so that $T^*_\alpha= T_\alpha \setminus S_1(c(n))$, restricting level $\alpha$ of $T$ to a tail of $\omega_2$ thus forcing $\dot S_2(n)>S_1(c(n))$ and indeed $S_1(m)$ for all $m<c(n)$. Thus below $T^*$, the least $m$ such that $S_1(m)$ exceeds $\dot S_2(n)$ must be larger than $c(n)$, as desired.
Thus $r$ is new, since $r=(f\circ \Gamma)^{S_2}$ differs from any arbitrary ground-model real somewhere.
\end{proof}

\begin{cor} Let $\Namba$ name a parameter for ground-model Namba forcing. We have that $\forces_\Namba ``\check \Namba \text{ is not subcomplete.}"$ So $``\Namba \text{ is not subcomplete}"$ is $sc$-forceable. \end{cor}

Even after countably closed forcing, it is possible for the ground model's version of Namba forcing to fail to be subcomplete anymore.

\begin{thm}
Forcing with $\Coll(\omega_1, \omega_2) \times \Namba$ adds a new real and is thus not subcomplete.
\end{thm}
\begin{proof} 
Let $S:\omega \to \omega_2^V$ be a cofinal Namba sequence added by $\Namba$. Let $\P = \Coll(\omega_1, \omega_2)$, as defined in $V$. Let $G \subseteq \P$ be generic. 
Thus in $V[G]$, there is a cofinal, normal function $f:\omega_1^V \longrightarrow \omega_2^V$. Thus in $V[G][S]$ we may define a cofinal function $g: \omega \longrightarrow \omega_1^V$ by letting 
$$g(n) = \text{ the least } \alpha < \omega_1 \text{ such that } f(\alpha) > S(n).$$ 
To see that $g$ is cofinal, let $\beta < \omega_1$. Then there is $n<\omega$ such that $S(n) > f(\beta)$, since both $S$ and $f$ are cofinal. But again, since $f$ is cofinal, there is $\alpha < \omega_1$ such that $f(\alpha)>S(n)>f(\beta)$; indeed we may take the least such $\alpha$ satisfying $f(\alpha) > S(n)$. But by normality of $f$, we have that $g(n)=\alpha>\beta$, as desired.

So in $V[G][S]$, we have that $\omega_1^V$ is collapsed to $\omega$ as witnessed by $g$. Thus a real has been added as desired. (So we may see that in the extension, there is a bijection $h:\omega \cong \omega_1$. Now define a real $r$ by letting 
	$$r = \set{ (m,n) }{ h(m) < h(n) } \subseteq \omega \times \omega.$$ 
Then $r$ codes $h$ since each $h(n)$ is an ordinal, and so it is determined exactly by its order-type, in $V[G][S]$.)
\end{proof}

Of course in the above theorem, we have also shown that $\Coll(\omega_1, \omega_2) \times \Namba$ collapses $\omega_1$, so it is no longer subproper either and doesn't preserve stationary subsets of $\omega_1$.

\begin{cor} We have that $\forces_\Namba ``\check \P \text{ is not subcomplete.}"$ So the statement ``\,$\Coll(\omega_1, \omega_2)$ is not subcomplete" is $sc$-forceable, where $\Coll(\omega_1, \omega_2)$ is a parameter for the ground model's version of the collapse forcing. \end{cor}

\begin{cor} Let $\P$ be a parameter for a poset in $V$. The sentence ``\,$\P$ is subcomplete" is not $sc$-forceably $sc$-necessary. \end{cor} 
\begin{proof} Above we gave examples of subcomplete forcing notions that necessarily force that another subcomplete forcing is not subcomplete. \end{proof}

The following theorem is in contrast with the previous results. It states that if a ground-model forcing is $sc$-forceably subcomplete, then it is subcomplete. 

\begin{thm}
Suppose both $\P, \Q$ are forcing notions in $V$ satisfying $\delta(\Q) \geq \delta(\P)$, such that $\P$ is subcomplete and $\forces_{\P} ``\check{\Q}$ is subcomplete." Then $\Q$ is subcomplete (in $V$).
\end{thm}
\begin{proof} 
Let $\P$ be subcomplete and suppose $\forces_{\P} ``\check{\Q}$ is subcomplete." We may assume that $\delta(\P) = \delta(\Q) = \delta$, since we can always increase the size of $\delta(\P)$ by replacing $\P$ with an appropriately sized lottery sum, as is seen in $\textbf{Lemma \ref{lem:deltalottery}}$. To show that $\Q$ is subcomplete, let $\theta$ verify the subcompleteness of $\P$ and that $\check \theta$ verifies the subcompleteness of $\check \Q$ in $V^\P$, so that we are in the standard setup where we have an embedding $\sigma \in V$ as follows: \begin{itemize} 
	\item $\P, \Q \in H_\theta \subseteq N = L_\tau[A] \models \ZFC^-$ where $\tau>\theta$ and $A \subseteq \tau$
	\item $\sigma: \N \cong X \preccurlyeq N$ where $X$ is countable and $\overline N$ is full
	\item $\sigma(\overline \theta, \overline{\Q}, \overline{\P}, \overline s) = \theta, \Q, \P, s$ for some $s \in N$.
\end{itemize}

Toward showing subcompleteness, let $\overline H$ be $\overline{\Q}$-generic over $\N$. We would like to find a generic $H \subseteq \Q$ such that in $V[H]$ there is an embedding $\sigma': \N \preccurlyeq N$, where $\N$ and $N$ are the same as above, and where $\sigma'$ is sufficiently similar to $\sigma$ but satisfying that $\sigma'$ has a lift in $V[H]$, i.e., $\sigma' ``\H \subseteq H$.

Let $\G \subseteq \overline{\P}$ be generic over $\N$. Of course $\N$ is countable, so we may build such a generic in $V[\H]$. Then $\G$ is also generic over $\N[\H]$. By subcompleteness of $\P$, we may let $G \subseteq \P$ be generic over $N$ so that we have that there is $\sigma_0 \in V[G]$ satisfying: \begin{itemize}
	\item $\sigma_0: \N \prec N$
	\item $\sigma_0(\overline \theta, \overline{\Q}, \overline{\P}, \overline s)=\theta, \Q, \P, s$
	\item $\sk{N}{\delta}{\sigma_0} = \Sk{N}{\delta}{X}$
	\item $\sigma_0 ``\, \G \subseteq G$.
\end{itemize}
As the last bullet tells us that $\sigma_0$ lifts in $V[G]$, let $\sigma^*_0: \N[\G] \prec N[G]$ denote the lift of $\sigma_0$, which is also an elementary embedding in $V[G]$. 
In $V[G]$ we have that $\Q$ is subcomplete by our assumption. Furthermore, in $V[G]$ we have that 
\begin{itemize} 
	\item $\Q \in H^{V[G]}_\theta \subseteq N[G] = L_\tau[A][G] \models \ZFC^-$
	\item $\sigma_0^*: \N[\G] \cong \ran(\sigma_0^*) \preccurlyeq N[G]$ where $\ran(\sigma_0^*)$ is countable and $\N[\G]$ is full
	\item $\sigma_0^*(\overline \theta, \overline{\Q}, \overline s) = \theta, \Q, s$.
\end{itemize}
We have fullness of $\N[\G]$ since $\N$ is full: indeed, as $\N$ is transitive, so is $\N[\G]$, and additionally we have that there is a $\gamma$ such that $\N$ is regular in $L_\gamma(\N) \models \ZFC^-$ by fullness of $\N$. But this clearly means that $\N[\G]$ is regular in $L_\gamma(\N[\G]) \models \ZFC^-$.

So by subcompleteness of $\Q$ in $V[G]$ let $H \subseteq \Q$ be generic over $N[G]$ so that there is $\sigma_1 \in V[G][H]$ satisfying: \begin{itemize}
	\item $\sigma_1: \N[\G] \prec N[G]$
	\item $\sigma_1(\overline \theta, \overline{\Q}, \overline s)=\theta, \Q, s$
	\item $\sk{N[G]}{\delta}{\sigma_1} = \sk{N[G]}{\delta}{\sigma_0^*}$
	\item $\sigma_1 `` \H \subseteq H$.
\end{itemize}
Note that here we take $H$ generic over $N[G]$, which implies that $H$ is generic over $N$, and furthermore $H \subseteq N$ as $\Q \in N$.

Toward defining the embedding we ultimately need, let $\overline \sigma=\sigma_1 \rest \N$.

\begin{claimno} \label{claim:BarSigmaGood}The map $\overline \sigma \in V[G][H]$ satisfies: \begin{enumerate}
	\item $\overline \sigma: \N \prec N$ 
	\item $\overline \sigma(\overline \theta, \overline{\Q}, \overline s)=\theta, \Q, s$
	\item $\sk{N}{\delta}{\overline \sigma} = \Sk{N}{\delta}{X}$
	\item $\overline \sigma `` \H \subseteq H$.
\end{enumerate} \end{claimno}
\begin{proof}[Pf.] In order to see \textbf{\textsl{1}}, let's say that $\varphi[\overline \sigma(\overline a)]$ holds in $N$. Then $\varphi[\overline a]$ holds in $\overline N$, since $\sigma_1$ is elementary.
Also \textbf{\textsl{2}} is clear, since $\overline \theta$, $\overline\Q$, and $\overline s$ are all elements of $\N$. Item \textbf{\textsl{4}} must hold, since $\H \subseteq \N$ and since $H$ is a subset of $N$. All that is left to be shown is item \textbf{\textsl{3}}: that $$\sk{N}{\delta}{\overline{\sigma}} = \sk{N}{\delta}{\sigma}.$$
We already have that 
	\begin{equation} \label{eqn:SkSigma1=SkSigma0*} \sk{N[G]}{\delta}{\sigma_1} = \sk{N[G]}{\delta}{\sigma_0^*} \end{equation}
and
	\begin{equation} \label{eqn:SkSigma0=SkX} \sk{N}{\delta}{\sigma_0} = \Sk{N}{\delta}{X}. \end{equation}
The following hold by \textbf{Lemma \ref{lem:Ctrick}}:
 \begin{enumerate}[label=(\roman*)]
	\item $\sk{N}{\delta}{\overline \sigma} =\sk{N[G]}{\delta}{\sigma_1} \cap N$
	\item $\sk{N}{\delta}{\sigma_0}=\sk{N[G]}{\delta}{\sigma_0^*} \cap N$.
\end{enumerate}
Item \textbf{\textsl{3}} follows:
\begin{align*}
	\sk{N}{\delta}{\overline \sigma} \ = & \ \ \sk{N[G]}{\delta}{\sigma_1} \cap N && \text{by (i)}\\
							= & \ \ \sk{N[G]}{\delta}{\sigma_0^*} \cap N  && \text{by (\ref{eqn:SkSigma1=SkSigma0*})} \\
							= & \ \ \sk{N}{\delta}{\sigma_0} && \text{by (ii)}\\
							= & \ \ \Sk{N}{\delta}{X}  && \text{by (\ref{eqn:SkSigma0=SkX})}
\end{align*}
This finishes the proof of \textit{Claim} \ref{claim:BarSigmaGood}.
\end{proof}
Although $\overline \sigma$ has all of the properties that we desire, $\overline \sigma$ is alas only in $V[G][H]$. We need such an embedding to exist in $V[H]$, in order to show that $\Q$ is really subcomplete in $V$ not just in $V[G]$. To find the required embedding, we shall use Barwise theory, and show that, roughly speaking, the existence of such an embedding is consistent. Using Barwise Completeness (\textit{Fact} \ref{fact:completeness}), we will obtain our desired embedding. First we shall define the infinitary language we will be using to analyze consistency of the existence of the embedding. Refer to Section \ref{sec:BarwiseTheory}, especially \textbf{Definition \ref{def:InTheoriesAndBasicAxioms}}, for a review of the terminology (such as $\in$-theories and the \textsf{Basic Axioms}).

Let $\mathfrak M$ be an admissible structure. We define the infinitary $\in$-theory $\mathcal L(\mathfrak M)$ as follows: 
\begin{description}
	\item[predicates] $\in$ 
	\item[constants] $\mathring{\sigma}, \underline x$ for $x \in \mathfrak M$
	\item[axioms] \begin{itemize} \item $\ZFC^-$ and \textsf{Basic Axioms}
		\item $\mathring \sigma : \underline\N \prec \underline N$
		\item $\mathring{\sigma}(\underline{\overline \theta}, \underline{\overline{\Q}}, \underline{\overline s})=\underline{\theta}, \underline{\Q}, \underline{s}$
		\item $\sk{\underline N}{\underline{\delta}}{\mathring \sigma} = \sk{\underline N}{\underline \delta}{\underline \sigma}$
		\item $\mathring{\sigma}``\underline{\overline H} \subseteq \underline H$
	\end{itemize}
\end{description}
This $\in$-theory is $\Sigma_1(\mathfrak M)$, as the only axioms that are $\mathfrak M$-$re$ are the \textsf{Basic Axioms}, everything else is $\mathfrak M$-finite.

Let $\mu$ be regular in $V[H]$ with $N \in H_\mu$, so that 
	$$M= \langle H_\mu; N, H, \theta, \Q, \delta, s; \sigma \rangle \text{ is an admissible structure.}$$
In fact, by \textit{Claim} 1 we have that $\langle M, \overline \sigma \rangle$ models the consistency of $\mathcal L(M)$; that is to say that letting $\mathfrak M$ be $M$ and $\mathring \sigma$ be $\overline \sigma$, the axioms will be satisfied, since all of hte required constants are then available to $M$. So the theory $\mathcal L(M)$ is consistent by Barwise Correctness (\textit{Fact} \ref{fact:correctness}). However this is still in $V[G][H]$ that we have $\overline \sigma$ available. In order to produce a model in $V[H]$ we shall use Barwise Completeness, which requires a countable admissible structure in the place of $M$. In fact, using the countable Mostowski collapse of $M$ works.

Take $\pi: \tilde M \prec M$ where $\tilde M$ is countable transitive. Then the corresponding theory $\mathcal L(\tilde{M})$ is a consistent, $\Sigma_1(\tilde M)$ $\in$-theory. In other words, using $\tilde M$ as our source of special constants is consistent, since any inconsistency could be mapped via $\pi$ to be an inconsistency for $\mathcal L(M)$. Thus by Barwise Completeness , we have a solid model 
	$\tilde{\mathfrak A} = \langle \tilde{\mathfrak A}, \mathring \sigma^{\tilde{\mathfrak A}} \rangle $
in $V$ such that there is agreement between $\tilde M$ and $\mathfrak A$ on the ordinals.
	$$\Ord \cap \wfc(\tilde{\mathfrak A})=\Ord \cap \tilde M.$$ 
We shall show that $\sigma'=\pi \circ \mathring \sigma^{\tilde{\mathfrak A}} \in V[H]$ is as desired. 
\begin{claimno} The map $\sigma' \in V[H]$ satisfies: \begin{enumerate}
	\item $\sigma': \N \prec N$
	\item $\sigma'(\overline \theta, \overline{\Q}, \overline s)=\theta, \Q, s$
	\item $\sk{N}{\delta}{\sigma'} = \Sk{N}{\delta}{X}$
	\item $\sigma' `` \H \subseteq H$.
\end{enumerate} \end{claimno}
\begin{proof}[Pf] We shall use the agreement between $\tilde{\mathfrak A}$ and $\tilde M$ on the special constants of $\tilde M$ and on the ordinals.
In order to see \textsl{\textbf{1}}, suppose that $\varphi[\sigma'(\overline a)]$ holds in $N$. So $\varphi[\pi(\mathring \sigma^{\tilde{\mathfrak A}}(\overline a))]^N$ holds in $M$. Thus $\varphi[\mathring \sigma^{\tilde{\mathfrak A}}(\overline a)]^{\pi^-1(N)}$ holds in $\tilde M$, and thus also in $\tilde A$. Indeed we know that $\underline{\N}^{\tilde{\mathfrak A}} = \N$, since $\tilde M$ has the correct interpretation of $\N$ as $\N$ is transitive and countable, and thus $\tilde{\mathfrak A}$ has the correct interpretation of $\N$ as well. This means that $\varphi[\overline a]$ holds in $\N$, as desired.
To see \textbf{\textsl{2}}, we have $\sigma'(\overline \theta, \overline{\Q}, \overline s)= \pi(\underline{\theta}^{\tilde{\mathfrak A}}, \underline{\Q}^{\tilde{\mathfrak A}}, {\underline s}^{\tilde{\mathfrak A}})=\theta, \Q, s$.
For item \textbf{\textsl{3}}, let $\tilde N = \underline N^{\tilde{\mathfrak A}}$, $\tilde \delta = \underline{\delta}^{\tilde{\mathfrak A}}$, and $\tilde \sigma = \underline{\sigma}^{\tilde{\mathfrak A}}$. So $\pi(\tilde \delta)= \delta$ and $\pi(\tilde \sigma)=\sigma$. We already have that 
	\begin{equation} \label{eqn:SkRingSigma=SkTildeSigma} \sk{\tilde N}{\tilde \delta}{\mathring \sigma^{\tilde{\mathfrak A}}} = \sk{\tilde N}{\tilde \delta}{\tilde \sigma}. \end{equation} 

	We show that $\sk{N}{\delta}{\sigma'} = \Sk{N}{\delta}{X}$.
	
To see $\sk{N}{\delta}{\sigma'} \subseteq \Sk{N}{\delta}{X}$, suppose $x \in \sk{N}{\delta}{\sigma'}$. Then we have that $N$ sees that there is some formula $\varphi$, $\overline z \in \N$, and $\xi < \delta$ where $x$ is unique such that $\varphi(x, \pi(\mathring \sigma^{\tilde{\mathfrak A}}(\overline z)), \xi)$. In particular, $x$ is in the range of $\pi$. Thus $\tilde x = \pi^{-1}(x) \in \tilde N$ and $\tilde \xi < \tilde \delta$ such that $\varphi(\tilde x, \mathring \sigma^{\tilde{\mathfrak A}}(\overline z), \tilde \xi)$ holds. Thus by (\ref{eqn:SkRingSigma=SkTildeSigma}), we have that $\tilde x$ is unique such that 
$\varphi(\tilde x, \tilde \sigma(\overline y), \tilde \zeta)$ for some $\overline y \in \N$ and $\tilde \zeta < \tilde \delta$. Thus pushing back up through $\pi$, letting $\zeta = \pi(\tilde \zeta)$, we have that $x = \pi(\tilde x)$ is unique such that $\varphi(x, \sigma(\overline y), \zeta)$ holds, so $x \in \Sk{N}{\delta}{X}$.

To see that $\Sk{N}{\delta}{X} \subseteq \sk{N}{\delta}{\sigma'}$ works similarly. Let $x \in \Sk{N}{\delta}{X}$. Then there is $\overline z \in \N$ and $\xi < \delta$ such that $\varphi(x, \sigma(\overline z), \xi)$ holds. So in particular, $x$ is in the domain of $\pi$. So we may find $\tilde \xi < \tilde \delta$ such that $\tilde x = \pi^{-1}(x)$ is unique such that $\varphi(\tilde x, \tilde \sigma(\overline z), \tilde \xi)$. So by (\ref{eqn:SkDotSigma=SkTildeSigma}), we have that there is $\overline y \in \N$ and $\tilde \zeta < \tilde \delta$ such that $\tilde x$ is unique satisfying $\varphi(\tilde x, \mathring{\sigma}^{\mathfrak A}(\overline y), \tilde \zeta)$. Finally, by pushing back up through $\pi$, letting $\pi(\tilde \zeta) = \zeta$, we have that $x = \pi(\tilde x)$ is unique such that $\varphi(x, \sigma'(\overline y), \zeta)$, so $x \in \sk{N}{\delta}{\sigma'}$ as desired.
	
To see item \textbf{\textsl{4}}, let $\overline h \in \overline H$. Then $\tilde M$ sees that $\tilde \sigma(\overline h) \in \tilde H$. Thus $M$ sees that $\sigma'(\overline h) \in H$ as desired.
This completes the proof of \textit{Claim} 2.
\end{proof}
So $\sigma' \in V[H]$ certifies the subcompleteness of $\Q$, completing the proof of the theorem.
\end{proof}


\section{The Resurrection Axiom}
\label{sec:RA}
Joel David Hamkins and Thomas Johnstone \cite{Hamkins:2013qv} first introduced the notion of resurrection in set theory. The idea behind resurrection axioms is to look at the model-theoretic concept of existential closure in the realm of forcing, because, as Hamkins and Johnstone point out, the notions of resurrection and existential closure are tightly connected in model theory. In model theory, a submodel $\mathcal M \subseteq \mathcal N$ is \textit{existentially closed} in $\mathcal N$ if existential statements in $\mathcal N$ using parameters from $\mathcal M$ are already true in $\mathcal M$, i.e., $\mathcal M$ is a $\Sigma_1$-elementary substructure of $\mathcal N$. Many forcing axioms can be expressed informally by stating that the universe is existentially closed in its forcing extensions, since forcing axioms posit that generic filters, which normally exist in a forcing extension, exist already in the ground model. It turns out that in set theory the appropriate notion of resurrection implies the truth of its associated forcing axiom, but not the other way around. So what is obtained with resurrection axioms is described by Hamkins and Johnstone as a more ``robust" formulation of forcing axioms for various forcing classes.

\begin{defn} Let $\Gamma$ be a fixed, definable class of forcing notions. Let $\tau$ be a term for a cardinality to be computed in various models; e.g. $\c$, $\omega_1$, etc. The \textbf{\emph{Resurrection Axiom}} $\textsf{RA}_\Gamma(H_\tau)$ asserts that for every forcing notion $\Q \in \Gamma$ there is a further forcing $\dot{\R}$ with $\forces_{\Q} \dot{\R}  \in \Gamma$ such that if $g*h \subseteq \Q * \dot{\R}$ is $V$-generic, then $$H_{\tau}^V \preccurlyeq H_{\tau}^{V[g*h]}.$$
Just like with the maximality principle, we will write $\RAsc$ to denote the resurrection axiom for the class of subcomplete forcing notions, and $\RAc$ for the class of countably closed forcing notions. 

We may define $\Box \RAsc(H_\tau)$ as the \textit{necessary form of the principle itself}, asserting that $\RAsc(H_{\tau})$ holds in every forcing extension obtained by subcomplete forcing, with $H_{\tau}$ interpreted in the extension.
\end{defn}

Hamkins and Johnstone  \cite{Hamkins:2013qv} examine the resurrection axiom for various forcing notions, which is a great reference for an array of similar results on other forcing classes. In particular, they looked at $\textsf{RA}_\Gamma(H_\c)$ for various classes $\Gamma$ such as that of proper, $ccc$, countably closed, and so on. 

The reason $H_\c$ is required in general is that if some forcing notion in $\Gamma$ adds new reals, then $H_\kappa$, where $\kappa>\c$ in $V$, simply cannot be existentially closed in the forcing extension; the added real itself is witnessing the lack of existential closure. So certainly $\textsf{RA}_\Gamma(H_\kappa)$ for any class of forcing notions $\Gamma$ that potentially add new reals, cannot hold. However, by the following result (due to \cite[Theorem 6]{Hamkins:2013qv}) we see that $\textsf{RA}_{\Gamma}(H_\c)$ is in fact equivalent to $\CH$, if $\Gamma$ is any class of forcing notions necessarily closed under finite iterations and containing a poset forcing $\CH$ without adding reals. So in particular, $\RAc(H_\c)$ and $\RAsc(H_\c)$ are both equivalent to $\CH$. 
Below I give the argument for subcomplete forcing, but it is the same proof as given by Hamkins and Johnstone \cite[Theorem 6]{Hamkins:2013qv}. 

\begin{prop} \label{prop:CHiffRAscHc} $\CH \iff \RAsc(H_\c)$. \end{prop}
\begin{proof}
For the forward implication, suppose that $\CH$ holds. Then since subcomplete forcing doesn't add new reals, $H_{\omega_1}$ is unaffected by subcomplete forcing, and moreover, $\c$ remains $\omega_1$ in every subcomplete extension. 

For the backward direction, assume $\RAsc(H_\c)$ holds. Let $\P$ be the canonical forcing to force $\CH$, which is countably closed and hence subcomplete. By the resurrection axiom $\RAsc(H_\c)$, there is a further forcing $\dot \R$ such that $\forces_\P ``\dot \R \text{ is subcomplete}"$ such that letting $g *h \subseteq \P * \dot \R$ be generic, we have $H_\c \preccurlyeq H_\c^{V[g*h]}.$ We know that $\CH$ has to hold still in $V[g*h]$ since subcomplete forcing cannot add new reals to make $\c$ larger. Thus $\CH$ holds in $V$ by elementarity, as desired.

Indeed, $\CH$ is equivalent to the statement that $H_{\c}$ contains only one infinite cardinal, which can be expressed in $H_{\c}$. 
\end{proof}
This result is certainly interesting in its own right, but it also indicates that the $\RAsc(H_\c)$, or indeed $\RAc(H_\c)$, is not necessarily the right axiom to look at. So what is the correct axiom to examine? I will discuss two reasonable possibilities: $\RAsc(H_{\omega_2})$ and $\RAsc(H_{2^{\omega_1}})$, and the same for the countably closed resurrection axiom. First we will look at what $\RAsc(H_{2^{\omega_1}})$ and $\RAc(H_{2^{\omega_1}})$ imply about the size of $2^{\omega_1}$.

\begin{prop} $\RAsc(H_{2^{\omega_1}}) \implies 2^{\omega_1} = \omega_2$. 
Indeed, $\RAc(H_{2^{\omega_1}}) \implies 2^{\omega_1} = \omega_2$. \end{prop}
\begin{proof}
We show the contrapositive. Let $2^{\omega_1} \geq \omega_3$. Let $\kappa = \omega_2^V$. Then $H_{2^{\omega_1}}$ can see that $\kappa = \omega_2$. But after forcing with $\Coll(\omega_1, \kappa)$, which is subcomplete since it is countably closed, we have that $H_{2^{\omega_1}}^{V[g]} \models ``\kappa < \omega_2"$ where $g \subseteq \Coll(\omega_1, \kappa)$ is generic. Moreover, if $\R$ is any further subcomplete (or countably closed) forcing, we will still have that for $h \subseteq \R$ generic, $H_{2^{\omega_1}}^{V[g][h]} \models ``\kappa < \omega_2"$. So $\RAsc(H_{2^{\omega_1}})$ (or $\RAc(H_{2^{\omega_1}})$) must fail.
\end{proof}

The next proposition gives a relationship between $\RAsc(H_{2^{\omega_1}})$ and $\RAsc(H_{\omega_2})$. However, the answer to the following question is unknown.
\begin{question} Is it the case that $\RAsc(H_{2^{\omega_1}}) \iff \RAsc(H_{\omega_2})$? 
Indeed, is it the case that $\RAc(H_{2^{\omega_1}}) \iff \RAc(H_{\omega_2})$?\end{question}

\begin{prop} $\RAsc(H_{2^{\omega_1}}) \iff 2^{\omega_1}=\omega_2 + \RAsc(H_{\omega_2})$.

Indeed, $\RAc(H_{2^{\omega_1}}) \iff 2^{\omega_1}=\omega_2 + \RAc(H_{\omega_2})$.\end{prop}
\begin{proof}
For the forward direction for the subcomplete resurrection axiom, we already have that $\RAsc(H_{2^{\omega_1}}) \implies 2^{\omega_1}=\omega_2$ by the previous proposition. Moreover, if $\RAsc(H_{2^{\omega_1}})$ holds, so does $\RAsc(H_{\omega_2})$, since even if some subcomplete forcing makes $2^{\omega_1} \geq \omega_3$ hold in the extension, we can always then force again to make $2^{\omega_1} = \omega_2$ in a further collapse (and thus countably closed, so subcomplete) extension. 

For the backward direction, suppose that $\RAsc(H_{\omega_2})$ holds and $2^{\omega_1}=\omega_2$. We would like to show that $\RAsc(H_{2^{\omega_1}})$ holds.
Toward that end, suppose that $\Q$ is subcomplete and let $g \subseteq \Q$ be generic. Then we have that there is some forcing $\R$ with $h \subseteq \R$ generic over $V[g]$, such that 
	$$H_{2^{\omega_1}}^V=H_{\omega_2}^V \preccurlyeq H_{\omega_2}^{V[g*h]}.$$ 
So if in $V[g*h]$ we have that $2^{\omega_1}=\omega_2$, then we are done. If not, i.e. if $2^{\omega_1}> \omega_2$ in $V[g]$, then let $G \subseteq \Coll(\omega_2, 2^{\omega_1})$ be generic over $V[g*h]$. Then $H_{\omega_2}^{V[g*h]}=H_{\omega_2}^{V[g*h*G]}=H_{2^{\omega_1}}^{V[g*h*G]}$, so we are done.

Since the ony forcing used in this proof is the collapse forcing, we also have the desired result for countably closed forcing.
\end{proof}

In comparison to \textbf{Proposition \ref{prop:CHiffRAscHc}} and \cite[Theorem 6]{Hamkins:2013qv}, the lack of any obvious restraints for the size of $2^{\omega_1}$ which is witnessed by both $\RAsc(H_{\omega_2})$ and $\RAc(H_{\omega_2})$ lend credibility to $H_{\omega_2}$ being the right parameter set to consider for our purposes. So this is what we will be using.

In the following we show that necessary form of the resurrection axiom is false, using the same argument as with the inconsistency of the necessary form of subcomplete maximality principle (\textbf{Proposition \ref{prop:necessaryMPinconsistent}}) that was originally demonstrated for countably closed forcings by Fuchs \cite{Fuchs:2008rt}.

\begin{prop}\label{prop:KurepaRA} Assume $\RAsc(H_\tau)$. Then there are no Kurepa trees in $H_\tau$. Indeed, in the case $H_\tau=H_{\omega_2}$, there are no Kurepa trees.
\end{prop}
\begin{proof}
Assume that $T \in H_\tau$ is a Kurepa tree with $\lambda$-many branches, where $\lambda \geq \omega_2$. Then forcing with $\Coll(\omega_1, \lambda)$ doesn't add branches to $T$, since the forcing is of course countably closed. But now $T$ has $\omega_1$-many branches in the extension $V[G]$, so $T$ is no longer a Kurepa tree in $V[G]$. Indeed, $T$ can never become a Kurepa tree in any further extension by a subcomplete forcing since subcomplete forcing does not add branches to $\omega_1$-trees by \textbf{Theorem \ref{thm:scbranch}}. Furthermore, the $H_{\tau}$ of any further extension of $V[G]$ may verify that ``$T$ does not have set-many branches." As $T$ is assumed to be in $H_{\tau}$ and is thus allowed as a parameter in $\RAsc(H_\tau)$, it follows that $T$ is not a Kurepa tree in further extensions, a contradiction.
\end{proof}

\begin{prop} $\Box \RAsc(H_{\omega_2})$ is inconsistent with $\ZFC$. \end{prop}
\begin{proof}
Work in $\ZFC + \Box \RAsc (H_{\omega_2})$. There is a generic extension obtained by subcomplete forcing in which there is a Kurepa tree (since forcing to add one is countably closed.) In this extension, $\RAsc(H_{\omega_2})$ has to still hold, since we are assuming the necessary form of the principle. But this is a contradiction to \textbf{Proposition \ref{prop:KurepaRA}}. 
\end{proof}

\subsection{Consistency of the Resurrection Axiom}
\label{subsec:ConRA}
Hamkins and Johnstone \cite[Section 5]{Hamkins:2013qv} show that the resurrection axiom for various classes of forcings is equiconsistent with the existence of an uplifting cardinal. By adapting their methods, we show that the large cardinal strength of $\RAsc(H_{\omega_2})$ is also that of an uplifting cardinal.

\begin{defn} An inaccessible cardinal $\kappa$ is \textbf{\emph{uplifting}} so long as  for every ordinal $\theta$ it is $\theta$-uplifting, meaning that there is an inaccessible $\gamma \geq \theta$ such that $V_\kappa \preccurlyeq V_\gamma$ is a proper elementary extension. \end{defn}

In the remarks after defining uplifting cardinals, Hamkins and Johnstone add that whenever $V_\kappa \preccurlyeq V_\gamma$ for large enough $\gamma$, where $\kappa$ and $\gamma$ are not necessarily inaccessible, it must be that $\kappa$ and thus $\gamma$ are $\beth$-fixed points. This gives rise to the following useful equivalent characterization of uplifting cardinals:

\begin{fact} A cardinal $\kappa$ is uplifting iff $\kappa$ is regular and for arbitrarily large regular cardinals $\gamma$ we have $H_\kappa \preccurlyeq H_\gamma$. \end{fact}

In terms of consistency strength, if $\kappa$ is Mahlo, then $V_\kappa$ contains a proper class of uplifting cardinals. If $\kappa$ is uplifting, then letting $V_\kappa \preccurlyeq V_\gamma$, we have that $V_\gamma$ is a transitive set model of the existence of a fully reflecting cardinal. For more details, see \cite[Theorem 11]{Hamkins:2013qv}.

We define some ``niceness" properties of class forcing, which will be needed in the following argument.

\begin{defn} A forcing notion $\P \subseteq M$ is \textbf{\textit{nice}} for a class forcing over $\mathcal M= \langle M, \in, A \rangle$ if $\P$ is definable in $\mathcal M$ and the truth lemma - stating that a sentence is true in a forcing extension if and only if it is forced by a condition in the generic filter - holds for forcing with $\P$ over $\mathcal M$. 

If $\P$ is nice for class forcing over $\mathcal M$ then we say that the niceness of $\P$ is \textbf{\textit{preserved}} to $\mathcal M^*=\langle M^*, \in, A^* \rangle$ if $\P^*$ is nice for class forcing over $\mathcal M^*$, is defined by the same formula over $M^*$, and the forcing relation for forcing with $\P^*$ over $\mathcal M^*$ are defined in $\mathcal M^*$ by the same formulas and same parameters as those for forcing with $\P$ over $\mathcal M$.
\end{defn} 

We will make use of the following fact from \cite[Lemma 17]{Hamkins:2014yu}, allowing the lifting of embeddings to generic extensions while performing class forcing. 

\begin{fact}[Lifting Lemma] \label{fact:liftinglemma}
Suppose that $\langle M, \in A \rangle  \preccurlyeq \langle M^*, \in, A^* \rangle$ are transitive models of $\ZFC$, that $\P$ is a definable class in $\langle M, \in, A \rangle$ that is nice for forcing and that the niceness of $\P$ is preserved to the analogous class $\P^*$ defined in $\langle M^*, \in, A^* \rangle$. If $G \subseteq \P$ is an $M$-generic filter and $G^* \subseteq \P$ is $M^*$-generic with $G = G^* \cap \P$, then $$\langle M[G], \in, A, G \rangle \preccurlyeq \langle M^*[G^*], \in, A^*, G^* \rangle.$$
\end{fact}

%
%

We deviate slightly from the proofs presented by Hamkins and Johnstone that use an uplifting cardinal to force specific resurrection axioms. Instead of using an ordinal-anticipating uplifting Laver function, we use a least-counterexample lottery sum iteration.

\begin{thm} \label{thm:uplifting->RA} If $\kappa$ is an uplifting cardinal, then there is a subcomplete iteration of length $\kappa$ such that in the forcing extension, $\RAsc(H_{\omega_2})$ holds and $\kappa = \omega_2$. \end{thm} 
\begin{proof}
Let $\kappa$ be uplifting. Suppose that $\RAsc(H_{\omega_2})$ fails. We shall define $\P$ to be the subcomplete least-counterexample to $\RAsc(H_{\omega_2})$ lottery sum $rcs$ iteration of length $\kappa$. The iteration is of the form:  $\P = \P_\kappa = \seq{ ( \P_\alpha, \dot{\Q}_\alpha ) }{ \alpha < \kappa }$, and is a revised countable support iteration of length $\kappa$, defined so that at stage $\alpha$: 
first let $\mathcal Q$ be the collection of subcomplete forcing posets in $V^{\P_\alpha}$ of \textit{minimal rank} for which resurrection fails, and define
$$\text{$\P_{\alpha+1} = \P_\alpha * \dot{\Q}_\alpha * \Coll(\omega_1, |\P_\alpha|)$ where $\dot{\Q}_\alpha$ is a term for the lottery sum $\oplus \mathcal Q$.}$$
So at each stage $\alpha$, we ask whether the resurrection axiom has been forced yet in $V^{\P_\alpha}$. If not, i.e., if there is a further subcomplete forcing in $V^{\P_\alpha}$ that is of least rank that cannot be resurrected, then force with such a poset. 

We shall refer to this definition as the subcomplete least-counterexample to $\RAsc(H_{\omega_2})$ lottery sum iteration of length $\kappa$.

Notice that this iteration may as well have been defined in $H_\kappa$, since $\kappa$ is uplifting.\footnote{Indeed, uplifting cardinals are $\Sigma_3$-reflecting, as Hamkins and Johnstone show.} In particular, if $V^{\P_\alpha}$ thinks that $\dot \Q$ is a least rank counterexample to resurrection at stage $\alpha+1$, let $\gamma$ be large enough so that $\dot \Q \in H_\gamma$ and $H_\gamma^{\P_\alpha} \models ``\dot \Q$ is subcomplete." Since $\kappa$ is uplifting, this means that there is a $\dot \Q'$ in $H_\kappa^{\P_\alpha}$ which is a counterexample to resurrection. But since $\dot \Q$ is chosen to be of least rank, it must in fact be that $\dot \Q$ has rank at most that of $\dot \Q'$, which means that $\dot \Q$ is in $H^{\P_\alpha}_\kappa$, and so $H_\kappa$ is enough to recognize least-rank counterexamples to resurrection. Furthermore each of the posets $\P_\alpha$, for $\alpha <\kappa$, are in $H_\kappa$, and $H_\kappa$ agrees that these posets are subcomplete and cannot be resurrected. Indeed, if instead it were possible for $H_\kappa$ to think that there is some subcomplete forcing $\Q'$ that cannot be resurrected, i.e., all further subcomplete $\dot \R$ have the property that if $g*h \subseteq \Q'*\dot \R$ then $H_{\omega_2}=H_{\omega_2}^{H_\kappa}$ is not elementary in $H_{\omega_2}^{V[g][h]}= H_{\omega_2}^{H_\kappa[g][h]}$, then it is correct. For the only way that $H_\kappa$ could possibly be wrong is that instead there is some subcomplete $\dot \R$ that in fact does resurrect $\Q'$ in $V$. But this implies that there must be such a poset that is an element of $H_\kappa$. To see why, take $\gamma$ to be large enough so that $\dot \R \in H_\gamma$, where $\gamma$ is larger than the verification of $\dot \R$'s subcompleteness in $H_\gamma^{\Q'}$. Then $H_\gamma$ sees that there is a further subcomplete forcing $\dot \R$ that resurrects $\Q'$, namely, letting $H \subseteq \R$ be generic over $V[g]$, we have that $H_{\omega_2}^V = H_{\omega_2}^{V_\kappa}$ is elementary in $H_{\omega_2}^{V[g][H]}=H_{\omega_2}^{H_\gamma[g][H]}$. As $\kappa$ is uplifting, $H_\kappa \preccurlyeq H_\gamma$, so there must be a witness to resurrection in $H_\kappa$. Thus least rank counterexamples of the iteration fall in $H_\kappa$. So each of the iterants of $\P$ are in $H_\kappa$, as they are always assumed to have least rank. This also shows that the forcing $\P$ is $\kappa$-$cc$, since each $\P_\alpha$ is relatively small and $\kappa$ is inaccessible. 

We then have that letting $\mathcal M = \langle H_\kappa, \in \rangle$, the subcomplete least-counterexample to  $\RAsc(H_{\omega_2})$ lottery sum iteration of length $\kappa$, $\P \subseteq H_\kappa$, is nice for class forcing over $\mathcal M$.

Let us now show that this iteration $\P$ actually works as planned, and shows that $\RAsc(H_{\omega_2})$ holds in the extension. Let $G \subseteq \P$ be generic. Suppose toward a contradiction that $\RAsc(H_{\omega_2})$ does not hold in $V[G]$. Then there is a subcomplete forcing $\Q$ of least rank in $V[G]$ for which resurrection fails. Let $\dot{\Q}$ be a name for $\Q$ such that there is $p \in G$ forcing that resurrection fails for this $\dot{\Q}$: in particular, $p$ forces that $\Q$ is a subcomplete poset of least rank so that after any further forcing $\dot \R$ satisfying $\forces_{\Q} ``\dot{\R} \text{ is subcomplete}"$, letting $h \subseteq \Q *\dot \R$ be $V[G]$-generic, it is not the case that $H^{V[G]}_{\omega_2} \preccurlyeq H_{\omega_2}^{V[G*h]}$. Let $H \subseteq \Q$ be generic over $V[G]$.

We will use the uplifting property of $\kappa$ to argue that $\Q$ appears at stage $\kappa$ of the same exact iteration, except the iteration as defined in some larger inaccessible $\gamma$, toward a contradiction.

To find a suitable $\gamma$ we use the fact that $\kappa$ is highly uplifting.
In particular we may let $\gamma > \theta$, where $\theta$ verifies the subcompleteness of $\P$. Also let $\gamma$ be larger than the verification of $\Q$'s subcompleteness. Then any such $\gamma$ will also verify the subcompletenes of $\P$ and $\Q$ by \textbf{Lemma \ref{lem:scabsolute}}. Moreover, to make sure that $H_\gamma[G]$ will agree that $\Q$ is of least rank, make sure that $\gamma$ is large enough so that $H_\gamma[G]$ contains the further, set-many subcomplete posets $\dot \R$, which provide the reasons that lower-rank posets are not in fact counterexamples to $\RAsc(H_{\omega_2})$ in $V[G]$. In particular, for every subcomplete poset $\Q' \in V[G]$ of lower rank than $\Q$, make sure that there is a further forcing $\dot \R$ in $H_\gamma[G]$ such that $\forces_{\Q'} ``\dot{\R} \text{ is subcomplete}"$ in $H_\gamma[G]$, such that letting $h \subseteq \Q' *\dot \R$ be $V[G]$-generic, it is the case that $H^{V[G]}_{\omega_2} \preccurlyeq H_{\omega_2}^{V[G*h]}$ -- in other words, in $H_\gamma[G]$, we have that $\dot \R$ resurrects $H_{\omega_2}$. This ensures that $H_\gamma[G]$ agrees with $V[G]$ that $\Q$ is a least-rank counterexample. This follows since for any stage $\alpha<\kappa$, we have $H_\kappa[G_\alpha] \preccurlyeq H_\gamma[G_\alpha]$ where $G_\alpha \subseteq \P_\alpha$.
Finding such an inaccessible $\gamma$ is possible, since this process is bounded by the size of the iteration $\P$ and the size of the poset $\Q$.


Let $\P^*=\P_\gamma$ be defined in the same way as the least-rank counterexample to $\RAsc(H_{\omega_2})$ lottery sum $rcs$ iteration of length $\kappa$ as we defined above, except with length $\gamma$ instead of $\kappa$. Above, we ensured that $\P$ and $\P^*$ agree below stage $\kappa$. 

Thus letting $\mathcal M^*=\langle H_\gamma, \in \rangle$, we have that the niceness of $\P$ is preserved to $\mathcal M^*$, as $\P_\gamma \subseteq H_\gamma$ is defined the same way over $H_\gamma$ as $\P$ is over $\mathcal M$.

Above we have shown that 
$\dot \Q$ may be chosen at stage $\kappa$. Thus we may say that below a condition picking $\dot \Q$ at stage $\kappa$, namely $p \in \P$, we have that $\P^*$ factors as $\P * \dot{\Q} * \Ptail$. Let $G^*=G*H*G_{\text{tail}} \subseteq \P * \dot{\Q} * \Ptail$ be generic over $V$. Then by the Lifting Lemma (\textit{Fact} \ref{fact:liftinglemma}), as we have argued that the iteration is defined over $H_\kappa$, we have that $H_\kappa \preccurlyeq H_\gamma$ given by $\kappa$ uplifting lifts to 
	$$H_\kappa[G] \preccurlyeq H_\gamma[G^*]$$ in $V[G^*]$.

Note that we already have that $H_\kappa[G]=H_\kappa^{V[G]} = H_{\omega_2}^{V[G]}$, since $\kappa$ is regular and $\P$ has the $\kappa$-cc, and by \textbf{Lemma \ref{lem:lengthcollapse}} the length of the iteration is collapsed to $\omega_2$.
We can argue the same way to get that $H_\gamma[G^*]=H_\gamma^{V[G^*]} = H_{\omega_2}^{V[G^*]}$. 

Thus we have that $H_{\omega_2}^{V[G]} \preccurlyeq H_{\omega_2}^{V[G*H*G_{\text{tail}}]}$, contradicting the choice of $\Q$. So $\RAsc(H_{\omega_2})$ holds as desired.
\end{proof}

The same proof may of course be used to show that $\RAc(H_{\omega_2})$ may be forced from the existence of an uplifting cardinal. For the other direction in showing the equiconsistency of $\RAsc(H_{\omega_2})$ (or indeed $\RAc(H_{\omega_2})$,) with an uplifting cardinal, we shall use the same methods as in \cite{Hamkins:2013qv} to see that if the resurrection axiom holds, then $\omega_2$ is uplifting in $L$.

\begin{thm} $\RAsc(H_{\omega_2})$ implies that $\omega_2^V$ is uplifting in $L$. \end{thm}
\begin{proof}
Let $\kappa = \omega_2^V$. To see that $\kappa$ is uplifting in $L$, we need to show that for arbitrarily large ordinals $\gamma$ that are regular cardinals in $L$, we have that
	$H_\kappa^L \preccurlyeq H_\gamma^L.$
Since $\kappa$ is regular, $\kappa$ is regular in $L$. 
For any $\theta>\kappa$, let $\Q$ be the poset forcing $\Coll(\omega_1, \theta)$, which is of course subcomplete. By $\RAsc(H_{\omega_2})$ we have a further forcing $\dot \R$ such that $g*h \subseteq \Q * \dot \R$ is generic over $V$ and letting $\gamma = \omega_2^{V[g][h]}$, we have that 
	$$H_{\kappa}^V \preccurlyeq H_{\gamma}^{V[g*h]}.$$ 
Since $\theta$ was forced to have cardinality $\omega_1$ in $V[g]$, $\theta < \gamma$. Since $\gamma$ is regular in $V[g*h]$ we have that $\gamma$ is regular in $L$. Thus by relativizing formulas to the constructible universe, we have that 
	$$H_\kappa^L = (H_\kappa^V \cap L) \preccurlyeq (H_\gamma^{V[g*h]} \cap L) = H_\gamma^L.$$ Thus $\kappa$ is uplifting in $L$ as desired. \end{proof}

Of course we may also show using a virtually identical proof that $\RAc(H_{\omega_2})$ implies $\omega_2^V$ is uplifting in $L$.

What do models of $\RAsc(H_{\omega_2})$ look like? Many similar results shown in models of the maximality principle for subcomplete forcing also hold in models of the resurrection axiom for subcomplete forcing. One can view this as a consequence of the fact that both of these axioms imply the local maximality principle.

\begin{lem} \label{lem:RAsc->LMPsc} $\RAsc(H_{\omega_2}) \implies \lMPsc$. \end{lem}
\begin{proof}
Suppose that subcomplete resurrection $\RAsc(H_{\omega_2})$ holds. To see that the local subcomplete maximality principle holds, suppose that $\varphi(a)$ is a sentence such that the sentence ``$H_{\omega_2} \models \varphi(a)$" is $sc$-forceably $sc$-necessary. So there is a subcomplete forcing $\P$ such that after any further forcing, we have that ``$H_{\omega_2} \models \varphi(a)$" holds in the extension. By resurrection, there is a further $\dot \R$ such that $\forces_\P ``\text{$\dot \R$ is $sc$}"$ such that letting $G*h \subseteq \P *\dot \R$ be generic we have $H_{\omega_2} \preccurlyeq H_{\omega_2}^{V[G*h]}$. Since $H_{\omega_2}^{V[G*h]} \models \varphi(a)$, this means that $H_{\omega_2} \models \varphi(a)$ holds by elementarity, so $\lMPsc$ holds as desired.
\end{proof}

\begin{prop} \label{prop:RAsc->SuslinDiamondCH} $\RAsc(H_{\omega_2})$ implies the following: \begin{enumerate}
	\item There is a Suslin tree.
	\item $\lozenge$ holds.
	\item $\CH$ holds.
\end{enumerate} \end{prop} 
\begin{proof}
By the above \textbf{Lemma \ref{lem:RAsc->LMPsc}}, $\RAsc(H_{\omega_2})$ implies $\lMPsc$, 
 the local subcomplete maximality principle. Indeed, looking at \textit{Fact} \ref{Fact:BFA} we see that the resurrection axiom is a kind of generalization of the bounded subcomplete forcing axiom. So we are done by \textbf{Proposition \ref{prop:localproperties}}.
\end{proof}

Just as in the case of the maximality principle, cf. \textbf{Proposition \ref{prop:L[a]notV}}, if $V=L[a]$ for some set $a$, then it is not possible for both $\RAc(H_{\omega_2})$ and $\RAsc(H_{\omega_2})$ to both hold.

\begin{thm} \label{thm:V=L[a]+RA} If $V=L[a]$ where $a$ is a set, and $\RAc(H_{\omega_2})$ holds, then $\RAsc(H_{\omega_2})$ fails. \end{thm}
\begin{proof}
Without loss of generality we may assume that the set $a$ is a set of ordinals, $a \subseteq \theta$ for some $\theta$. Let $V=L[a]$, and assume that $\RAc(H_{\omega_2})$ holds. 

Let $G \subseteq \Coll(\omega_1, \theta)$ be generic. Then as $V=L[a]$ we have that 
	$$H_{\omega_2}^{V[G]} \models ``\text{There is a set $b \subseteq \Ord$ such that every countable set of ordinals is in $L[b]$.}"$$ 
Namely, this sentence is witnessed by $b=a$. In particular, letting $\kappa = \omega_2^{V[G]}$, we have that 
	$$\kappa^\omega \cap L[a] = \kappa^\omega \cap L_\kappa[a] = \left( \Ord^\omega \cap L[a] \right)^{H_\kappa^{V[G]}}.$$ 
Then by $\RAc(H_{\omega_2})$, letting $H$ be generic for some countably closed forcing over $V[G]$, we have that 
	$H_{\omega_2}^V \preccurlyeq H_{\omega_2}^{V[G][H]}.$ 
So 
	$$H_{\omega_2}^V \models ``\text{There is a set $b \subseteq \Ord$ such that every countable set of ordinals is in $L[b]$.}"$$
Let $b$ be such a set in $H_{\omega_2}^V$. Then $b$ is a bounded subset of $\omega_2$, and in particular, we have that
	$\omega_2^\omega = \left(\omega_2^\omega \right)^{L_{\omega_2}[b]}$, i.e., countable subsets of $\omega_2$ of $V$ are the same as those of $L_{\omega_2}[b]$.
But this means that Namba forcing, which adds a new countable $\omega_2$-sequence not in $V$, can't be resurrected, so $\RAsc(H_{\omega_2})$ fails.	
\end{proof}

As with the maximality principle, cf. \textbf{Theorems \ref{thm:ConBFMPscandNotLFMPc}, \ref{thm:ConBFMPcandNotLFMPsc}}, it is not surprising that it is possible for the subcomplete resurrection axiom to hold while the countably closed resurrection axiom fails. 

\begin{cor} It is consistent for $\RAsc(H_{\omega_2})$ to hold while $\RAc(H_{\omega_2})$ fails, and vice-versa. \end{cor}
\begin{proof}
Force $\RAsc(H_{\omega_2})$ to hold over $L$, with $\delta$ an uplifting cardinal. Then we obtain $L[G]$, which models the subcomplete resurrection axiom. Thus by the previous \textbf{Theorem \ref{thm:V=L[a]+RA}}, we have that $\RAc(H_{\omega_2})$ fails, as desired. The same proof may be used to show that it is possible for $\RAsc(H_{\omega_2})$ to fail while $\RAc(H_{\omega_2})$ holds.
\end{proof}

It is then reasonable to ask whether it is possible for the subcomplete resurrection axiom and the countably closed resurrection axiom to hold at the same time. 

\begin{question} Is it consistent for $\RAsc(H_{\omega_2})$ and $\RAc(H_{\omega_2})$ to both hold? \end{question}

This question is similar to the question for the local maximality principle, since if the above resurrection axioms did both hold, then both $\lMPsc$ and $\bflMPc$ hold, meaning that $\lMPsc$ and $\lMPc$ both hold. This would imply that $V$ is closed under sharps (\textbf{Theorem \ref{thm:LMPsc+LMPcImpliesSharpClosure}}). Furthermore, we know that if there is a hyper-huge cardinal, then $\lMPsc$ and $\lMPc$ do not both hold, by relativizing the proof of \textbf{Theorem \ref{thm:BFMPsc+LFMPcImpliesGrounds}} to $H_{\omega_2}$. Additionally, if the above resurrection axioms both hold, then $\lflMPsc$ and $\bflMPc$ both hold. This is consistent, for example if $\MPc(H_{\omega_2})$ and $\MPsc(\emptyset)$ both hold, as in \textbf{Theorem \ref{thm:ConBFMPc+LFMPsc}}. Thus we leave this question to future research.

\subsection{Consistency of the Local Maximality Principle}
\label{subsec:ConlocalMP}
We will now introduce the large cardinal property that is equiconsistent with the local maximality principle. We give the proof after showing the equiconsistency of the resurrection axiom because of the similarities in method and large cardinal used. When showing the consistency of the subcomplete resurrection axiom in \ref{subsec:ConRA}, we defined the notion of an \textit{uplifting cardinal}, of which the following property is the suitable ``local" version.

\begin{defn} An inaccessible cardinal $\delta$ is \textbf{\emph{locally uplifting}} so long as for every formula $\varphi(x)$ and $a \in V_\delta$, for every $\theta$ we have that $\theta$-locally uplifting, meaning that there is an inaccessible $\gamma > \theta$ such that 
	$V_\delta \models \varphi(a) \iff V_\gamma \models \varphi(a).$ \end{defn}
Note that if a regular cardinal $\delta$ has the property of being locally uplifting, without necessarily being inaccessible, then $\delta$ must be inaccessible, since otherwise if $2^\alpha \geq \delta$ for some $\alpha<\delta$, this is seen by some larger $V_\gamma$, i.e., $V_\gamma \models \exists \beta \ 2^\alpha = \beta$. So by elementarity there is some $\beta' = 2^\alpha$ in $V_\delta$, a contradiction.

We have the following relationship between locally uplifting and reflecting cardinals.

\begin{prop} If $\kappa$ is locally uplifting then $\kappa$ is reflecting. \end{prop}
\begin{proof} 
Suppose that $\kappa$ is locally uplifting. To show that $\kappa$ is reflecting, let $\varphi(x)$ be a formula and let $a \in H_\kappa$, and that there is $\theta > \kappa$ where $H_\theta \models \varphi(a)$. Define $\psi$ as follows:
	$$\psi(a): \ \  \exists \delta \ H_\delta \models \varphi(a).$$
Then if we take $\gamma>\theta$ satisfying $H_\theta \in H_\gamma$, we have that $H_\gamma \models \psi(a)$. As $\kappa$ is locally uplifting, this implies that $H_\kappa \models \psi(a)$. Thus there is $\delta < \kappa$ such that $H_\delta \models \varphi(a)$ as desired.
\end{proof}

It is not hard to see that the local maximality principle is implied by the resurrection axiom, as has been pointed out before in \textbf{Lemma \ref{lem:RAsc->LMPsc}}. Here we show that the local maximality principle is equiconsistent with the existence of a locally uplifting cardinal, using the same method as with the proof of the maximality principle but with some care in relativizing to $H_{\omega_2}$.
	
\begin{thm} If $\delta$ is locally uplifting, then there is a subcomplete forcing extension in which $\lMPsc$ holds and $\delta = \omega_2$. \end{thm}
\begin{proof}
Let $\delta$ be locally uplifting.
Define the $\delta$-length subcomplete lottery sum $rcs$ iteration $\P = \P_\delta$ as follows: for $\alpha < \delta$ let 
	$$\P_{\alpha+1} = \P_\alpha *\dot{\Q}_\alpha*\Coll(\omega_1, |\P_\alpha|),$$ 
where $\dot{\Q}_\alpha$ is a $\P_\alpha$-name for the lottery sum of all minimal rank subcomplete posets that force some sentence relativized to $H_{\omega_2}^{V^{\P_\alpha}}$ to be $sc$-necessary. In particular, let $\Phi$ be the collection of formulas $\varphi(x)$ with parameter $a \in H_{\omega_2}^{V^{\P_\alpha}}$ such that:
	$$\text{$V^{\P_\alpha} \models $ ``$\varphi^{H_{\omega_2}}(a)$ is $sc$-forceably $sc$-necessary,"} \text{ in other words,}$$
	$$\text{$V^{\P_\alpha} \models $ ``$`H_{\omega_2} \models \varphi(a)$' is $sc$-forceably $sc$-necessary."}$$
So $\Phi$ is the set of all possible $H_{\omega_2}$-buttons available at this point in the iteration. Then we let 
	$$\dot{\Q}_\alpha = \bigoplus_{\varphi \in \Phi} \set{ \dot {\Q} \in V^{\P_\alpha} }{ \dot \Q \text{ is least rank,  $V^{\P_\alpha} \models$ ``$\dot{\Q}$ is $sc$ and forces `$\varphi(a)^{H_{\omega_2}}$ is $sc$-necessary.'"} }$$ 
Since we will want the full iteration $\P$ to remain relatively small in size and to have the $\delta$-$cc$, notice that here we insist that the parameters for our sentences come from $H_{\omega_2}^{V^{\P_\alpha}}$, and that the iteration may as well have been defined in $V_\delta$. To see why amounts to the same argument as given in the proof of \textbf{Theorem \ref{thm:uplifting->RA}}. Firstly, as $\delta$ is inaccessible, it is large enough so that $H_{\omega_2}^V = H_{\omega_2}^{V_\delta}$, and moreover this remains true in each subsequent extension in the iteration, i.e., $V_\delta^{\P_\alpha} = V_\delta^{V^{\P_\alpha}}$, so $H_{\omega_2}$ in the subsequent extensions gets interpreted the same in $V_\delta^{\P_\alpha}$ as in $V^{\P_\alpha}$. If at stage $\alpha+1$ we have that $\dot \Q' \in V^{\P_\alpha}$ is subcomplete and forces that a sentence $\varphi(a)^{H_{\omega_2}}$ is $sc$-necessary, where $a \in H_{\omega_2}$, then take $\theta$ large enough so that $\dot \Q' \in V^{\P_\alpha}_\theta$ and $\theta$ verifies the subcompleteness of $\dot \Q'$. Then for all $\gamma > \theta$ we have by \textbf{Lemma \ref{lem:scabsolute}} that
	$$V^{\P_\alpha}_\gamma \models \text{``There is a subcomplete forcing $\dot \Q$ forcing that `$\varphi(a)^{H_{\omega_2}}$ is $sc$-necessary.'"}$$ 
As $\delta$ is locally uplifting, this means that we have $\gamma >\theta$ satisfying the above, which implies that
	$$V^{\P_\alpha}_\delta \models \text{``There is a subcomplete forcing $\dot \Q$ forcing that `$\varphi(a)^{H_{\omega_2}}$ is $sc$-necessary.'"}$$
So since each of the iterants of the forcing $\P$ are taken to be of least rank, they are all in $V_\delta$. If on the other hand at stage $\alpha+1$ we have that $\dot \Q' \in V_\delta^{\P_\alpha}$ is subcomplete and of least rank forcing that a sentence $\varphi(a)^{H_{\omega_2}}$ is $sc$-necessary, then the only way it could be wrong is that there is some further subcomplete forcing $\dot \R'$ that is not in $V_\delta^{\P_\alpha*\dot \Q'}$ that forces the sentence to be false. But then we may take $\gamma$ larger than the verification of this forcing $\dot \R'$, and use the fact that $\delta$ is locally uplifting to see that $$V^{\P_\alpha*\dot \Q'}_\delta \models ``\text{There is a subcomplete forcing $\dot \R$ forcing $\neg \varphi(a)^{H_{\omega_2}}$.'}$$ This contradicts the choice of $\dot \Q'$ in $V_\delta$, which means that $V_\delta$ is correct.
Thus the iteration is the same as if it were defined over $V_\delta$ as claimed.

Of course $\P$ is always defined, and for any name for a set $\dot x$ in $H_{\omega_2}^{V^{\P_\alpha}}$ the sentence ``$H_{\omega_2} \models |\dot x|=\omega_1$" is $sc$-forceably $sc$-necessary in $V^{\P_\alpha}$ via collapse forcing.

We shall refer to this definition as the subcomplete least-rank $\lMPsc$ lottery sum iteration of length $\delta$. 

Now suppose that $G \subseteq \P$ is generic over $V$. Let's see that $V[G] \models \lMPsc$. Assume toward a contradiction that it fails: namely $\varphi(x)$ is a formula and $a \in H^{V[G]}_{\omega_2}$ are such that in $V[G]$:
	$$\text{``$\varphi^{H_{\omega_2}}(a)$ is $sc$-forceably $sc$-necessary,"} \text{ in other words,}$$
	$$\text{ ``$`H_{\omega_2} \models \varphi(a)$' is $sc$-forceably $sc$-necessary,"}$$
however, $\varphi(a)^{H_{\omega_2}}$ is not true in $V[G]$. Let us also take $p \in G$ forcing the above to be the case.

Note that $H_\delta^{V[G]} = H_{\omega_2}^{V[G]}$, since $\delta$ is regular and $\P$ has the $\delta$-cc, and by \textbf{Lemma \ref{lem:lengthcollapse}} the length of the iteration is collapsed to $\omega_2$.		

Let $\dot{\Q}$ be a name for $\Q$, a least rank poset in $V^{\P}$ and $\dot a$ be a name for $a$ such that in $V^\P$, we have that
``$\varphi(\dot a)^{H_{\omega_2}}$ is $sc$-necessary."

Since $\P$ has the $\delta$-$cc$, at no stage in the iteration could $\delta$ be collapsed. This means that there is some stage where the parameter $\dot a$ appears, as with the proof of the Maximality Principle from a fully reflecting cardinal (\textbf{Theorem \ref{thm:forceMP}}). Thus we may find a stage in the iteration where the parameter $a$ is available, past the support of $p$, say $a \in V_\delta[G_\alpha]$. 

Now we let $\theta$ be larger than the verification needed for $\P$ and $\dot \Q$'s subcompleteness, and so that $\P \in V_\theta$ and $\dot \Q \in V_\theta^\P$. Then as $\delta$ is locally uplifting, we have that there is $\gamma >\theta$ satisfying 
	$$V_\gamma[G_\alpha] \models ``\varphi(a)^{H_{\omega_2}} \text{ is $sc$-forceably $sc$-necessary."}$$
Namely, $\P_{tail} * \dot \Q$ makes $\varphi(a)^{H_{\omega_2}}$ necessary, and as $\gamma$ is larger than the verification needed for the subcompleteness of $\P$ and $\dot \Q$, we have that this is true in $V_\gamma$.  So by the fact that $\delta$ is uplifting, we have that 
	$$V_\delta[G_\alpha] \models ``\varphi(a)^{H_{\omega_2}} \text{ is $sc$-forceably $sc$-necessary."}$$
Moreover, $\varphi(a)^{H_{\omega_2}}$ must continue to be $sc$-forceably $sc$-necessary in later stages, as $\varphi(a)^{H_{\omega_2}}$ continues to be a button in $V_\gamma$ in later stages of the iteration since we took $\alpha$ larger than the support of $p$. Thus it is dense for the button to be pushed, so that we have $\varphi(a)^{H_{\omega_2}}$ is forced to be $sc$-necessary at some stage $V_\delta[G_\beta]$. But this means that $\varphi(a)^{H_{\omega_2}}$ is $sc$-necessary in $V_\delta[G_\beta]$, since the rest of the iteration is subcomplete. Thus $\varphi(a)^{H_{\omega_2}}$ is true in $V[G]$, a contradiction as it was chosen as evidence of the failure of $\lMPsc$ in $V[G]$.
\end{proof}

For the other direction of the equiconsistency of $\lMPsc$ with a locally uplifting cardinals, we show that in $L$ there is a locally uplifting cardinal if the local maximality principle for subcomplete forcing holds.

\begin{thm} If $\lMPsc$ holds, then $\omega_2$ is locally uplifting in $L$. \end{thm}	
\begin{proof}
Let $\kappa=\omega_2$ and suppose that the subcomplete local maximality principle holds.

Firstly, $\kappa$ is a limit cardinal in $L$, since for $\gamma < \kappa$, the statement $H_{\omega_2} \models$ ``there is a cardinal in $L$ greater than $\gamma$" is $sc$-forceably $sc$-necessary (by taking $\Coll(\omega_1, \kappa)$) and thus true in $H_{\omega_2}$. So we have that $\kappa$ is inaccessible.

Assume $L_\kappa \models \varphi(a)$. In other words, $H_{\omega_2} \models \varphi^L(a)$. We need to show that there is a larger $\gamma$ such that $L_\gamma \models \varphi(a)$. In order to do this, let's work in $L$ and first see that the following is $sc$-forceably $sc$-necessary:
	\begin{equation} \label{eqn:ForcesHmodelsPhiAndUnbdedlyCardsInL} H_{\omega_2} \models (\varphi^L(a) \land \text{ ``there are unboundedly many cardinals in $L$"}). \end{equation}
This holds since otherwise it is $sc$-forceably $sc$-necessary that $H_{\omega_2} \models \neg \varphi^L(a)$, so $H_{\omega_2} \models \neg \varphi^L(a)$ holds, a contradiction.

So, given some $\gamma > \kappa$, we may use subcomplete (indeed, countably closed) forcing over $L$ to collapse $\gamma$ to $\omega_1$. Then by (\ref{eqn:ForcesHmodelsPhiAndUnbdedlyCardsInL}), there is further forcing to reach a model $V[G][H]$ such that 
	$$H_{\omega_2}^{V[G][H]} \models \varphi^L(a).$$
Thus in $V[G][H]$, $\varphi^{L_{\omega_2}}(a)$ holds. Since $\omega_2^{V[G][H]} = \gamma'>\gamma>\kappa$ in this extension now, and furthermore by (\ref{eqn:ForcesHmodelsPhiAndUnbdedlyCardsInL})
	$$L_{\gamma'} \models \varphi(a) \text{ and ``there are unboundedly many cardinals"},$$ 
we now have a suitable $\gamma' = \omega_2^{V[G][H]}$ that is inaccessible in $L$ and $L_{\gamma'} \models \varphi(a)$ as desired.
\end{proof}

Of course, this theorem would also work to show that if the countably closed, proper, or semiproper local maximality principle holds, then $\omega_2$ is uplfiting in $L$, since all we use is countably closed forcing for the proof.

\section{Combining Resurrection and Maximality}
\label{sec:RA+MP}
Hamkins and Johnstone \cite[Section 6]{Hamkins:2013qv} combine the resurrection axiom with forcing axioms, like $\PFA$ for example, to hold after a forcing iteration. The same proof techniques show that the subcomplete forcing axiom can hold alongside the subcomplete resurrection axiom. 

This section focuses on a different question in a similar vein: is it possible for the resurrection axiom and the maximality principle to hold at the same time? The axioms do not imply each other directly,
$\MPsc(H_{\omega_2})$ surely does not imply $\RAsc(H_{\omega_2})$, since the consistency strength of $\MPsc(H_{\omega_2})$ is that of a fully reflecting cardinal, while $\RAsc(H_{\omega_2})$ has the consistency strength an uplifting cardinal. If $\kappa$ is fully reflecting, take the least $\gamma$ such that $V_\kappa \preccurlyeq V_\gamma$. If there isn't such a $\gamma$, then $\kappa$ isn't uplifting anyway. Then in $V_\gamma$, we have that $\kappa$ is not uplifting.
There is no implication in the other direction as well. This is because working in a minimal model of 
	$$T \ = \ \ZFC \; + \; ``V=L" \; + \; \text{ ``there is an uplifting cardinal"}$$ 
(i.e., no initial segment of the model satisfies this theory), we may force over this minimal model to obtain $\RAsc(H_{\omega_2})$. Now $\MPsc(H_{\omega_2})$ can't hold in the extension, since letting $\kappa$ be the $\omega_2$ of the extension, if $\MPsc(H_{\omega_2})$ were true, then that would imply that $L_\kappa$ is elementary in $L$, i.e., $L_\kappa$ would be a model of $T$ - contradicting the minimality of the model we started with.

However, both $\MPsc(H_{\omega_2})$ and $\RAsc(H_{\omega_2})$ imply the local version of the maximality principle, $\lMPsc$.

Here we show that it is possible for subcomplete maximality and resurrection to both hold, by combining the techniques showing the consistency of each principle, all in one minimal counterexample iteration. To do this, we need a combination of two large cardinal properties. Although showing the maximality principle will again need the length of the iteration to be fully reflecting, we will need a stronger notion of uplifting than we have previously stated. This is because as before, proving that the resurrection axiom holds relies upon being able to define the iteration in some larger class. However, the iteration designed to force the maximality principle from a fully uplifting cardinal $\kappa$ is not first order definable over $V_\kappa$, which is what we would need to use the uplifting criterion. It is definable over $V$, but not $V_\kappa$, as the sentences have parameters in $V_\kappa$ and truth is evaluated over $V_\kappa$. But this is remedied by using the methods of Hamkins and Johnstone in \cite{Hamkins:2014yu}, where they show the consistency of strongly uplifting cardinals is exactly that of the boldface resurrection axiom, which we will define below. Here we are allowed to carry around predicates in our structures.

\begin{defn} An inaccessible cardinal $\kappa$ is \textbf{\textit{strongly uplifting fully reflecting}} so long as: \begin{itemize}
	\item $\kappa$ is fully reflecting, i.e. $V_\kappa \preccurlyeq V$
	\item $\kappa$ is \textbf{\textit{strongly uplifting}}; meaning that it is strongly $\theta$ uplifting for every ordinal $\theta$. This means that for every $A \subseteq V_\kappa$ there is an inaccessible cardinal $\gamma \geq \theta$ and a set $A^* \subseteq V_\gamma$ such that $\langle V_\kappa , \in, A \rangle \preccurlyeq \langle V_\gamma, \in, A^* \rangle$ is a proper elementary extension.\footnote{As described by Hamkins and Johnstone in the comments on page 5 of their paper, we may let $\gamma$ be regular, uplifting, weakly compact, etc. Additionally, if $\kappa$ is strongly uplifting then $\kappa$ is inaccessible.} \qedhere
	\end{itemize} \end{defn}

Combining these two large cardinal notions is almost natural. If $\kappa$ is uplifting then there are unboundedly many $\gamma$ such that $V_\kappa \preccurlyeq V_\gamma$, and if on top of that $\kappa$ is reflecting, we add that $V_\kappa \preccurlyeq V$ as well, where $V$ is in some sense the limit of the $V_\gamma$'s.

\begin{defn} A cardinal $\delta$ is \textbf{\emph{subtle}} so long as for any club $C \subseteq \delta$ and for any sequence $\mathcal A = \seq{ A_\alpha }{ \alpha \in C }$ with $A_\alpha \subseteq \alpha$, there is a pair of ordinals $\alpha < \beta$ in $C$ such that $A_\alpha=A_\beta \cap \alpha$. \end{defn}

\begin{fact} \label{fact:subtleinacc}
If a cardinal $\delta$ is subtle, then $\delta$ is inaccessible.
\end{fact}
%

\begin{prop} If $\delta$ is subtle, then it is consistent that there is a strongly uplifting fully reflecting cardinal. Namely, $\set{ \kappa < \delta }{ \text{$V_\delta \models ``\kappa$ is strongly uplifting and $V_\kappa \preccurlyeq V$"} }$ is stationary in $\delta$. \end{prop} 
\begin{proof}
Hamkins and Johnstone \cite[Theorem 7]{Hamkins:2014yu} show that if $\delta$ is subtle, then the set of cardinals $\kappa$ below $\delta$ that are strongly uplifting in $V_\delta$ is stationary. But since $\delta$ is subtle, it must also be inaccessible by \textit{Fact} \ref{fact:subtleinacc}. Thus in $V_\delta$, by the proof of the downward Lowenheim-Skolem theorem, there is a club $C \subseteq \delta$ of cardinals $\kappa$ such that $V_\kappa \preccurlyeq V_\delta$; meaning that $\kappa$ is fully reflecting in $V_\delta$. This means that there is some $\alpha < \kappa$ that is both strongly uplifting and fully reflecting in $V_\delta$, giving us the required consistency.
\end{proof}

Since we will be using the stronger large cardinal notion while showing that the subcomplete maximality principle and the subcomplete resurrection principle can hold at the same time, we might as well show that the stronger subcomplete resurrection principle holds as well, which we define below. 

\begin{defn} Let $\Gamma$ be a fixed, definable class of forcing notions. Let $\tau$ be a term for a cardinality to be computed in various models; e.g. $\c$, $\omega_1$, etc. The \textbf{\emph{Boldface Resurrection Axiom}} $\textsf{\textbf{RA}}_\Gamma(H_\tau)$ asserts that for every forcing notion $\Q \in \Gamma$ and $A \subseteq H_\tau$ there is a further forcing $\dot{\R}$ with $\forces_{\Q} \dot{\R}  \in \Gamma$ such that if $g*h \subseteq \Q * \dot{\R}$ is $V$-generic, then there is an $A^* \in V[G*h]$ such that $\langle H_{\tau}^V, \in, A \rangle \preccurlyeq \langle H_{\tau}^{V[g*h]}, \in, A^* \rangle. $ \end{defn}

As discussed in the previous section, it doesn't make too much sense to talk about the subcomplete resurrection axiom at the continuum, like is done in \cite{Hamkins:2014yu}, so as before we will be relativizing the notion to subcomplete forcing using $\omega_2$ instead.
Thus the notion we will be looking at is $\bfRAsc(H_{\omega_2})$. We will be following the proof of  \cite[Theorem 19]{Hamkins:2014yu} to show that the boldface resurrection axiom holds from a strongly uplifting cardinal. Of course the following theorem may be tweaked to work for for other classes of forcing notions such as proper forcing, but our focus here is on subcomplete forcing.

\begin{thm} Let $\kappa$ be a strongly uplifting fully reflecting cardinal. Then there is a forcing extension in which both $\bfRAsc(H_{\omega_2})$ and $\MPsc(H_{\omega_2})$ hold, and $\kappa=\omega_2$. \end{thm}
\begin{proof}
Let $\kappa$ be strongly uplifting fully reflecting. Below we define $\P$ to be the subcomplete least-counterexample to $\bfRAsc(H_{\omega_2}) + \MPsc(H_{\omega_2})$ lottery sum $rcs$ iteration of length $\kappa$. In particular, generically pick, using the lottery sum, whether at each stage to force with the least-rank counterexamples to the maximality principle or the least-rank counterexamples to the boldface resurrection axiom.

In particular, we are defining the poset $\P = \P_\kappa = \seq{ ( \P_\alpha, \dot{\Q}_\alpha ) }{ \alpha < \kappa } $ as follows: 

At stage $\alpha+1$, look at all the formulas with parameters in $H_{\omega_2}^{V^{\P_\alpha}}$ that are not true in $V_\kappa^{\P_\alpha}$, but can be forced by some subcomplete poset $\dot \Q$ to be $sc$-necessary, where $\forces_{\P_\alpha} ``\dot \Q \text{ is subcomplete.}"$ Let $\mathcal M$ be the collection of such possible subcomplete forcing notions $\dot \Q$ in $V_\kappa^{\P_\alpha}$ of minimal rank in $V_\kappa$ for which the above holds. So $\mathcal M$ contains the minimal rank counterexamples to the maximality principle.

Additionally, let $\mathcal R$ be the collection of subcomplete forcings of minimal rank, $\dot \Q$, such that there is $\dot A \subseteq H^{V^{\P_\alpha}}_{\omega_2}$ where after any further forcing $\ddot \R$ satisfying $\forces_{\P_\alpha *\dot \Q} ``\ddot{\R} \text{ is subcomplete}"$, it is not the case that $\langle H^{V^{\P_\alpha}}_{\omega_2}, \in, \dot A \rangle  \preccurlyeq \langle H_{\omega_2}^{V^{\P_\alpha *\dot \Q *\ddot \R}}, \in, \dot A^* \rangle$ for all $\dot A^*$. Thus $\mathcal R$ contains the minimal rank counterexamples to the boldface resurrection axiom.


Then take:
$$\P_{\alpha+1} = \P_\alpha * \dot{\Q}_\alpha * \Coll(\omega_1, |\P_\alpha|)$$ where $\dot{\Q}_\alpha$ is a term for the lottery sum $\oplus \mathcal R \bigoplus \oplus \mathcal M$.

Let $G \subseteq \P$ be generic. We need to show that both $\bfRAsc(H_{\omega_2})$ and $\MPsc(H_{\omega_2})$ hold in $V[G]$. 

First we show that $\MPsc(H_{\omega_2})$ holds in $V[G]$, by assuming toward a contradiction that it fails. Assume $\varphi(a)$ satisfies that:
	$$\text{$V[G] \models ``\varphi(a)$ is $sc$-forceably $sc$-necessary but $\varphi(a)$ is false."}$$
Additionally choose a condition $p \in G$ that forces the above statement.
$\P$ has the $\kappa$-$cc$, which follows from arguing that the separate iterations are (see \textbf{Theorems \ref{thm:forceMP}} and \textbf{\ref{thm:uplifting->RA}}), so at no stage in the iteration is $\kappa$ collapsed. This means that there has to be some stage where $a$ appears.
So there is some stage in the iteration beyond the support of $p$, say $\alpha < \kappa$, where $a \in V_\kappa[G_\alpha]$. Specifically $\varphi(a)$ is an available button at stage $\alpha+1$, since after the rest of the iteration, $\P \rest [\alpha+1, \kappa)$ where $\P=\P_\alpha * \P \rest [\alpha+1, \kappa)$, which is subcomplete, we have that $\varphi(a)$ is $sc$-forceably $sc$-necessary. Thus $V[G_\alpha]$ sees that $\varphi(a)$ is $sc$-forceably $sc$-necessary, so as $\kappa$ is fully reflecting,
	$$V_\kappa[G_\alpha] \models ``\varphi(a) \text{ is $sc$-forceably $sc$-necessary"}.$$ 
From that point on, $\varphi(a)$ continues to be a button, since we have that $\alpha$ is beyond the support of $p$, so $\forces_{\P \rest [\alpha+1, \kappa)} ``\varphi(a)$ is $sc$-forceably $sc$-necessary."
Thus it is dense, in $\P$, for $\varphi(a)$ to be ``pushed" at some point after stage $\alpha$, say $\beta$. 
So we have $\beta < \kappa$ such that there is $\Q$ forcing $\varphi(a)$ to be $sc$-necessary in $V_\kappa[G_\beta]$. Let $H \subseteq \Q$ be generic over $V[G_\beta]$ so that there is some $G_{tail}$ generic for the rest of $\P$ satisfying $V[G_\beta][H][G_{tail}]=V[G]$. The sentence $\varphi(a)$ is now $sc$-necessary in $V_\kappa[G_\beta][H]$. But then since $V_\kappa[G_\beta][H] \preccurlyeq V[G_\beta][H]$, as we are still in an initial segment of the full iteration, we have that $\varphi(a)$ is $sc$-necessary in $V[G_\beta][H]$, by elementarity. Thus since the rest of the iteration is subcomplete, $\varphi(a)$ is true in $V[G_\beta][H][G_{tail}]=V[G]$, contradicting our assumption that $\varphi(a)$ is false in $V[G]$.

Thus $\MPsc(H_{\omega_2})$ holds in $V[G]$. 

In order to show that $\bfRAsc(H_{\omega_2})$ holds in $V[G]$, assume toward a contradiction that it fails. This means we can choose a least rank counterexample, a subcomplete forcing $\Q$ in $V[G]$ that supposedly cannot be resurrected. Let $A \subseteq H_\kappa$, where $\kappa = \omega_2^{V[G]}$, be its associated predicate. Let $\dot{\Q}$ be a name for $\Q$ of minimal rank that necessarily yields a subcomplete poset. Since $\P$ has the $\kappa$-$cc$, there must be a name for the predicate in the extension such that $\dot A \subseteq H_{\kappa}$ with $A = \dot A^G$.

We will argue that $\dot \Q$ appears at stage $\kappa$ of the same exact iteration, except defined in some larger $V_\gamma[G]=V_\gamma^{V[G]}$ where $\gamma$ is inaccessible. 
Use the strong uplifting property of $\kappa$, and code the iteration $\P$ as a subset of $\kappa$, to find a sufficiently large inaccessible cardinal $\gamma$ so that 
	$$\langle V_\kappa, \in, \P, \dot A \rangle \preccurlyeq \langle V_\gamma, \in, \P^*, \dot A^* \rangle,$$
where $\P^*$ is the subcomplete least-counterexample to $\bfRAsc(H_{\omega_2})$ lottery sum $rcs$ iteration of length $\gamma$ as defined in $V_\gamma$. However, obtaining a large enough $\gamma$ again requires a process of closing under least-rank counterexamples, as was done in the proof of \textbf{Theorem \ref{thm:uplifting->RA}}.  The difference in this proof is that not only does $V_\gamma$ need to agree about the rank of least-rank counterexamples to the resurrection axiom throughout the iteration $\P$, it must also compute the least-rank counterexamples to the maximality principle appropriately as well. Since $V_\kappa \in V_\gamma$, and the ranks throughout the maximality iteration were computed in $V_\kappa$, the minimal ranks are guaranteed to be computed properly in $V_\gamma$.
And so we obtain: 
	\begin{eqnarray*} V_\gamma &\models &``\P \text{ is a $sc$ least-counterexample to $\bfRAsc(H_{\omega_2})$ lottery sum iteration"}\\ &&\text{ and } \\ && ``\dot \Q \text{ is $sc$ of minimal rank witnessing the failure of $\bfRAsc(H_{\omega_2})$.}" \end{eqnarray*} 
Indeed we have argued above that $\P^*$ is defined the same way as $\P$ below stage $\kappa$, so we may assume below a condition that $\dot \Q$ may be picked at stage $\kappa$. So below a condition that opts for $\dot \Q$ at the stage $\kappa$ lottery we may say that $\P^*$ factors as $\P * \dot{\Q} * \Ptail$. 

Let $H*\Gtail \subseteq \Q * \Ptail$ be $V[G]$-generic. Letting $G^*=G*H*\Gtail$, this means that $G^* \subseteq \P^*$ generic is over $V$. Moreover, by \textit{Fact} \ref{fact:liftinglemma} the strongly uplifting embedding 
	$$\langle V_\kappa, \in, \P, \dot A \rangle \preccurlyeq \langle V_\gamma, \in, \P^*, \dot A^* \rangle$$ 
lifts to 
	$$\langle V_\kappa[G], \in, \P, \dot A, G \rangle \preccurlyeq \langle V_\gamma[G^*], \in, \P^*, \dot A^*, G^* \rangle $$
in $V[G^*]$. Since $A$ is definable from $\dot A$ and $G$, we may say that 
	$$\langle V_\kappa[G], \in, \P, A \rangle \preccurlyeq \langle V_\gamma[G^*], \in, \P^*, A^* \rangle.$$
As we had in the proof of $\RAsc(H_{\omega_2})$ from an uplifting cardinal, again we have that 
	$$V_\kappa[G]=H_\kappa^{V[G]} = H_{\omega_2}^{V[G]},$$
since $\kappa$ is inaccessible and $\P$ has the $\kappa$-$cc$, and by \textbf{Lemma \ref{lem:lengthcollapse}}. We can argue the same way as above, replacing $\kappa$ with $\gamma$, to get that 
	$$V_\gamma[G^*]=H_\gamma^{V[G^*]} = H_{\omega_2}^{V[G^*]}.$$ 
This establishes 
	$$\langle H_{\omega_2}^{V[G]}, \in, A \rangle \preccurlyeq \langle H_{\omega_2}^{V[G^*]}, \in, A^* \rangle,$$ 
so $\bfRAsc(H_{\omega_2})$ in fact holds as desired. 
\end{proof}



In fact $\bfRAsc(H_{\omega_2}) + \MPsc(H_{\omega_2})$ is equiconsistent with the existence of a strongly uplifting fully reflecting cardinal.

\begin{thm}
If both $\bfRAsc(H_{\omega_2})$ and $\MPsc(H_{\omega_2})$ hold, then $\omega_2$ is strongly uplifting fully reflecting in $L$. \end{thm}
\begin{proof}
Assume that both $\bfRAsc(H_{\omega_2})$ and $\MPsc(H_{\omega_2})$ hold.

We already have established that since $\MPsc(H_{\omega_2})$ holds, $\omega_2$ is fully reflecting in $L$ by \textbf{Lemma \ref{lem:MP->fullyreflectinginL}}.

Furthermore we claim that $\kappa=\omega_2^V$ is strongly uplifting in $L$. Fix any subset $A \subseteq H_\kappa$ in $L$, and let $\theta$ be a large ordinal in $L$. So $A \subseteq L_\kappa$ and $A \in L$. Thus $A \in L_\xi$ for some $\xi < (\kappa^+)^L$. We need to find a regular $\kappa' > \theta$ and a set $A^* \subseteq H^L_{\kappa'}$ such that 
	$$\langle H^L_\kappa , \in, A \rangle \preccurlyeq \langle H^L_{\kappa'}, \in, A^* \rangle$$ 
is an elementary extension. Note that here we are using the characterization of strongly uplifting using $H_\kappa$ and $H_{\kappa'}$, and as is outlined after the proof of Theorem 3 in \cite{Hamkins:2014yu}, we do not have to ensure that $\kappa'$ is inaccessible. 

Code $L_\xi$ by a relation $E \subseteq \kappa \times \kappa$, so that we have $\pi_E: \langle \kappa, E \rangle \cong \langle L_\xi, \in \rangle$.
Let $\Q = \Coll(\omega_1, \theta)$. By $\bfRAsc(H_{\omega_2})$ we have a further forcing $\dot{\R}$ with $\forces_{\Q} ``\dot{\R} \text{ is subcomplete}"$ such that if $g*h \subseteq \Q * \dot{\R}$ is $V$-generic, then there is an $E' \in V[g*h]$ such that 
	\begin{equation} \label{eqn:H^VelemH^V[g*h]} \langle H_{\omega_2}^V, \in, E \rangle \preccurlyeq \langle H_{\omega_2}^{V[g*h]}, \in, E' \rangle. \end{equation}
Since $\theta$ was forced to have cardinality $\omega_1$ in $V[g]$, it must be that $\theta < \kappa' = \omega_2^{V[g*h]}$. Since $\kappa'$ is regular in $V[g*h]$ we have that $\kappa'$ is regular in $L$. 

We have that $E' \subseteq \kappa' \times \kappa'$ is well founded, and we have $\pi_{E'}: \langle \kappa', E' \rangle \cong \langle L_{\xi'}, \in \rangle$. Thus there is an elementary embedding $$\pi_{E'} \circ \pi_E^{-1}=\sigma: \langle L_\xi, \in \rangle \prec \langle L_{\xi'}, \in \rangle.$$ By (\ref{eqn:H^VelemH^V[g*h]}) we have that $\pi_E^{-1} \rest \kappa = \pi_{E'}^{-1} \rest \kappa$, and thus $\sigma \rest L_\kappa = \id$. Therefore $\sigma \rest L_\kappa$ gives rise to the elementary embedding:
	$ \langle H_\kappa^L, \in, A \rangle \preccurlyeq \langle H_{\kappa'}^L, \in, A^* \rangle$, where $A^*=\sigma(A)$, as desired.
\end{proof}

\chapter{Generalized Diagonal Prikry Forcing}
\label{chap:GenDiagonalPrikryForcing}
Jensen \cite[Section 3.3]{Jensen:2012fr} shows that Prikry forcing and Namba forcing (under $\CH$) are subcomplete. Below we use an adaptation of Jensen's proof showing that Prikry forcing is subcomplete to see that some kinds of generalized diagonal Prikry forcing, in particular those Prikry forcings that we refer to here as generalized diagonal Prikry forcing, studied by Fuchs \cite{Fuchs:2005kx}, are subcomplete.

\begin{defn}
Let $D$ be an infinite discrete set of measurable cardinals, meaning a set of measurable cardinals that does not contain any of its limit points. For $\kappa \in D$ let $U(\kappa)$ be a normal measure on $\kappa$, and let $\U$ denote the sequence of the $U(\kappa)$'s.

Define $\D=\D(\U)$, \emph{\textbf{generalized diagonal Prikry forcing}} from the list of measures $\U$, by taking conditions of the form  
$( s, A )$ satisfying the following:
\begin{itemize}
	\item The \textit{stem} of the condition, $s$, is a function with domain in $[D]^{<\omega}$ taking each measurable cardinal $\kappa \in \dom(s)$ to some ordinal $s(\kappa) < \kappa$.
	\item The \textit{upper part} of the condition, $A$, is a function with domain $D \setminus {\dom(s)}$ taking each measurable cardinal $\kappa \in \dom(A)$ to some measure-one set $A(\kappa) \in U(\kappa)$.
\end{itemize}
The extension relation on conditions in $\D$ is defined so that $( s, A ) \leq ( t, B )$ so long as 
\begin{itemize}
	\item $s \supseteq t$.
	\item The points in $s$ not in $t$ come from $B$, i.e., for all $\kappa \in \dom(s) \setminus \dom(t)$, $s(\kappa) \in B(\kappa)$.
	\item For all $\kappa \in \dom(A)$, $A(\kappa) \subseteq B(\kappa)$.
\end{itemize}
If $G$ is a generic filter for $\D$, then its associated $\D$-generic sequence is \[S = S_G = \bigcup \set{ s }{ \exists A \ ( s, A ) \in G }.\qedhere\]
\end{defn}

Note that our definition of $\D(\U)$ differs from that as in Fuchs \cite{Fuchs:2005kx}. The main difference is that here we only add one point below each measurable cardinal $\kappa \in D$, which is done for simplicity's sake. It is not hard to see that the following theorem showing generalized diagonal Prikry forcing is subcomplete also shows that the forcing adding countably many points below each measurable cardinal in $D$ (where the conditions consist of finite stems) is subcomplete. Adding countably many points below each measurable cardinal in $D$ would collapse the cofinality of each $\kappa \in D$ to be $\omega$, as one expects of a Prikry-like forcing. 

Also in the above definition we haven't enforced that the stem of our conditions only consist of ordinals that are wedged between successive measurables in $D$; ie. for $\kappa \in D$, we do not explicitly insist that $s(\kappa) \in [\sup(D\cap \kappa), \kappa)$. However, it is dense in $\D(\U)$ for the conditions to be that way, since we can always strengthen conditions by restricting their upper parts to a tail. Thus in the following characterization, we may freely add the condition to the following genericity condition on $\D(\U)$.
Thie following is a genericity criterion on generalized diagonal Prikry forcing similar to the Mathias criterion for Prikry forcing. It was shown in \cite[Theorem 1]{Fuchs:2005kx}. 

\begin{fact}[Fuchs] \label{fact:diagprikrymathias} Let $D$ be an infinite discrete set of measurable cardinals, with $\U$ a corresponding list of measures $\langle U(\kappa) \;|\; \kappa \in D \rangle$. Then an increasing sequence of ordinals $S = \langle S(\kappa) \;|\; \kappa \in D \rangle$, where for each $\kappa \in D$ $$\sup(D \cap \kappa) < S(\kappa) < \kappa,$$ is a $\D(\U)$-generic  sequence if and only if for all $\mathcal X = \langle X_\kappa \in U(\kappa) \;|\; \kappa \in D \rangle$, the set $$\{ \kappa \in D \;|\; S(\kappa) \notin X_\kappa\} \text{ is finite.}$$
\end{fact}

\begin{thm} Let $D$ be an infinite discrete set of measurable cardinals. 
Let $\U = \langle U(\kappa) \;|\; \kappa \in D \rangle$ be a list of measures associated to $D$. Then $\D=\D(\U)$ is subcomplete. \end{thm}
\begin{proof} 
Let $\theta >> \delta(\D) =\delta$ be large enough, so that $[\delta]^{<\omega_1} \in H_\theta$. 

It must be the case that $\delta \geq \sup D$: to see this, suppose instead that there is a dense $E \subseteq \D$ such that $\sup D \geq \kappa^* >|E|$ for some $\kappa^* \in D$. Then for each condition $(s,A) \in E$ either $\kappa^* \in \dom s$ or $\kappa^* \in \dom A$. So taking $E^*=\set{ (s, A) \in E }{ \kappa \in \dom s } \subseteq E$, since $|E^*| < \kappa^*$ as well, there is an $\alpha<\kappa^*$ such that $\sup_{(s,A) \in E^*} s(\kappa^*)$. Let $p=(t, B) \in \D$ be defined so that $t(\kappa^*)=\alpha$ and $B(\kappa)=\kappa$ for all $\kappa \in D \setminus \{ \kappa^*\}$. Then $p$ cannot be strengthened by any condition in $E$ since $\kappa^*$ is not in any of the stems of conditions in $E$. So dense subsets of $\D$ must have size at least $\sup D$.

Let $\nu = \delta^+$. Let $\kappa(0)$ be the first measurable cardinal in $D$. 

In order to show that $\D$ is subcomplete, suppose we are in the following situation: \begin{itemize}
	\item $\D \in H_\theta \subseteq N = L_\tau[A] \models \ZFC^-$ where $\tau>\theta$ and $A \subseteq \tau$
	\item $\sigma: \N \cong X \preccurlyeq N$ where $X$ is countable and $\N$ is full
	\item $\sigma(\overline \theta, \overline{\D}, \overline{\U}, \overline c)=\theta, \D, \U, c$ for some $c \in N$.
\end{itemize}
By our requirement on $\theta$, we've ensured that $N$ is closed under countable sequences of ordinals less than $\delta$.

In what follows, we will be taking a few different transitive liftups of restrictions of $\sigma$, and it will useful to keep track of embeddings between $\N$ and $N$ pictorially. Although it's not extraordinarily illuminating at this point in our discussion, the following figure shows the situation we are currently in, where $\sigma(\overline \delta)=\delta$, $\overline \nu = {\overline \delta^+}^{\N}$, $\overline D$ is the discrete set of measurables in $\N$ that each measure in $\overline \U$ comes from, and $\overline \kappa(0)$ is the first measurable in $\overline D$, in the sense of $\N$. 

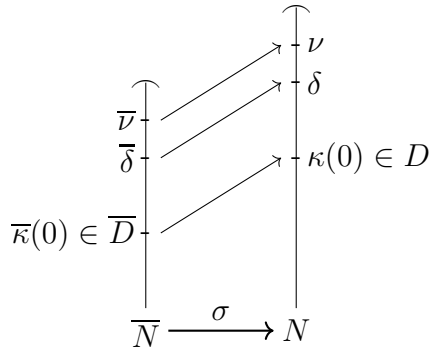
\begin{center}
\begin{tikzpicture}
\draw (-1,0)--(-1,3);
\node [below] at (-1,0) {$\N$};
\node at (-1,2.95) {$\frown$};

\node [left] at (-1,1) {$\overline \kappa(0) \in \overline D$};
\node at (-1,1) {-};
\node [left] at (-1,2) {$\overline \delta$};
\node at (-1,2) {-};
\node [left] at (-1,2.5) {$\overline \nu$};
\node at (-1,2.5) {-};

\draw [->] (-.8,1)--(.8,2);
\draw [->] (-.8,2)--(.8,3);
\draw [->] (-.8,2.5)--(.8,3.5);

\node [right] at (1,2) {$\kappa(0) \in D$};
\node at (1,2) {-};
\node [right] at (1,3) {$\delta$};
\node at (1,3) {-};
\node [right] at (1,3.5) {$\nu$};
\node at (1,3.5) {-};

\draw (1,0)--(1,4);
\node [below] at (1,0) {$N$};
\node at (1,3.95) {$\frown$};

\draw [->, thick] (-.7,-0.3)--(.7,-0.3);
\node [above] at (0,-0.3) {$\sigma$};
\end{tikzpicture}
\end{center}

Toward showing that $\D$ is subcomplete, we are additionally given some $\G \subseteq \overline{\D}$ that is generic over $\N$. Rather than working with $\G$, we will work with $\S = \seq{ \S(\overline \kappa) }{ \overline \kappa \in \overline D }$, its associated $\overline \D$-generic sequence. We must show following, where $C=\Sk{N}{\delta}{X}$:

\begin{claim}[Main] There is a $\D$-generic sequence $S$ and a map $\sigma' \in V[S]$ such that: \begin{enumerate}
	\item $\sigma': \N \prec N$
	\item $\sigma'(\overline \theta, \overline{\D}, \overline{\U}, \overline c)=\theta, \D, \U, c$
	\item $\sk{N}{\delta}{\sigma'} = C$	
	\item $\sigma' `` \S \subseteq S$
\end{enumerate} \end{claim}
\begin{proof}[Pf.]  This proof uses Barwise theory (key definitions, facts, and theorems are summarized Section \ref{sec:BarwiseTheory}) heavily, and ultimately amounts to showing that a certain 
$\in$-theory, $\mathcal T$, which posits the existence of such a $\sigma'$, is consistent. Such an embedding $\sigma'$ can only possibly exist in a suitable generic extension, $V[S]$, where $S$ is $\D$-generic sequence that we will find later. 
Once we have such a suitable $V[S]$, we will use Barwise theory to find an appropriate admissible structure in $V[S]$ for which the theory $\mathcal T$, positing the existence of such a suitable $\sigma'$, defined below, has a model.

\begin{samepage}
For an admissible structure $\mathfrak M$ with $S, \overline S, \sigma, N, \theta, \D, \U, c \in \mathfrak M$ let the infinitary $\in$-theory $\mathcal T(\mathfrak M)$ be defined over $\mathfrak M$ as follows: 
\begin{description}
	\item[predicates] $\in$ 
	\item[constants] $\dot{\sigma}, \underline x$ for $x \in \mathfrak M$
	\item[axioms] \begin{itemize} \item $\ZFC^-$ and \textsf{Basic Axioms}.
		\item $\dot \sigma : \overline{\underline N} \prec \underline N$
		\item $\dot{\sigma}(\overline{\underline{\theta}}, \overline{\underline{\D}}, \overline{\underline{\U}}, \overline{\underline c})=\underline{\theta}, \underline{\D}, \underline{\U}, \underline{c}$
		\item $\sk{\underline N}{\underline{\delta}}{\dot \sigma} = \sk{\underline N}{\underline \delta}{\underline \sigma}$
		\item $\dot \sigma ``\overline{\underline S} \subseteq \underline S$.
	\end{itemize}
\end{description}
\end{samepage}

The $\in$-theory is $\Sigma_1(\mathfrak M)$, since all of the axioms are $\mathfrak M$-finite except for the \textsf{Basic Axioms}, which altogether are $\mathfrak M$-$re$ as each of them are $\mathfrak M$-finite.

We need to find an appropriate $\D$-generic sequence $S$ and a suitable admissible structure $\mathfrak M$ containing $S$ so that $\mathcal T(\mathfrak M)$ is consistent. To do this we use transitive liftups and Barwise theory. Transitive liftups will give us the consistency of certain embeddings that approximate the one we are looking for, and we will rely on Barwise Completeness (\textit{Fact} \ref{fact:completeness}) to obtain the existence of a model with our desired properties. 

Toward this end, let's take what will turn out to be our first transitive liftup, which is in some sense ensuring the consistency of having property \textsl{\textbf{3}} of our main claim.

Let $k_0 : N_0 \cong C$ where $N_0$ is transitive, and set $\sigma_0 = k_0^{-1} \circ \sigma$ and $\sigma_0(\overline \theta, \overline{\D}, \overline{\U}, \overline c) = \theta_0, \D_0, \U_0, c_0.$
Since $\delta \subseteq C$ and $N_0$ is transitive, $\sigma_0(\overline \delta)=\delta$. 

Indeed $N_0$ is actually a transitive liftup:

\begin{claimno} $\langle N_0, \sigma_0 \rangle$ is the transitive liftup of $\langle \N, \sigma \upharpoonright H_{\overline \nu}^{\N} \rangle$. \end{claimno}
\begin{proof}[Pf.] Recall that $\nu=\delta^+$, and $\overline \nu={\overline \delta^+}^{\N}$. It must be shown that the embedding $\sigma_0: \overline N \prec N_0$ is $\overline \nu$-cofinal and that $\sigma_0 \upharpoonright H_{\overline \nu}^{\N}=\sigma \upharpoonright H_{\overline \nu}^{\overline N}$. 

To see that $\sigma_0$ is $\overline \nu$-cofinal, let $x \in N_0$. Then $k_0(x) \in C = \Sk{N}{\delta}{X}$ so $k_0(x)$ is uniquely $N$-definable from $\xi < \delta$ and $\sigma(\overline z)$ where $\overline z \in \N$. In other words, $$k_0(x) = \text{that } y \text{ such that } N \models \varphi(y, \xi, \sigma(\overline z)).$$ Let $u \in \N$ be defined as 
	$$u=\set{ w \in \N }{  w=\text{that } y \text{ such that } \N \models \varphi(y, \zeta, \overline z) \text{ for some } \zeta < \overline \delta }.$$
Certainly $u$ is non-empty by elementarity, since $k_0(x) \in \sigma(u).$
Furthermore, $|u| \leq \overline \delta < \overline \nu$ since every $w \in u$ is unique, and needs a corresponding $\zeta<\overline \delta$ to satisfy the formula $\varphi$ with.
Thus $x \in k_0^{-1}(\sigma(u))=\sigma_0(u)$ with $|u| < \overline \nu$ in $\N$, as desired. 

Since $X \cup \delta \subseteq C$, the Skolem hull in $N$, we know that $\sigma``H_{\overline \nu}^{\N} \subseteq C$. Thus $k_0^{-1} \upharpoonright \sigma``H_{\overline \nu}^{\overline N} = \text{id}$. Therefore $\sigma_0 \rest H_{\overline \nu}^{\N}=\sigma \rest H_{\overline \nu}^{\overline N}$, finishing the proof of the claim.
\end{proof}

Since $\overline \nu$ is regular in $\N$, in $N_0$ so is $\nu_0=\sigma_0(\overline \nu)=\sup \sigma_0``\overline \nu$. By Interpolation (\textit{Fact} \ref{fact:Interpolation}), we may say that $k_0$ is defined by 
	$$k_0: N_0 \prec N \text{ where } k_0 \circ \sigma_0=\sigma \text{ and } k_0 \rest \nu_0 = \id.$$
In particular, $\nu_0$ is the critical point of $k_0$, which is continuous below $\nu_0$. Thus we are in a situation that we will represent with the following  diagram:
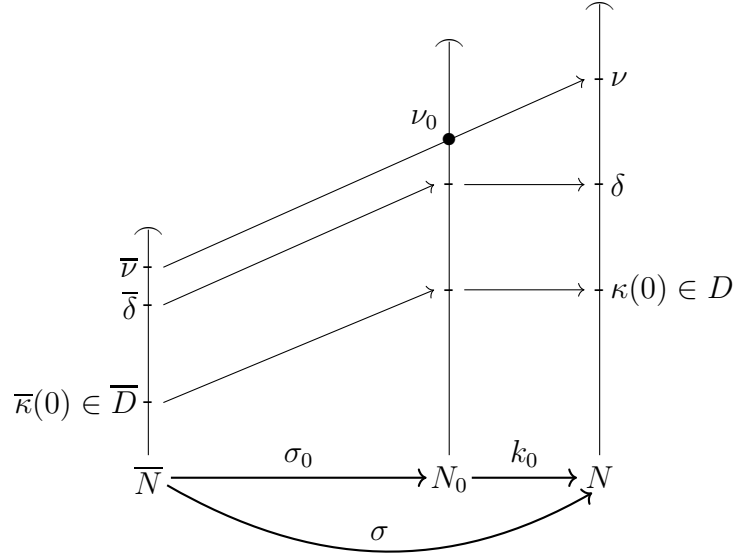
\begin{center}
\begin{tikzpicture}
\draw (-3,0)--(-3,3);
\node [below] at (-3,0) {$\N$};
\node at (-3,2.95) {$\frown$};

\node [left] at (-3,0.7) {$\overline \kappa(0) \in \overline D$};
\node at (-3,0.7) {-};
\node [left] at (-3,2) {$\overline \delta$};
\node at (-3,2) {-};
\node [left] at (-3,2.5) {$\overline \nu$};
\node at (-3,2.5) {-};

\draw [->] (-2.8,0.7)--(0.8,2.2);
\draw [->] (-2.8,2)--(0.8,3.6);

\node at (1,2.2) {-};

\node at (1,3.6) {-};
\node [above left] at (1,4.2) {$\nu_0$};
\node at (1,4.2) {$\bullet$};

\draw (1,0) -- (1,5.5);
\node [below] at (1,0) {$N_0$};
\node at (1,5.45) {$\frown$};

\draw [->] (1.2,2.2)--(2.8,2.2);
\draw [->] (1.2,3.6)--(2.8,3.6);
\draw [->] (-2.8,2.5)--(2.8,5);

\node [right] at (3,2.2) {$\kappa(0) \in D$};
\node at (3,2.2) {-};
\node [right] at (3,3.6) {$\delta$};
\node at (3,3.6) {-};
\node [right] at (3,5) {$\nu$};
\node at (3,5) {-};

\draw (3,0)--(3,6);
\node [below] at (3,0) {$N$};
\node at (3,5.95) {$\frown$};

\draw [->, thick] (-2.7,-0.3)--(0.7,-0.3);
\draw [->, thick] (1.3, -0.3)--(2.7,-0.3);
\node [above] at (-1,-0.3) {$\sigma_0$};
\node [above] at (2,-0.3) {$k_0$};
\draw[->,thick] (-2.75,-0.4) to [out = -30, in =-150] node[above]{$\sigma$} (2.9, -0.5);
\end{tikzpicture}
\end{center}

Already, we can say that $\sigma_0$ looks like it has one nice property: $\sk{N_0}{\delta}{\sigma_0}= C$, which somewhat looks like \textsl{\textbf{3}} of the main claim. However, we have not yet performed forcing, $\sigma_0$ is definable in $V$, and we still need to find a way to extend the generic sequence $\S$ to a $\D$-generic sequence over $N$. We still have a lot more work to do before finding $\sigma'$.

We shall define another $\in$-theory, $\mathcal L_*$ that will assist us in obtaining the diagonal Prikry extension $V[S]$ we need to satisfy our main claim. In order to do this, we will take another transitive liftup and apply Transfer (\textit{Fact} \ref{fact:Transfer}), in order to see that this new $\in$-theory is consistent over an admissible structure on $N_0$. 

Since we will be referring to the same $\in$-theory over two different transitive liftups, I would like to think of $``*"$ as a kind of placeholder for a transitive liftup in the following definition.

Suppose that $\langle N_*, \sigma_* \rangle$ is a transitive liftup of $\N$ along with some reasonable restriction of $\sigma$, ie. the liftup of $\langle \N, \sigma \rest H_{\alpha}^{\N} \rangle$, where $\alpha \geq \overline \kappa(0)$ is regular in $\N$, and say $$\sigma_*(\overline \theta, \overline{\D}, \overline{\U}, \overline c) = \theta_*, \D_*, \U_*, c_*.$$ Recall that $\delta_{N_*}$ is the least such that $L_{\delta_{N_*}}(N_*)$ is admissible. 

Define the infinitary $\in$-theory $\mathcal L(N_*, \sigma_*)=\mathcal L_*$ as follows: 

\begin{description}
	\item[predicates] $\in$ 
	\item[constants] $\mathring{\sigma}, \mathring S, \underline x$ for $x \in L_{\delta_{N_*}}(N_*)$
	\item[axioms] \begin{itemize} \item $\ZFC^-$ and \textsf{Basic Axioms}
		\item $\mathring \sigma : \underline \N \prec \underline{N_*}$ is $\underline{\overline \kappa(0)}$--cofinal
		\item $\mathring{\sigma}(\overline{\underline{\theta}}, \overline{\underline{\D}}, \overline{\underline{\U}}, \overline{\underline c})=\underline{\theta_*}, \underline{\D_*}, \underline{\U_*}, \underline{c_*}$
		\item $\mathring S$ is a $\underline{\D_*}$-generic sequence over $\underline{N_*}$
		\item $\mathring \sigma ``\overline{\underline S} \subseteq \mathring S$.
	\end{itemize}
\end{description} 

As defined, we have that $\mathcal L_*$ is a $\Sigma_1(L_{\delta_{N_*}}(N_*))$-theory, since altogether the \textsf{Basic Axioms} are $\Sigma_1(L_{\delta_{N_*}}(N_*))$.

We claim that the theory is consistent.

\begin{claimno} $\mathcal L_0=\mathcal L(N_0, \sigma_0)$ is consistent. \end{claimno}
\begin{proof}[Pf] Of course, it is not the case that $\sigma_0$ is $\overline \kappa(0)$-cofinal - all we know is that it is $\overline \nu$-cofinal. However, we know how to find an elementary embedding that is $\overline \kappa(0)$ cofinal: by taking a suitable transitive liftup.

Let $\langle N_1, \sigma_1 \rangle$ be the transitive liftup of $\langle \N, \sigma \rest H^{\N}_{\overline \kappa(0)} \rangle$, which exists by Interpolation (\textit{Fact} \ref{fact:Interpolation}). So we have that $\sigma_1 \rest H^\N_{\overline \kappa(0)} = \sigma \rest H^\N_{\overline \kappa(0)}$. Let $$\sigma_1(\overline \theta, \overline{\D}, \overline{\U}, \overline c) = \theta_1, \D_1, \U_1, c_1.$$ 
Since $\sigma_1 \rest H^\N_{\overline \kappa(0)} = \sigma \rest H^\N_{\overline \kappa(0)} = \sigma_0 \rest H^\N_{\overline \kappa(0)}$, we also have a unique $$ \text{$k_1: N_1 \prec N_0$ where $k_1\circ \sigma_1 = \sigma_0$ and $k_1 \rest \kappa_1(0) = \text{id}$ where $k_1(0)=k_1(\kappa(0))$.}$$ 

Indeed $k_1$ is continuous below $\kappa_1(0)$. We illustrate the final picture below, with $\sigma$ and all of the relevant transitive liftups.
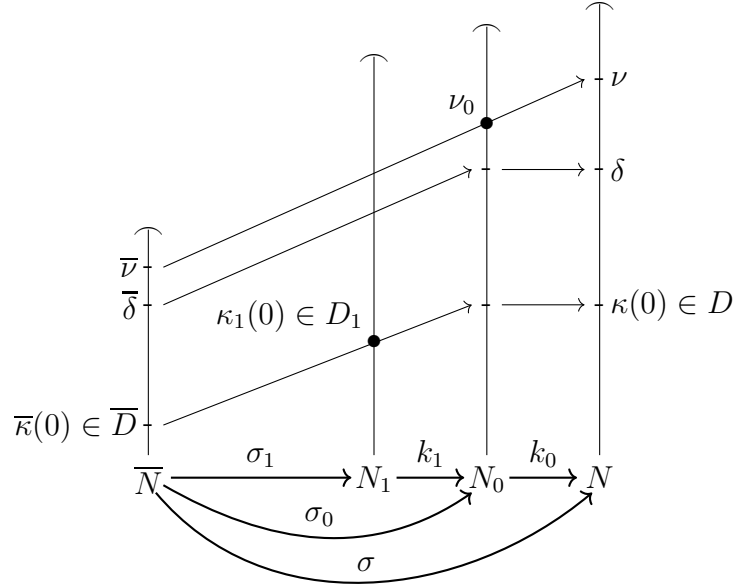
\begin{center}
\begin{tikzpicture}
\draw (-3,0)--(-3,3);
\node [below] at (-3,0) {$\overline N$};
\node at (-3,2.95) {$\frown$};

\node [left] at (-3,0.4) {$\overline \kappa(0) \in \overline D$};
\node at (-3,0.4) {-};
\node [left] at (-3,2) {$\overline \delta$};
\node at (-3,2) {-};
\node [left] at (-3,2.5) {$\overline \nu$};
\node at (-3,2.5) {-};

\draw [->] (-2.8,0.4)--(1.3,2);

\draw (0,0)--(0,5.3);
\node [below] at (0,0) {$N_1$};
\node at (0,5.25) {$\frown$};
\node at (0,1.5) {$\bullet$};
\node [above left] at (0,1.5) {$\kappa_1(0) \in D_1$};

\draw [->] (-2.8,2)--(1.3,3.8);

\draw (1.5,0) -- (1.5,5.7);
\node [below] at (1.5,0) {$N_0$};
\node at (1.5,5.65) {$\frown$};
\node at (1.5,2) {-};
\node at (1.5,3.8) {-};
\node [above left] at (1.5,4.4) {$\nu_0$};
\node at (1.5,4.4) {$\bullet$};

\draw [->] (1.7,2)--(2.8,2);
\draw [->] (1.7,3.8)--(2.8,3.8);

\draw [->] (-2.8,2.5)--(2.8,5);

\node [right] at (3,2) {$\kappa(0) \in D$};
\node at (3,2) {-};
\node [right] at (3,3.8) {$\delta$};
\node at (3,3.8) {-};
\node [right] at (3,5) {$\nu$};
\node at (3,5) {-};

\draw (3,0)--(3,6);
\node [below] at (3,0) {$N$};
\node at (3,5.95) {$\frown$};

\draw [->, thick] (-2.7,-0.3)--(-0.3,-0.3);
\draw [->, thick] (0.3, -0.3)--(1.2,-0.3);
\draw [->, thick] (1.8, -0.3)--(2.7,-0.3);
\draw[->, thick] (-2.8,-0.4) to  [out=-30, in=-145] node[above]{$\sigma_0$} (1.3,-0.5);
\draw[->,thick] (-2.9,-0.5) to [out = -50, in =-140] node[above]{$\sigma$} (2.9, -0.5);

\node [above] at (-1.5,-0.3) {$\sigma_1$};
\node [above] at (0.75,-0.3) {$k_1$};
\node [above] at (2.25, -0.3) {$k_0$};
\end{tikzpicture}
\end{center} 
We first show that $\mathcal L_1=\mathcal L(N_1, \sigma_1)$ is consistent, by seeing that it has a model. To do this, we will find a sequence extending $\sigma_1``\S$ that is $\D_1$-generic over $N_1$. Then we will use the Transfer Lemma to see that this transfers to the consistency of $\mathcal L_0$. 

First, force with $\D_1$, which is diagonal Prikry over $N_1$, to obtain a diagonal Prikry sequence $S_1'$. Define, in $V[S_1']$, a new sequence $S_1$ as follows:

$$S_1(\kappa) = \begin{cases} S_1'(\kappa) &\text{ if } \kappa \in D_1 \setminus \sigma_1``\overline D \\
					\sigma_1(\overline S(\overline \kappa)) &\text{ if } \kappa = \sigma_1(\overline \kappa) \in \sigma_1``\overline D. \end{cases}$$

\begin{claim} The sequence $S_1$ is a $\D_1$-generic sequence over $N_1$. \end{claim}
\begin{proof}[Pf.]
We will show that $S_1$ satisfies the generalized diagonal Prikry genericity criterion (\textit{Fact} \ref{fact:diagprikrymathias}) over $N_1$. To do this, let $\mathcal X = \seq{ X_\kappa \in U_1(\kappa) }{ \kappa \in D_1 }$, with $\mathcal X \in N_1$, be a sequence of measure-one sets in the sequence of measures $\U_1$.

Note first that $S_1'$ is a generic sequence, it already satisfies the generalized diagonal Prikry genericity criterion, namely:
$$\set{ \kappa \in D_1}{ S_1'(\kappa) \notin X_\kappa } \text{ is finite.}$$
Recall that $\S = \seq{ \S(\overline \kappa) }{ \overline \kappa \in \overline D }$ is a $\overline{\D}$-generic sequence as well.
We need to see that in addition,
$$\set{ \overline \kappa \in \overline D}{\sigma_1(\S(\overline \kappa)) \notin X_{\sigma_1(\overline \kappa)} }  \text{ is finite,}$$
since then 
$$\set{ \kappa \in D_1}{ S_1(\kappa) \notin X_\kappa } = \set{ \kappa \in D_1 \setminus \sigma_1``\overline D}{ S_1'(\kappa) \notin X_\kappa } \cup \set{ \kappa = \sigma_1(\overline \kappa) \in \sigma_1``\overline D}{\sigma_1(\S(\overline \kappa)) \notin X_\kappa }$$
is finite as well, completing the proof as desired.

By the $\overline{\kappa}(0)$-cofinality of $\sigma_1$, there is some $w \in \N$ such that $\mathcal X \in \sigma_1(w)$, where $|w| < \overline{\kappa}(0)$ in $\N$. Thus in $N_1$, we have that $|\sigma_1(w)| < \kappa_1(0)$. 
We may assume that $w$ consists of functions $f \in \prod_{\overline \kappa \in \overline D} \overline U(\overline \kappa)$.
So for each $\kappa \in \sigma_1``\overline D$, we have that $X_\kappa \in \sigma_1(w)_\kappa = \set{\sigma_1(f)(\overline \kappa) }{ f \in \prod_{\overline \kappa \in \overline D} \overline U(\overline \kappa) \ \land \ f \in w }$ and also $|\sigma_1(w)_\kappa|<\kappa_1(0).$ So all $\kappa \in \sigma_1``\overline D$ of course satisfy $\kappa \geq \kappa_1(0)$ and thus by the $\kappa$-completeness of $U_1(\kappa)$, we have that $W_\kappa := \cap \, \sigma_1(w)_\kappa \in  U_1(\kappa).$
So we have established that $\mathcal W$, the sequence of $W_\kappa$ for $\kappa \geq \kappa_1(0)$, is also a sequence of measure-one sets in $N_1$. Note in addition that for $\kappa \in \sigma_1``\overline D$, we have that $W_\kappa \subseteq X_\kappa$. 

By elementarity, for each $\overline \kappa \in \overline D$, we have $\overline W_{\overline \kappa} = \cap \set{f(\overline \kappa) }{ f \in \prod_{\overline \kappa \in \overline D} \overline U(\overline \kappa) \ \land \ f \in w }$ is a measure-one set in $\overline U(\overline \kappa)$ and we also have that $\sigma_1(\overline W_{\overline \kappa}) = W_{\sigma_1(\overline \kappa)}$. Moreover, 
$$\set{\overline \kappa \in \overline D}{\overline S(\overline \kappa) \notin \overline W_{\overline \kappa}} \text{ is finite}$$ by the generalized diagonal Prikry genericity criterion for $\overline{\D}$, which must be satisfied by $\overline S$.
Thus by elementarity,
$$\set{ \overline \kappa \in \overline D }{ \sigma_1(\S(\overline \kappa)) \notin W_{\sigma_1(\overline \kappa)} } \supseteq \set{ \overline \kappa \in \overline D }{ \sigma_1(\S(\overline \kappa)) \notin X_{\sigma_1(\overline \kappa)}} \text{ is finite,}$$
as is desired, completing the proof of our claim.
\end{proof}

Moreover we have now shown that $\mathcal L(N_1, \sigma_1) = \mathcal L_1$ is consistent by Barwise Correctness (\textit{Fact} \ref{fact:correctness}), since we have just shown that $\langle H_{\delta}; \sigma_1, S_1 \rangle$ is a model of $\mathcal L_1$.

Let's check that we may now apply Transfer (\textit{Fact} \ref{fact:Transfer}) to the embedding $k_1:N_1 \prec N_0$. By \textbf{Lemma \ref{lem:liftupfull}} we have that $N_1$ is almost full. We also have that $\mathcal L_1 = \mathcal L(L_{\delta_{N_1}}(N_1))$ is $\Sigma_1$ over parameters $$\text{$N_1$ and $\overline \theta, \overline{\D}, \overline{\U}, \overline c, \theta_1, \D_1, \U_1, c_1 \in N_1$}$$ while $\mathcal L_0=\mathcal L(L_{\delta_{N_0}}(N_0))$ is $\Sigma_1$ over parameters 
	$$\text{$N_0$ and $k_1(\overline \theta, \overline{\D}, \overline{\U}, \overline c, \theta_1, \D_1, \U_1, c_1) \in N_0$.}$$ 
Furthermore $k_1$ is cofinal, since for each element $x \in N_0$, as $\sigma_0$ is cofinal, there is $u \in \N$ such that $x \in \sigma_0(u)$. Thus $\sigma_1(u) \in N_1$, and moreover $x \in k_1(\sigma_1(u))=\sigma_0(u)$. Therefore, we have that since $\mathcal L_1$ is consistent, $\mathcal L_0$ is consistent as desired. This completes the proof of \textit{Claim} 2. \end{proof}

From the consistency of $\mathcal L_0$, we would now like to use Barwise Completeness (\textit{Fact} \ref{fact:completeness}) to obtain a model of $\mathcal L_0$. To do this, we need the admissible structure the theory is defined over to be countable. So let's work in $V[F]$, a generic extension that collapses $L_{\delta_{N_0}}(N_0)$ to be countable. Then by Barwise Completeness, $\mathcal L_0$ has a solid model 
	$$\mathfrak A = \langle \mathfrak A; \mathring{S}^{\mathfrak A}, \mathring{\sigma}^{\mathfrak A} \rangle$$ such that $$\Ord \cap \wfc(\mathfrak A) = \Ord \cap L_{\delta_{N_0}}(N_0).$$ 
Thus we have that $\mathring \sigma^{\mathfrak A}: \overline{\underline N}^{\mathfrak A} \prec \underline{N_0}^{\mathfrak A}$. By the \textsf{Basic Axioms} we have that $\overline{\underline N}^{\mathfrak A}=\N$ and $N_0=\underline{N_0}^{\mathfrak A}$. Thus we may say that  $\mathring \sigma^{\mathfrak A}: \N \prec N_0$. 

Let $S = \mathring{S}^{\mathfrak A}$ and $\overset{*} {\sigma}=k_0 \circ \mathring{\sigma}^{\mathfrak A}$. 

Then $S$ is a $\D_0$-generic sequence over $N_0$, and as $k_0:N_0 \cong C$ we also have that $k_0``S$ is $C$-generic for $\D$. We need to see that $S$ is $\D$-generic over $V$. To do this, let $\mathcal X = \seq{ X_\kappa \in U_1(\kappa) }{ \kappa \in D }$ be a sequence of measure-one sets in the sequence of measures $\U$. We will verify the generalized diagonal Prikry genericity criterion. To do this, let $E \subseteq \D$ be dense and have size $\delta$ with $E \in C$. Since $\delta \subseteq C$, we have that $E \subseteq C$ as well. Find a condition $(s,A) \in E$ that strengthens $(\emptyset, \mathcal X)$. Thus for $\kappa \in \dom A$, we have that $A(\kappa) \subseteq X_\kappa$. Define a sequence of measure-one sets $B$ in $C$ so that 
$$B(\kappa) = \begin{cases} A(\kappa) &\text{if} \ \kappa \in \dom A \\ \kappa &\text{if} \ \kappa \in \dom s. \end{cases}$$
So we have that $B$ is a sequence of measure-one sets in $C$. So $\set{ \kappa \in D }{S(\kappa) \notin B(\kappa) }$ is finite. Thus $\set{\kappa \in D}{ S(\kappa) \notin X_\kappa}$ is finite.

This will be the $\D$-generic sequence we need to satisfy our main claim. We will see in the following claim that $\overset{*} \sigma$ has all of the desired properties of our main claim, but it fails to be in $V[S]$, which is what we need. However, based on the following claim, $\overset{*}{\sigma}$ will at least enable us to see that our $\in$-theory $\mathcal T$ from long ago, defined to assist us in proving the main claim, is consistent over a suitable admissible structure.
\begin{claimno} The map $\overset{*}{\sigma}$ satisfies:
\begin{enumerate}
	\item $\overset{*} {\sigma}: \N \prec N$
	\item $\overset{*} {\sigma}(\overline \theta, \overline{\D}, \overline{\U}, \overline c)=\theta, \D, \U, c$
	\item $\sk{N}{\delta}{\overset{*}{\sigma}} = C$	
	\item $\overset{*} {\sigma} `` \S \subseteq S$
\end{enumerate}
\end{claimno}
\begin{proof}[Pf.]
For item \textsl{\textbf{1}}, we have already seen above that $\mathring \sigma^{\mathfrak A}: \N \prec N_0$. Since $k_0:N_0 \prec N$, the desired result follows.

For item \textsl{\textbf{2}}, $\overset{*} {\sigma}(\overline \theta, \overline{\D}, \overline{\U}, \overline c)= k_0(\theta_0, \D_0, \U_0, c_0) =\theta, \D, \U, c$.

Item \textsl{\textbf{3}} holds since $N_0 = \sk{N_0}{\delta}{\mathring \sigma^{\mathfrak A}}$. To see this, clearly we have that $\sk{N_0}{\delta}{\mathring \sigma^{\mathfrak A}} \subseteq N_0$, since $\delta \in N_0$ as $N_0 \cong C$, and certainly $\ran(\mathring \sigma^{\mathfrak A})) \subseteq N_0$ as well. Then because $\mathring \sigma^{\mathfrak A}$ is $\overline \kappa'$-cofinal, by \textbf{Lemma \ref{lem:liftupchar}}, we have, since $\mathring \sigma^{\mathfrak A}(\overline \kappa') < \delta$, that:
$$N_0 = \set{\mathring \sigma^{\mathfrak A}(f)(\beta)}{ f: \gamma \longrightarrow N_0, \ \gamma < \overline \kappa(0) \text{ and } \beta < \mathring \sigma^{\mathfrak A}(\gamma) } \subseteq \sk{N}{\delta}{\mathring \sigma^{\mathfrak A}}.$$
Thus $C = k_0``N_0=\sk{N_0}{\delta}{k_0 \circ \mathring \sigma^{\mathfrak A}}$ as desired. 

To see \textsl{\textbf{4}}, note that $\overset{*} {\sigma} \rest \overline \kappa(0) = \mathring{\sigma}^{\mathfrak A} \rest \overline \kappa(0)$ since $k_0 \rest \nu_0 = \id$.

This completes the proof of \textit{Claim} 3.
\end{proof}

We are almost done, but like we stated above, $\mathring{\sigma}^{\mathfrak A}$ is in $V[F]$, the generic extension needed to obtain a countable admissible structure to apply Barwise Completeness to. But $V[F]$ is not in $V[S]$, and so $\overset{*} {\sigma}$ is not necessarily in $V[S]$. We will use Barwise Completeness one last time, to finally find an embedding $\sigma'$ with which to satisfy the main claim along with the $S$ we found above.

Let $\lambda$ be regular in $V[S]$ with $N \in H_\lambda^{V[S]}$. Then 
	$$M = \langle H^{V[S]}_\lambda;\ N, \sigma, S;\ \theta, \delta, \D, \U, c \rangle \text{ is admissible.}$$

In order to satisfy our main claim, we need a model of $\mathcal T(M)$ in $V[S]$. By \textit{Claim} 3, we have that $\langle M , \overset{*} {\sigma}\rangle$ is a model of $\mathcal T(M)$, but in $V[F]$, which is not $V[S]$. Still this means that by Barwise Correctness, $\mathcal T(M)$ is consistent.
 
Consider the Mostowski collapse of $M$; let 
	$$\pi: \tilde M \prec M \text{ where $\tilde M$ is countable and transitive.}$$
Note that $\tilde M \in H_{\omega_1}^{V[S]}=H_{\omega_1}^V$ since it is countable and diagonal Prikry forcing doesn't add bounded subsets to any $\kappa \in D$.\footnote{As Fuchs \cite{Fuchs:2005kx} points out, this result is a modification to the proof that generalized diagonal Prikry forcing preserves cardinalities.} We also have that $\overline{\underline N}^{\tilde{\mathfrak A}} = \N$, since $M$ sees that $\N$ is countable so $\tilde{M}$ sees that $\pi^{-1}(\N)$ is, and it follows that $\pi^{-1}(\N)=\N$.

Plus, $\mathcal T(\tilde M)$ is consistent, since otherwise its inconsistency could be pushed up via $\pi$ to one in $\mathcal T(M)$, contradicting the model witnessing its consistency that we found in $V[F]$. 

So by Barwise Completeness, $\mathcal T(\tilde M)$ has a solid model $$\tilde{\mathfrak A} = \langle \tilde{\mathfrak A}; \dot{\sigma}^{\tilde{\mathfrak A}} \rangle$$ such that $$\text{Ord} \cap \text{wfc}(\tilde{\mathfrak A}) = \Ord \cap \tilde M.$$ 

Letting $\sigma'=\pi \circ \dot{\sigma}^{\tilde{\mathfrak A}}$, the main claim is now satisfied with $\sigma'$ and our $\lambda$-diagonal Prikry sequence $S$. 

Let us verify each of the properties of $\sigma'$ required by the main claim. The verification of these properties shall use the agreement between $\tilde{\mathfrak A}$ and $\tilde M$ on the special constants of $\tilde M$ and on the ordinals. The fact that $\pi$ does not affect $\N$ will be greatly taken advantage of.

First we show \textsl{\textbf{1}} of the main claim. Let's say that $\varphi[\sigma'(\overline a)]$ holds in $N$. So $\varphi[\pi(\dot \sigma^{\tilde{\mathfrak A}}(\overline a))]^N$ holds in $M$. Thus $\varphi[\dot \sigma^{\tilde{\mathfrak A}}(\overline a)]^{\pi^{-1}(N)}$ holds in $\tilde M$, and thus also in $\tilde{\mathfrak A}$. Indeed we know that $\underline{\N}^{\tilde{\mathfrak A}} = \N$, since $\tilde{\mathfrak A}$ agrees with $\tilde M$ about countable ordinals, which we may use to code $\N$. This means that $\varphi[\overline a]$ holds in $\N$, as desired.

To see \textsl{\textbf{2}}, we have $\sigma'(\overline \theta, \overline{\D}, \overline{\U}, \overline c)= \pi(\underline{\theta}^{\tilde{\mathfrak A}}, \underline{\D}^{\tilde{\mathfrak A}}, \underline{\U}^{\tilde{\mathfrak A}}, \underline{c}^{\tilde{\mathfrak A}})=\theta, \D, \U, c$.

For item \textsl{\textbf{3}}, let $\tilde N = \underline N^{\tilde{\mathfrak A}}$, $\tilde \sigma = \underline \sigma^{\tilde{\mathfrak A}}$ and $\tilde \delta = \underline{\delta}^{\tilde{\mathfrak A}}$. So $\pi(\tilde \delta)= \delta$ and $\pi(\tilde \sigma)=\sigma$. Because of the way the $\in$-theory $\mathcal T$ was defined, we already have: 
\begin{equation}\label{eqn:SkDotSigma=SkTildeSigma} \sk{\tilde N}{\tilde \delta}{\dot \sigma^{\tilde{\mathfrak A}}} = \sk{\tilde N}{\tilde \delta}{\tilde \sigma}. \end{equation}

To see $\sk{N}{\delta}{\sigma'} \subseteq \Sk{N}{\delta}{X}$, suppose $x \in \sk{N}{\delta}{\sigma'}$. Then we have that $N$ sees that there is some formula $\varphi$, $\overline z \in \N$, and $\xi < \delta$ where $x$ is unique such that $\varphi(x, \pi(\dot \sigma^{\tilde{\mathfrak A}}(\overline z)), \xi)$. In particular, $x$ is in the range of $\pi$. Thus $\tilde x = \pi^{-1}(x) \in \tilde N$ and $\tilde \xi < \tilde \delta$ such that $\varphi(\tilde x, \dot \sigma^{\tilde{\mathfrak A}}(\overline z), \tilde \xi)$ holds. Thus by (\ref{eqn:SkDotSigma=SkTildeSigma}), we have that $\tilde x$ is unique such that 
$\varphi(\tilde x, \tilde \sigma(\overline y), \tilde \zeta)$ for some $\overline y \in \N$ and $\tilde \zeta < \tilde \delta$. Thus pushing back up through $\pi$, letting $\zeta = \pi(\tilde \zeta)$, we have that $x = \pi(\tilde x)$ is unique such that $\varphi(x, \sigma(\overline y), \zeta)$ holds, so $x \in \Sk{N}{\delta}{X}$.

To see that $\Sk{N}{\delta}{X} \subseteq \sk{N}{\delta}{\sigma'}$ works similarly. Let $x \in \Sk{N}{\delta}{X}$. Then there is $\overline z \in \N$ and $\xi < \delta$ such that $\varphi(x, \sigma(\overline z), \xi)$ holds. So in particular, $x$ is in the domain of $\pi$. So we may find $\tilde \xi < \tilde \delta$ such that $\tilde x = \pi^{-1}(x)$ is unique such that $\varphi(\tilde x, \tilde \sigma(\overline z), \tilde \xi)$. So by (\ref{eqn:SkDotSigma=SkTildeSigma}), we have that there is $\overline y \in \N$ and $\tilde \zeta < \tilde \delta$ such that $\tilde x$ is unique satisfying $\varphi(\tilde x, \dot{\sigma}^{\mathfrak A}(\overline y), \tilde \zeta)$. Finally, by pushing back up through $\pi$, letting $\pi(\tilde \zeta) = \zeta$, we have that $x = \pi(\tilde x)$ is unique such that $\varphi(x, \sigma'(\overline y), \zeta)$, so $x \in \sk{N}{\delta}{\sigma'}$ as desired.


To see item \textsl{\textbf{4}}, note that $\overline{\underline S}^{\tilde{\mathfrak A}} = \overline S$ since $\overline S \subseteq \overline N$. So we have already by the definition of $\mathcal T$  that $\dot{\sigma}^{\tilde{\mathfrak A}}``\overline S \subseteq \underline S^{\tilde{\mathfrak A}}$. Thus $\pi \circ \dot{\sigma}^{\tilde{\mathfrak A}}``\overline S \subseteq \pi `` \underline S^{\tilde{\mathfrak A}} \subseteq S$ as desired. 

This completes the proof of the main claim.
\end{proof}
We have satisfied the main claim, so we are done, we have shown that $\D$ is subcomplete.
\end{proof}

Some slight modifications to the above proof give the following two corollaries.

The first point is that the above proof also shows that generalized diagonal Prikry forcing that adds a countable sequence to each measurable cardinal is subcomplete. Before stating the corolloary let's define the forcing. Again let $D$ be an infinite discrete set of measurable cardinals. Let $\U = \seq{ U(\kappa) }{ \kappa \in D }$ be a list of measures associated to $D$. 
Let $\D^*(\U) = \D^*$ be defined the same as $\D(\U)$ except the stem of a condition, $s$, in $\D^*(\U)$ is a function with domain in $[D]^{<\omega}$ taking each measurable cardinal $\kappa \in \dom(s)$ to finitely many ordinals $s(\kappa) \subseteq \kappa$. The upper part and extension relation is defined in the same way; the only slight modification is that again we have $(s, A) \leq (t, B)$ so long as points in $s$ not in $t$ come from $B$, which in the case means that for $\kappa \in \dom(s)$ we have that each element of $s(\kappa)$ not in $t(\kappa)$ is in $B(\kappa)$. We may again form a $\D^*$-generic sequence $S = S_G$ for a generic $G \subseteq \D^*$, and we may write $S = \seq{ S(\kappa) }{ \kappa \in D }$ where $S(\kappa)$ is a countable sequence of ordinals less than $\kappa$. The genericity criterion for generic diagonal Prikry sequences is as that for $\D$, which is given in  \cite[Theorem 1]{Fuchs:2005kx}, as stated in \textbf{\emph{Fact} \ref{fact:diagprikrymathias}}, with the modification that $S$ is $\D^*$ generic if and only if for all $\mathcal X$, the set $\set{ \alpha }{ \exists \kappa \in D \ \alpha \in S(\kappa) \setminus X_\kappa }$ is finite.

\begin{cor} 
Let $D$ be an infinite discrete set of measurable cardinals. Let $\U = \seq{ U(\kappa) }{ \kappa \in D }$ be a list of measures associated to $D$. Then $\D^*(\U)$ is subcomplete.
\end{cor}
\begin{proof}[Proof Sketch.]
The modifications are mostly notational, and the main one that needs to be made is to adjust the proof of the \textit{Claim} within the proof of \textit{Claim} 2. Here we have $\D^*_1$, the generalized diagonal Prikry forcing as computed in $N_1$, as well as $\overline{\D}$ of $\N$, and $S_1$, which we would like to show is a $\D_1^*$-generic sequence over $N_1$ in this case. $S_1$ is defined as $\sigma_1 ``\overline S$, using a diagonal Prikry sequence $S_1'$ to fill in the missing coordinates, where $S_1'$ is obtained by forcing with $\D_1$ over $V$.

We will show that $S_1$ satisfies the generalized diagonal Prikry genericity criterion over $N_1$ and follow the above proof. To do this, let $\mathcal X = \seq{ X_\kappa \in U_1(\kappa) }{ \kappa \in D_1 }$, with $\mathcal X \in N_1$, be a sequence of measure-one sets in the sequence of measures $\U_1$.

Note first that $S_1'$ is a generic sequence, it already satisfies the generalized diagonal Prikry genericity criterion, namely:
$\set{ \alpha }{ \exists \kappa \in D_1 \ \alpha \in S_1'(\kappa) \setminus X_\kappa }$ is finite.
Recall that $\S = \seq{ \S(\overline \kappa) }{ \overline \kappa \in \overline D }$ is a $\overline{\D}$-generic sequence over $\N$ as well.
We need to see that in addition, $\set{ \alpha }{ \exists \overline \kappa \in \overline D \ \alpha \in \sigma_1(\S(\overline \kappa)) \setminus X_{\sigma_1(\overline \kappa)} }$ is finite.

By the $\overline{\kappa}(0)$-cofinality of $\sigma_1$, there is some $w \in \N$ such that $\mathcal X \in \sigma_1(w)$, where $|w| < \overline{\kappa}(0)$ in $\N$. Thus in $N_1$, $|\sigma_1(w)| < \kappa_1(0)$. 
For each $\kappa \in \sigma_1``\overline D$, we have that $X_\kappa \in \sigma_1(w)_\kappa = \set{\sigma_1(f)(\overline \kappa) }{ f \in \prod_{\overline \kappa \in \overline D} \overline U(\overline \kappa) \ \land \ f \in w }$ and also $|\sigma_1(w)_\kappa|<\kappa_1(0).$ All $\kappa \in \sigma_1``\overline D$ of course satisfy $\kappa \geq \kappa_1(0)$ so by the $\kappa$-completeness of $U_1(\kappa)$, we have that $W_\kappa := \cap \, \sigma_1(w)_\kappa \in U_1(\kappa)$ since $X_\kappa \in \sigma_1(w)_\kappa$.
So we have established that $\mathcal W$, the sequence of $W_\kappa$ for $\kappa \geq \kappa_1(0)$, is also a sequence of measure-one sets in $N_1$. Note in addition that for $\kappa \in \sigma_1``\overline D$, we have that $W_\kappa \subseteq X_\kappa$. 

By elementarity, for each $\overline \kappa \in \overline D$, we have $\overline W_{\overline \kappa} = \cap \set{f(\overline \kappa) }{ f \in \prod_{\overline \kappa \in \overline D} \overline U(\overline \kappa) \ \land \ f \in w }$ is a measure-one set in $\overline U(\overline \kappa)$ and we also have that $\sigma_1(\overline W_{\overline \kappa}) = W_{\sigma_1(\overline \kappa)}$. Moreover, 
$$\set{ \alpha }{ \exists \overline \kappa \in \D \ \alpha \in \overline S(\overline \kappa) \setminus \overline W_{\overline \kappa} } \text{ is finite}$$ by the generalized diagonal Prikry genericity criterion for $\overline{\D}$, which must be satisfied by $\overline S$.
Thus by elementarity,
$$\set{ \alpha }{ \exists \overline \kappa \in \D \ \sigma_1(\S(\overline \kappa)) \setminus W_{\sigma_1(\overline \kappa)} } \supseteq \set{ \alpha }{ \exists \overline \kappa \in \D \ \sigma_1(\S(\overline \kappa)) \setminus X_{\sigma_1(\overline \kappa)}} \text{ is finite,}$$
as is desired, completing the proof of the claim. 
\end{proof}

One might consider a forcing like $\D$ and $\D^*$ that adds one point below each measurable cardinal sometimes, and other times adds a cofinal $\omega$-sequence below the measurable cardinal. This forcing is clearly subcomplete as well.

Below we refer to the concept of subcompleteness above $\mu$, which was introduced in \ref{subsec:LevelsofSubcompleteness}.

\begin{cor} Let $D$ be an infinite discrete set of measurable cardinals. Let $\U = \seq{ U(\kappa) }{ \kappa \in D }$ be a list of measures associated to $D$. 

Furthermore, let $\mu < \lambda$ be a regular cardinal, where $\lambda = \sup_{n<\omega} \kappa_n$, the first limit point of $D$.
Then $\D=\D(\U)$ is subcomplete above $\mu$. \end{cor}

\begin{proof}[Proof Sketch] The idea is to follow the same exact proof as in the above theorem, except we achieve the following diagram: 

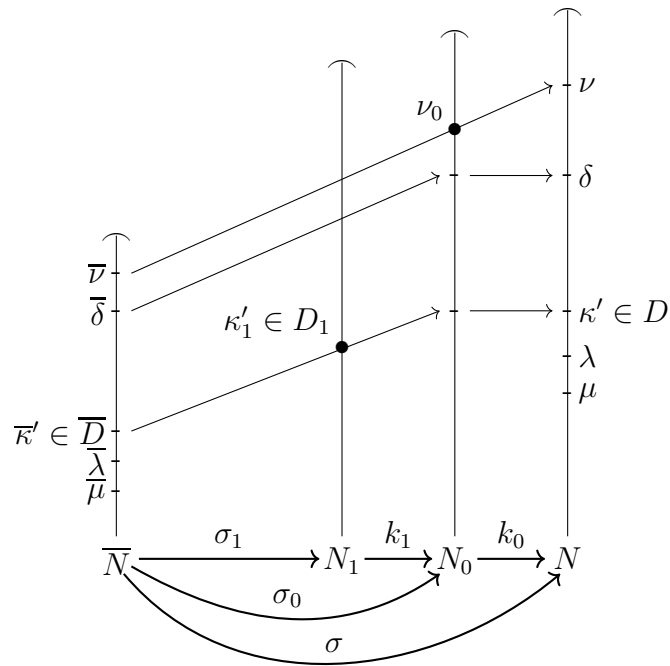
\begin{center}
\begin{tikzpicture}
\draw (-3,-1)--(-3,3);
\node [below] at (-3,-1) {$\N$};
\node at (-3,2.95) {$\frown$};

\node [left] at (-3,-0.4) {$\overline \mu$};
\node at (-3,-0.4) {-};
\node [left] at (-3,0) {$\overline \lambda$};
\node at (-3,0) {-};
\node [left] at (-3,0.4) {$\overline \kappa' \in \overline D$};
\node at (-3,0.4) {-};
\node [left] at (-3,2) {$\overline \delta$};
\node at (-3,2) {-};
\node [left] at (-3,2.5) {$\overline \nu$};
\node at (-3,2.5) {-};

\draw [->] (-2.8,0.4)--(1.3,2);

\draw (0,-1)--(0,5.3);
\node [below] at (0,-1) {$N_1$};
\node at (0,5.25) {$\frown$};
\node at (0,1.5) {$\bullet$};
\node [above left] at (0,1.5) {$\kappa_1' \in D_1$};

\draw [->] (-2.8,2)--(1.3,3.8);

\draw (1.5,-1) -- (1.5,5.7);
\node [below] at (1.5,-1) {$N_0$};
\node at (1.5,5.65) {$\frown$};
\node at (1.5,2) {-};
\node at (1.5,3.8) {-};
\node [above left] at (1.5,4.4) {$\nu_0$};
\node at (1.5,4.4) {$\bullet$};

\draw [->] (1.7,2)--(2.8,2);
\draw [->] (1.7,3.8)--(2.8,3.8);

\draw [->] (-2.8,2.5)--(2.8,5);

\node [right] at (3,.9) {$\mu$};
\node at (3,.9) {-};
\node [right] at (3,1.4) {$\lambda$};
\node at (3,1.4) {-};
\node [right] at (3,2) {$\kappa' \in D$};
\node at (3,2) {-};
\node [right] at (3,3.8) {$\delta$};
\node at (3,3.8) {-};
\node [right] at (3,5) {$\nu$};
\node at (3,5) {-};

\draw (3,-1)--(3,6);
\node [below] at (3,-1) {$N$};
\node at (3,5.95) {$\frown$};

\draw [->, thick] (-2.7,-1.3)--(-0.3,-1.3);
\draw [->, thick] (0.3, -1.3)--(1.2,-1.3);
\draw [->, thick] (1.8, -1.3)--(2.7,-1.3);
\draw[->, thick] (-2.8,-1.4) to  [out=-30, in=-145] node[above]{$\sigma_0$} (1.3,-1.5);
\draw[->,thick] (-2.9,-1.5) to [out = -50, in =-140] node[above]{$\sigma$} (2.9, -1.5);

\node [above] at (-1.5,-1.3) {$\sigma_1$};
\node [above] at (0.75,-1.3) {$k_1$};
\node [above] at (2.25, -1.3) {$k_0$};
\end{tikzpicture}
\end{center} 

Here we replace $\kappa(0)$ with some $\kappa' \in D$ such that $\lambda < \kappa' $, where there are finitely many measurables of $D$ below $\kappa'$. So in particular, we let $\langle N_1, \sigma_1 \rangle$ be the liftup of $\langle \N, \sigma \rest H_{\overline{\kappa'}}^\N\rangle$ in \textit{Claim} 2. In order to show the \textit{Claim} that we have a generic sequence over $\D_1$, we follow the same argument as follows:

Let $\mathcal X = \seq{ X_\kappa \in U_1(\kappa) }{ \kappa \in \sigma_1``\overline D }$, with $\mathcal X \in N_1$, be a sequence of measure one sets in the sequence of measures $\U_1$ with only coordinates coming from $\sigma_1``\overline D$.
We need to see that 
$$\set{ \overline \kappa \in \overline D}{\sigma_1(\S(\overline \kappa)) \notin X_{\sigma_1(\overline \kappa)} }  \text{ is finite.}$$
By the $\overline{\kappa}'$-cofinality of $\sigma_1$, there is some $w \in \N$ such that $\mathcal X \in \sigma_1(w)$, where $|w| < \overline{\kappa}'$ in $\N$. Thus in $N_1$, $|\sigma_1(w)| < \kappa_1'$. For each $\kappa \in \sigma_1``\overline D$, we have that $X_\kappa \in \sigma_1(w)_\kappa = \set{\sigma_1(f)(\overline \kappa) }{ f \in \prod_{\overline \kappa \in \overline D} \overline U(\overline \kappa) \ \land \ f \in w }$ and also $$|\sigma_1(w)|<\kappa_1'.$$ So for all but finitely many $\kappa \in \sigma_1``\overline D$, namely for $\kappa \geq \kappa_1'$, by the $\kappa$-completeness of $U_1(\kappa)$, we have that $$\cap \, \sigma_1(w)=W_\kappa \in U_1(\kappa).$$

So we have established that $\mathcal W$, the sequence of $W_\kappa$ for $\kappa>\kappa_1'$, is also a sequence of measure-one sets in $N_1$. Note in addition that for $\kappa \in \sigma_1``\overline D$, $\kappa> \kappa_1'$, we have that $W_\kappa \subseteq X_\kappa$. 

By elementarity, for each $\overline \kappa \in \overline D$, where $\overline \kappa > \overline \kappa'$, we have $$\overline W_{\overline \kappa} =\cap \set{f(\overline \kappa) }{ f \in \textstyle\prod_{\overline \kappa \in \overline D} \overline U(\overline \kappa) \ \land \ f \in w }$$ is a measure-one set in $\overline U(\overline \kappa)$ and we also have that $\sigma_1(\overline W_{\overline \kappa}) = W_{\sigma_1(\overline \kappa)}$. Moreover, 
	$$\set{\overline \kappa \in \overline D}{\overline S(\overline \kappa) \notin \overline W_{\overline \kappa}} \text{ is finite}$$ by the generalized diagonal Prikry genericity criterion for $\overline{\D}$, which must be satisfied by $\overline S$, and since there are only finitely many measurables in $\overline D$ less than $\overline \kappa'$ in $\N$.
Thus by elementarity,
	$$\set{ \overline \kappa \in \overline D }{ \sigma_1(\overline S(\overline \kappa)) \notin W_{\sigma_1(\overline \kappa)} } \supseteq \set{ \overline \kappa \in \overline D }{ \sigma_1(\overline S(\overline \kappa) \notin X_{\sigma_1(\overline \kappa)}} \text{ is finite.}$$

Additionally the $\in$-theories $\mathcal L$ and $\mathcal T$ would have to be defined so as to include as an axiom that $\mathring \sigma \rest \overline{\underline{\mu}} = \underline{\sigma} \rest \overline{\underline \mu}$ and $\dot \sigma \rest \overline{\underline{\mu}} = \underline{\sigma} \rest \overline{\underline \mu}$ respectively. 
We would then need to show that $\overset{*}\sigma \rest \overline{\mu} = \sigma \rest \overline{\mu}$, but this would follow since $k_0$ is the identity on $\nu_0$. 
Furthermore, it would need to be shown that $\sigma' \rest \overline{\mu} = \sigma \rest \overline{\mu}$, but this would follow from the requirement that $\dot \sigma^{\tilde{\mathfrak A}} \rest \overline{\mu} = \underline{\sigma}^{\tilde{\mathfrak A}} \rest \overline{\mu}$, and since ordinals are computed properly by $\tilde{\mathfrak A}$. \end{proof}

\singlespacing
\bibliographystyle{amsalpha}
\bibliography{../Thesis/BIB}

\end{document}